\documentclass[mlq,a4paper,fleqn%
]{w-art_arXiv}
\usepackage{times,cite,w-thm}
\theoremstyle{plain}

\theoremstyle{definition}

\newtheorem{fact}[theorem]{Fact}
\usepackage[]{graphicx}

\usepackage{stmaryrd}

\usepackage{amscd}
\usepackage{amssymb,amsmath}

\usepackage{enumerate}
\usepackage{enumitem}

\usepackage{mathptmx}
\usepackage{mathrsfs}
\usepackage{color}
\usepackage{xspace}
\usepackage{bussproofs}
\usepackage{tikz}
\usepackage{bpextra}

\usepackage{multirow}

\DeclareMathAlphabet{\mathcal}{OMS}{cmsy}{m}{n}

\usepackage{array}
\newcolumntype{C}[1]{>{\centering\arraybackslash}p{#1}}
\newcolumntype{L}[1]{>{\arraybackslash}p{#1}}

\newcommand{\FO}{\ensuremath{\mathbf{FO}}\xspace}
\newcommand{\Inc}{\ensuremath{\mathsf{Inc}}\xspace}
\newcommand{\pGFP}{\ensuremath{\mathsf{posGFP}}\xspace}

\newcommand{\LL}{\ensuremath{\mathscr{L}}\xspace}

\newcommand{\cn}{\ensuremath{\mathop{\dot\sim}}\xspace}

\newcommand{\incexc}{\ensuremath{\subseteq\!\mathsf{Exc}}\xspace}
\newcommand{\incwe}{\ensuremath{\subseteq\!\mathsf{W}_\exists}\xspace}
\newcommand{\incwu}{\ensuremath{\subseteq\!\mathsf{W}_\forall}\xspace}

\newcommand{\existse}{\ensuremath{\exists\mathsf{E}}\xspace}
\newcommand{\existsi}{\ensuremath{\exists\mathsf{I}}\xspace}

\newcommand{\incc}{\ensuremath{\subseteq\!\mathsf{Cmp}}\xspace}
\newcommand{\incctr}{\ensuremath{\subseteq\!\mathsf{Ctr}}\xspace}
\newcommand{\inctr}{\ensuremath{\subseteq\!\mathsf{Trs}}\xspace}
\newcommand{\incdstr}{\ensuremath{\exists\!\subseteq\!\mathsf{Ext}}\xspace}

\newcommand{\inci}{\ensuremath{\forall\!\!\subseteq\!\mathsf{Sim}}\xspace}
\newcommand{\ince}{\ensuremath{\forall\!\!\subseteq\!\mathsf{Sim}}\xspace}

\newcommand{\eqi}{\ensuremath{=\!\mathsf{I}}\xspace}

\newcommand{\eqsub}{\ensuremath{=\!\mathsf{Sub}}\xspace}

\newcommand{\conji}{\ensuremath{\wedge\mathsf{I}}\xspace}
\newcommand{\conje}{\ensuremath{\wedge\mathsf{E}}\xspace}
\newcommand{\tensori}{\ensuremath{\vee\mathsf{I}}\xspace}
\newcommand{\tensore}{\ensuremath{\vee\mathsf{E}}\xspace}

\newcommand{\nege}{\ensuremath{\neg\mathsf{E}}\xspace}
\newcommand{\negi}{\ensuremath{\neg\mathsf{I}}\xspace}

\newcommand{\uqs}{\ensuremath{\forall\mathsf{Sub}}\xspace}
\newcommand{\uqi}{\ensuremath{\forall\mathsf{I}}\xspace}
\newcommand{\uqe}{\ensuremath{\forall\mathsf{E}}\xspace}
\newcommand{\uqextd}{\ensuremath{\forall_\vee\mathsf{Ext}}\xspace}
\newcommand{\uqextc}{\ensuremath{\forall_\wedge\mathsf{Ext}}\xspace}

\newcommand{\uqez}{\ensuremath{\forall\mathsf{E}_0}\xspace}

\newcommand{\uqexc}{\ensuremath{\forall\mathsf{Exc}}\xspace}

\allowdisplaybreaks

\begin{document}
%%    The information for the title page will be placed between
%%    \begin{document} and \maketitle. The order of most entries
%%    is determined by the class file and can not be changed by
%%    rearranging them. The maketitle command follows after the
%%    abstract.
%%
%%    Most of the following commands will be completed by the publisher.
%%
%%    The copyrightyear is defined in the .clo file as the first argument
%%    of the copyrightinfo command. If the copyrightyear differs from that
%%    value it might be adjusted by the following definition:
%%
%% \renewcommand{\copyrightyear}{2007}% uncomment to change the copyrightyear.
%%
\DOIsuffix{theDOIsuffix}
%%
%% issueinfo for the header line
\Volume{53}
\Issue{0}
\Month{01}
\Year{2007}
%%
%%    First and last pagenumber of the article. If the option
%%    'autolastpage' is set (default) the second argument may be left empty.
\pagespan{1}{}
%%
%%    Dates will be filled in by the publisher. The 'reviseddate' and
%%    'dateposted' (Published online) entry may be left empty.
\Receiveddate{}
\Reviseddate{}
\Accepteddate{}
\Dateposted{}

%% \pretitle{Editor's Choice}

%% We have a short and a long form for the title. The short form
%% (optional argument) goes into the running head.

\title[Axiomatizing first-order consequences in inclusion logic]{Axiomatizing first-order consequences in inclusion logic}

%% Please do not enter footnotes or \inst{}-notes into the optional
%% argument of the author command. The optional argument will go into
%% the header.  If there is only one address the marker \inst{x} may be
%% omitted.

%% Information for the first author.
\author[F. Yang]{Fan Yang%\inst{1,}%
  \footnote{This research was supported by grant 308712 of the Academy of Finland, and also by Research Funds of the University of Helsinki.
}}
\address{PL 68 (Pietari Kalmin katu 5), 00014 University of Helsinki, Finland}
%%
%%    Information for the second author
%\author[S. Author]{Second Author\inst{1,2,}\footnote{Second author footnote.}}
%\address[\inst{2}]{Second address}
%%
%%    Information for the third author
%\author[T. Author]{Third Author\inst{2,}\footnote{Third author footnote.}}
%%
%%    \dedicatory{This is a dedicatory.}
\begin{abstract}
 Inclusion logic is a variant of dependence logic that was shown to have the same expressive power as positive greatest fixed-point logic. Inclusion logic is not axiomatizable in full, but its first-order consequences can be axiomatized. In this paper, we provide such an explicit partial axiomatization by introducing a system of natural deduction for inclusion logic that is sound and complete for first-order consequences in inclusion logic. 
\end{abstract}
%% maketitle must follow the abstract.
\maketitle                   % Produces the title.

%% If there is not enough space inside the running head
%% for all authors including the title you may provide
%% the leftmark in one of the following three forms:

%% \renewcommand{\leftmark}
%% {First Author: A Short Title}

%% \renewcommand{\leftmark}
%% {First Author and Second Author: A Short Title}

%% \renewcommand{\leftmark}
%% {First Author et al.: A Short Title}

%% \tableofcontents  % Produces the table of contents.

%%%%%%%%%%%%%%%%%%%%%%%%%%%%%%

\section{Introduction}

In this paper, we axiomatize first-order consequences of inclusion logic. {\em Inclusion logic} was introduced by Galliani \cite{PietroIE}. Together with {\em independence logic}, introduced by Gr\"{a}del and V\"{a}\"{a}n\"{a}nen \cite{D_Ind_GV}, inclusion logic is an important variant of {\em dependence logic}, which was  introduced by V\"{a}\"{a}n\"{a}nen \cite{Van07dl} as  an extension of first-order logic and a new framework for characterizing dependency notions.
% as a variant of {\em dependence logic}, which was introduced by V\"{a}\"{a}n\"{a}nen \cite{Van07dl}. Another important variant of dependence logic is {\em independence logic}, introduced by Gr\"{a}del and V\"{a}\"{a}n\"{a}nen \cite{D_Ind_GV}. 
%Dependence logic and its variants adopt the framework of  {\em team semantics} of Hodges \cite{Hodges1997a,Hodges1997b}. %to characterize dependency notions.
 %are extensions of first-order logic with new atoms characterizing dependency notions. These logics adopt the so-called {\em team semantics}, which was introduced by Hodges \cite{Hodges1997a,Hodges1997b}. 
%This line of research studies extensions of first-order logic that can characterize dependency notions. 
Inclusion logic aims to characterize inclusion dependencies by extending first-order logic with {\em inclusion atoms}, which are strings of the form 
%is the extension of first-order logic with inclusion atoms that characterize inclusion dependencies.
%An inclusion atom is a string of the form 
$x_1\dots x_n\subseteq y_1\dots y_n$, where $\langle x_1,\dots,x_n\rangle=\mathsf{x}$ and $\langle y_1,\dots,y_n\rangle=\mathsf{y}$ are sequences of variables of the same length. Inclusion logic adopts the  {\em team semantics} of Hodges \cite{Hodges1997a,Hodges1997b}, in which inclusion atoms and other formulas are evaluated in a model with respect to {\em sets} of assignments (called {\em teams}), in contrast to single assignments as in the usual first-order logic. Intuitively the inclusion atom $\mathsf{x}\subseteq\mathsf{y}$ specifies that all possible values for $\mathsf{x}$ in a team $X$ are included in the values of $\mathsf{y}$ in the same team $X$.

Galliani and Hella proved that inclusion logic is expressively equivalent to positive greatest fixed-point logic \cite{inclusion_logic_GH}. It then follows from the results of Immerman \cite{Immerman_86} and Vardi \cite{Vardi_82} that over finite ordered structures inclusion logic captures \textsf{PTIME}. Building on these results, Gr\"{a}del defined model-checking games for inclusion logic \cite{Gradel16}, which then found applications in \cite{GradelHegselmann16}.
%and applications of these games were developed in \cite{GradelHegselmann16}. 
There also emerged some studies \cite{HannulaKontinen15,Hannula_inc15,Ronnholm_thsis,HannulaHella19} on the computational complexity and syntactical fragments of inclusion logic. Embedding the semantics of inclusion atoms into the semantics of the quantifiers, R\"{o}nnholm \cite{Ronnholm18} introduced the interesting inclusion quantifiers that generalize the idea of the slashed quantifiers of {\em independence-friendly logic} \cite{Hintikka98book} (a close relative to dependence logic).
% in the same spirit of the slashed quantifiers of {\em independence-friendly logic} \cite{Hintikka98book}.
%R\"{o}nnholm \cite{Ronnholm18} has introduced the so-called inclusion quantifiers that naturally combines inclusion atoms with quantifiers, and showed that.
 Inclusion atoms have also found natural applications  in a recent formalization of Arrow's Theorem in social choice in dependence and independence logic \cite{PacuitYang2016}. Motivated by the increasing interest in inclusion logic, we present in this paper a proof-theoretic investigation of inclusion logic, which is currently missing in the literature.
%these increasing interests in inclusion logic call for a better understanding of inclusion logic from the proof-theoretic point of view, which is yet missing in the literature.

It is worth noting that inclusion atoms correspond exactly to the {\em inclusion dependencies} studied in database theory. The {\em implication problem} of inclusion dependencies, i.e., the problem of deciding whether $\Gamma\models\phi$ for a set $\Gamma\cup\{\phi\}$ of inclusion dependencies (or inclusion atoms), is completely axiomatized in \cite{CasanovaFaginPapadimitriou84} by the following three rules/axioms:
\begin{itemize}
\item $\mathsf{x}\subseteq\mathsf{x}$ (identity)
\item $x_1\dots x_n\subseteq y_1\dots y_n/x_{i_1}\dots x_{i_k}\subseteq y_{i_1}\dots y_{i_k}$ for $i_1,\dots,i_k\in\{1,\dots,n\}$ (projection and permutation)
\item $\mathsf{x}\subseteq\mathsf{y},\mathsf{y}\subseteq\mathsf{z}/\mathsf{x}\subseteq\mathsf{z}$ (transitivity)
\end{itemize}
The team semantics interpretation for inclusion atoms has recently been ulitized to study  the implication problems of inclusion atoms together with other dependency atoms \cite{HannulaKontinen14,HannulaKontinenLink17,HannulaLink18}. 
In this paper, we study, instead, the axiomatization problem of inclusion logic, i.e., inclusion atoms  enriched with connectives and quantifiers of first-order logic. We investigate the problem of finding a deduction system for which the completeness theorem
\begin{equation}\label{cmp_st_eq}
\Gamma\models\phi\iff\Gamma\vdash\phi
\end{equation}
holds for $\Gamma\cup\{\phi\}$ being a set of formulas of the logic. 

It is known that dependence logic is not (effectively) axiomatizable, since the sentences of the logic are equi-expressive with sentences of existential second-order logic (\textsf{ESO}) \cite{Van07dl}. Nevertheless, if one restrict the consequence $\phi$ in (\ref{cmp_st_eq}) to a first-order sentence and $\Gamma$ to a set of sentences in dependence logic, the axiomatization can be found. This is because, finding a model for such a set $\Gamma\cup\{\neg\phi\}$ of sentences of dependence logic is the same as finding a model for a set of \textsf{ESO} sentences (i.e., sentences of the form $\exists f_1\dots f_n\alpha$ for some first-order $\alpha$), which is then reduced to finding a model for a set of first-order sentences (of the form $\alpha$). A concrete system of natural deduction for dependence logic admitting this type of completeness theorem was given in \cite{Axiom_fo_d_KV}. The proof of the completeness theorem uses a nontrivial technique based on the equivalence between a dependence logic sentence and its so-called game expression (an infinitary first-order sentence describing a semantic game) over countable models, and the fact that the game expression can be finitely approximated over recursively saturated  models.
  Subsequently, using the similar method a system of natural deduction axiomatizing completely the first-order consequences in independence logic  with respect to sentences was also introduced \cite{Hannula_fo_ind_13}. These partial axiomatizations for sentences were first generalized in \cite{Kontinen15foc} to cover the cases for formulas by expanding the language with a new predicate symbol to interpret the teams, and later generalized further in \cite{Yang_neg18} to cover  the case when the consequence $\phi$ in (\ref{cmp_st_eq}) is not necessarily first-order itself but has an essentially first-order translation by applying a trick that involves the weak classical negation \cn and the addition of the $\mathsf{RAA}$ rule for \cn. 

As we will demonstrate formally in this paper, inclusion logic is not (effectively) axiomatizable either. Since inclusion logic is less expressive than \textsf{ESO}, by the same argument as  above, the first-order consequences of inclusion logic can also be axiomatized. In this paper, we give explicitly such an axiomatization. To be more precise, we introduce a system of natural deduction for inclusion logic for which the completeness theorem (\ref{cmp_st_eq}) holds for $\phi$ being a first-order formula and $\Gamma$ being a set of \Inc-formulas. Our completeness proof uses the technique developed in \cite{Axiom_fo_d_KV} together with the trick in \cite{Yang_neg18}. %As with the systems of dependence and independence logic defined in \cite{Axiom_fo_d_KV,Hannula_fo_ind_13}, 
Our system of inclusion logic is a conservative extension of the system of first-order logic, in the sense that it has the same rules as that of first-order logic when restricted to first-order formulas only. The rules for inclusion atoms  include some of those introduced in \cite{Hannula_fo_ind_13}, and the rules characterizing the interactions between inclusion atoms and the connectives and quantifiers appear to be simpler than the corresponding ones in the systems of dependence and independence logic defined in \cite{Axiom_fo_d_KV,Hannula_fo_ind_13}. The $\mathsf{RAA}$ rule for \cn, being a crucial (yet generally not effective) rule for applying the trick of \cite{Yang_neg18}, also behaves better in our system of inclusion logic than in the systems of dependence and independence logic. In particular, in the inclusion logic system, with respect to  first-order formulas, the $\mathsf{RAA}$ rule for \cn  becomes effective and also derivable from other more basic rules.

The paper is organized as follows. In Section 2 we recall the basics of inclusion logic, and also give a  proof that inclusion logic is not (effectively) axiomatizable. Section 3 discusses the normal form for inclusion logic. In Section 4, we define the game expressions and their finite approximations that are crucial for the proof of the completeness theorem of the  system of natural deduction for inclusion logic. We introduce this system in Section 5, and also prove the soundness theorem as well as some useful derivable clauses in the section. The  proof of the completeness theorem will be given in Section 6. We conclude in Section 7 by showing some applications of our system; in particular, we derive in our system the axioms for anonymity atoms proposed recently by V\"{a}\"{a}n\"{a}nen \cite{Vaananen_anonymity19}.

\section{Preliminaries}\label{sec:pre} % 

In this section, we recall the basics of inclusion logic and prove formally that  inclusion logic is not (effectively) axiomatizable.
We consider  first-order signatures \LL with a built-in equality symbol $=$. Fix a set $\mathsf{Var}$ of first-order variables, and denote its elements by $u,v,w,x,y,\dots$ (with or without subscripts). First-order \LL-terms $t$ are built recursively as usual. First-order \LL-formulas $\alpha$ are defined by the grammar:%(with connectives $\neg,\wedge,\vee$ and quantifiers $\exists,\forall$)  %In particular, an equality $t_1=t_2$ is a first-order formula. 
\[\alpha::=\bot\mid t_1=t_2\mid Rt_1,\dots,t_n\mid\neg\alpha\mid \alpha\wedge\alpha\mid\alpha\vee\alpha\mid\exists x\alpha\mid\forall x\alpha.\]
Throughout the paper, we reserve the first greek letters $\alpha,\beta,\gamma,\delta$ (with or without subscripts) for first-order formulas. As usual, we write $\alpha\to\beta:=\neg\alpha\vee\beta$ and $\alpha\leftrightarrow\beta:=(\alpha\to\beta)\wedge(\beta\to\alpha)$ for first-order formulas $\alpha$ and $\beta$. Formulas $\phi$ of inclusion logic (\Inc) are defined recursively as follows:
\[
\phi::=\bot\mid\alpha\mid \neg\alpha\mid   x_1\dots x_n\subseteq y_1\dots y_n\mid  \phi\wedge\phi\mid\phi\vee\phi\mid \exists x \phi\mid \forall x\phi
\]
where $\alpha$ is an arbitrary first-order formula.
The formula $x_1\cdots x_n\subseteq y_1\dots y_n$ is called an {\em inclusion atom}. Note that in the literature on inclusion logic, \Inc-formulas are usually assumed to be in negation normal form (i.e., negation occurs only in front of atomic formulas). We do not adopt this convention in this paper, but we do require that negation in \Inc applies only to first-order formulas. 
%where $\alpha,\alpha_1,\dots,\alpha_k,\beta_1,\dots,\beta_k$ are arbitrary first-order formulas, and $t_0,t_1,\dots,t_k,t_1',$\\$\dots,t_k'$ are arbitrary first-order terms. Both  $t_1,\dots,t_k\subseteq t_1',\dots,t_k'$ and $\alpha_1,\dots,\alpha_k\subseteq \beta_1,\dots,\beta_k$ are called {\em inclusion atoms}. Note that the syntax of our inclusion logic as defined above is richer than the usual one considered in the literature (e.g., \cite{inclusion_logic_GH}), especially we allow inclusion atoms to have first-order formulas as arguments. We shall see later that such defined inclusion logic has the same expressive power as the usual one. \todo{Is it? No dependence atom}

The set $\textsf{Fv}(\phi)$ of free variables of an \Inc-formula $\phi$ is defined inductively as usual except that we now have the new  case
\[\textsf{Fv}(x_1\cdots x_n\subseteq y_1\cdots y_n):=\{x_1,\dots,x_n,y_1,\dots,y_n\}.\]
%We write $\phi(x_1,\dots,x_n)$ to indicate that the free variables of $\phi$ are among $x_1,\dots,x_n$. 
We write $\phi(x_1,\dots,x_k)$ to indicate that the free variables of $\phi$ are among $ x_1,\dots,x_k$. \Inc-formulas with no free variable are called {\em sentences}. We write $\phi(t/x)$ for the formula obtained by substituting uniformly $t$ for $x$  in $\phi$, where we assume that $t$ is free for $x$.

We assume that the domain of a first-order model $M$ has at least two elements, and use the same letter $M$ to stand for both the model and its domain. 
An assignment  of an  \LL-model $M$ for a set $V \subseteq \mathsf{Var}$ of variables is a function $s:V\to M$. 
The interpretation of an \LL-term $t$ under $M$ and $s$ (denoted  $s(t^M)$) is defined as usual.
For any sequence $\mathsf{x}=\langle x_1,\dots,x_k\rangle$ of variables, we write $s(x_1,\dots,x_k)$ or $s(\mathsf{x})$ for $\langle s(x_1),\dots,s(x_k)\rangle$. For any element $a\in M$, $s(a/x)$ is the assignment defined as
\[s(a/x)(y)=\begin{cases}
a,&\text{if }y=x;\\
s(y),&\text{otherwise.}
\end{cases}
\]
A set $X$ of assignments of a model $M$ with the same domain $\mathsf{dom}(X)$ is called a {\em team} (of $M$). In particular, the empty set $\emptyset$ is a team, and the singleton $\{\emptyset\}$ is a team with the empty domain. %Given an \Inc-formula $\phi$, we say that a model $M$ is a {\em suitable} model and a team $X$ is a {\em suitable} team with respect to $\phi$ if all constant, relation and function symbols occurring in $\phi$ are in the signature of $M$, and all free variables of $\phi$ are in the domain of $X$.

\begin{defn}
 For any \LL-formula  $\phi$ of \Inc, any \LL-model $M$ and any team $X$ of $M$ with $\mathsf{dom}(X)\supseteq \mathsf{Fv}(\phi)$, we define the satisfaction relation $M\models_X\phi$ inductively as follows:  
\begin{itemize}
\item $M\models_X \bot$  \ iff \  $X=\emptyset$.
%\item $M\models_X t_1=t_2$ iff for all $s\in X$, $s(t_1^M)=s(t_2^M)$.
\item $M\models_X \alpha$    \ iff \   for all $s\in X$, $M\models_s\alpha$ in the usual sense.
\item $M\models_X \neg \alpha$    \ iff \   for all $s\in X$, $M\not\models_s\alpha$ in the usual sense.
\item $M\models_X \mathsf{x}\subseteq \mathsf{y}$   \ iff \   for all $s\in X$, there is $s'\in X$ such that $s(\mathsf{x})=s'(\mathsf{y})$.
%\[s(x_1,\dots,x_k)=s'(y_1,\dots,y_k).\] 
%\item $M\models_X x_1\dots x_k\subseteq^{\alpha(z)} y_1\dots y_k$ iff for all $s\in X$, there is $s'\in X$ such that %$s(\bar{w})=s'(\bar{u})$;
%\[M\models_s\alpha(x_1)\Leftrightarrow M\models_{s'}\alpha(y_1),\quad\dots,\quad M\models_s\alpha(x_k)\Leftrightarrow M\models_{s'}\alpha(y_k);\]
  \item $M\models_X\phi\wedge\psi$   \ iff \   $M\models_X\phi$ and $M\models_X\psi$.
   \item $M\models_X\phi\vee\psi$   \ iff \   there exist $Y,Z\subseteq X$ with $X=Y\cup Z$ such that %\index{tensor disjunction (\sor)}\index{tensor (\sor)}
  $M\models_{Y}\phi$ and $M\models_{Z}\psi$.
      \item $M\models_X\exists x\phi$   \ iff \   $M\models_{X(F/x)}\phi$ for some function $F:X\to \wp(M)\setminus\{\emptyset\}$, where 
    \[X(F/x)=\{s(a/x)\mid s\in X\text{ and }a\in F(s)\}.\]
  \item $M\models_X\forall x\phi$   \ iff \   $M\models_{X(M/x)}\phi$, where
  \(X(M/x)=\{s(a/x)\mid s\in X\text{ and }a\in M\}.\)
\end{itemize}
For any set $\Gamma$ of \Inc-formulas, we write $M\models_X\Gamma$ if $M\models_X\phi$ for all $\phi\in\Gamma$. We write $\Gamma\models\phi$ if  $M\models_X\Gamma$ implies $M\models_X\phi$ for all  models $M$ and teams $X$. We write simply $\models\phi$ for $\emptyset\models\phi$, and $\psi\models\phi$ for $\{\psi\}\models\phi$. If both $\phi\models\psi$ and $\psi\models\phi$, we wire $\phi\equiv\psi$.
\end{defn}

Our version of the team semantics for disjunction  and existential quantifier is known in the literature as {\em lax semantics}; see \cite{PietroIE} for further discussion.
In some literature  (e.g., \cite{PietroIE})  inclusion atoms are allowed to have arbitrary terms %or first-order formulas 
as arguments, namely  strings of the form $t_1\dots t_n\subseteq t_1'\dots t_n'$ 
%or $\alpha_1\dots\alpha_k\subseteq \beta_1\dots\beta_k$ 
are considered well-formed formulas, and the semantics of these  inclusion atoms are defined (naturally) as:
\begin{itemize}
\item $M\models_X t_1\cdots t_n\subseteq t_1'\cdots t_n'$ ~iff~ for all $s\in X$, there is $s'\in X$ such that %$s(\bar{w})=s'(\bar{u})$;
\(s(t_1^M,\dots,t_n^M)=s'((t_1')^M,\dots,(t_n')^M).\)
%\item $M\models_X \alpha_1\dots \alpha_k\subseteq \beta_1\dots \beta_k$ iff for all $s\in X$, there is $s'\in X$ such that %$s(\bar{w})=s'(\bar{u})$;
%\[M\models_s\alpha_1\Leftrightarrow M\models_{s'}\beta_1,\quad\dots,\quad M\models_s\alpha_k\Leftrightarrow M\models_{s'}\beta_k.\]
\end{itemize}
It is easy to verify that inclusion atoms of this type are definable in our version of inclusion logic, since $\mathsf{t}\subseteq\mathsf{t}'\equiv\exists \mathsf{x}\mathsf{y}(\mathsf{x}=\mathsf{t}\wedge \mathsf{y}=\mathsf{t}'\wedge \mathsf{x}\subseteq\mathsf{y})$, where $\exists \mathsf{v}$ abbreviates $\exists v_1\dots\exists v_k$ for some $k$, and $\mathsf{u}=\mathsf{v}$ is short for $\bigwedge_{i}u_i=v_i$.
%\begin{itemize}
%\item
% \[\displaystyle t_1\cdots t_k\subseteq t_1'\cdots t_k'\equiv \exists x_1\dots x_k\exists y_1\dots y_k(x_1\dots x_k\subseteq y_1\cdots y_k\wedge \bigwedge_{i=1}^k (x_i=t_i\wedge y_i=t_i')).\]
%\item $\displaystyle\alpha_1\cdots \alpha_k\subseteq \beta_1\cdots \beta_k\equiv \exists x_1\dots x_k\exists y_1\dots y_k\Big(x_1\dots x_k\subseteq y_1\cdots y_k\wedge$
%
%~~~~~~~~$\displaystyle\bigwedge_{i=1}^k\big( (x_i=\mathbf{1}\vee x_i=\mathbf{0})\wedge (\alpha_i\leftrightarrow x_i=\mathbf{1})\wedge(y_i=\mathbf{1}\vee y_i=\mathbf{0})\wedge (\beta_i\leftrightarrow y_i=\mathbf{1})\big)\Big)$.\footnotemark
%\end{itemize}
%\footnotetext{Here we are assuming that the signature contains two constant symbols $\mathbf{0}$ and $\mathbf{1}$ that are interpreted as different elements in the model.}

For any assignment $s$ and any set $V\subseteq\mathsf{Var}$ of variables, we write $s\upharpoonright V$ for the assignment $s$ restricted to $V$. For any team $X$, define $X\upharpoonright V=\{s\upharpoonright V\mid s\in X\}$. We list the most important properties of \Inc-formulas in the following lemma. The reader is referred to \cite{PietroIE,inclusion_logic_GH} for other properties.

\begin{lemma}
Let $\phi$ be an \LL-formula, $M$ an \LL-model, and $X$, $Y$, $X_i$ ($i\in I$) arbitrary teams of $M$ with $\mathsf{dom}(X),\mathsf{dom}(Y),\mathsf{dom}(X_i)\supseteq \mathsf{Fv}(\phi)$.
\begin{description}
\item[Locality] If $X\upharpoonright \mathsf{Fv}(\phi)=Y\upharpoonright \mathsf{Fv}(\phi)$, then $M\models_X\phi\iff M\models_Y\phi$. %In particular, if $\theta$ is a sentence, then $M\models_{\{\emptyset\}}\theta\iff M\models_X\theta$ for all teams $X$. We  write $M\models\theta$ in case $M\models_{\{\emptyset\}}\theta$.
\item[Union Closure] If $M\models_{X_i}\phi$ for all $i\in I$, then $M\models_{\bigcup_{i\in I}X_i}\phi$.
\item[Flatness of First-order Formulas] For any first-order \LL-formula $\alpha$, 
\[M\models_X\alpha\iff M\models_{\{s\}}\alpha\text{ for all }s\in X.\]
Consequently, first-order formulas are also {\em downwards closed}, that is,
$M\models_X\alpha$ and $Y\subseteq X$ imply $M\models_Y\alpha$. 
\end{description}
\end{lemma}

%\begin{lemma}
%First-order formulas are flat.
%\end{lemma}
If $\theta$ is a sentence, the locality property implies that $M\models_{\{\emptyset\}}\theta$ iff $M\models_X\theta$ for all teams $X$ of $M$. 
We call $M$ a {\em model} of $\theta$, written $M\models\theta$, if $M\models_{\{\emptyset\}}\theta$.

By the result of \cite{PietroIE}, \Inc sentences can be {\em translated into} existential second-order logic (\textsf{ESO}), namely, for every \Inc-sentence $\theta$, there exists a \textsf{ESO}-sentence $\tau(\theta)$ such that $M\models\theta$ iff $M\models\tau(\theta)$. Since \textsf{ESO} is well-known to be compact, it follows that \Inc is {\em compact} as well, that is, if every finite subset of a set $\Gamma$ of \Inc-sentences has a model, then the set $\Gamma$ itself has a model.
It was further proved in \cite{inclusion_logic_GH} that \Inc is expressively equivalent to positive greatest fixed point logic (\pGFP) in the sense of the following theorem.

%As a consequence of the flatness of first-order formulas, we have that for any set $\Gamma\cup\{\alpha\}$ of first-order formulas, 
%\begin{align}\label{flat_sem_equi_fo}
%\Gamma\models\alpha &\iff \text{ for }M\text{ and }s:~M\models_s\gamma\text{ for all }\gamma\in \Gamma \text{ implies }M\models_s\alpha
%\end{align}

%\Inc has the same expressive power as

\begin{theorem}[\cite{inclusion_logic_GH}]
For any \LL-formula $\phi$ of \Inc with $\mathsf{Fv}(\phi)=\{x_1,\dots,x_n\}$, there exists an  $\LL(R)$-formula $\psi(R)$ of \pGFP with a fresh $n$-ary relation symbol $R$ such that for all \LL-models $M$ and  teams $X$ of $M$ with $\mathsf{dom}(X)=\{x_1,\dots,x_n\}$,
\[M \models_X \phi\iff (M,rel(X)) \models_s \psi(R)\text{ for all }s \in X;\]
and vice versa, where $rel(X)=\{(s(x_1),\dots,s(x_n))\mid s\in X\}$ is an $n$-ary relation on $M$ that serves as the interpretation for $R$.
In particular,  \Inc-sentences can be translated into \pGFP and vice versa.

As a consequence of \cite{Immerman_86}, over finite models, \Inc and least fixed point logic have the same expressive power. In particular, by \cite{Immerman_86,Vardi_82}, over ordered finite models, \Inc captures $\mathsf{PTIME}$.
\end{theorem}

%Introduce natural deduction system of first-order logic with equality. Introduce the symbol $\dashv\vdash$.
%
%An example of the rules of equality: \todo{or move this example to somewhere after the full system is introduced}
%
%\begin{example}\
%\begin{itemize}
%\item $\displaystyle t_1\cdots t_k\subseteq t_1'\cdots t_k'\dashv\vdash \exists x_1y_1\cdots x_ky_k\Big(x_1\cdots x_k\subseteq y_1\cdots y_k\wedge \bigwedge_{i=1}^k(x_i=t_i\,\wedge\, y_i=t_i')\Big)$
%
%\item $\displaystyle\begin{array}[t]{rl}
%&\alpha_1\cdots \alpha_k\subseteq\beta_1\cdots \beta_k\dashv\vdash\exists x_1y_1\cdots x_ky_k\exists^1 z_0\exists^1 z_1\Big(x_1\cdots x_k\subseteq y_1\cdots y_k \wedge (z_0\neq z_1)\\
%&\displaystyle\wedge\bigwedge_{i=1}^k \big((\alpha_i\rightarrow x_i=z_1)\wedge (\neg\alpha_i\rightarrow x_i=z_0)\big)\wedge(\beta\rightarrow y=z_1)\wedge (\neg\beta\rightarrow y=z_0)\Big)\end{array}$
%
%\end{itemize}
%\end{example}

%\section{\Inc is not (effectively) axiomatizable in full}

Due to the strong expressive power, \Inc is not (effectively) axiomatizable. We now give an explicit proof of this fact by following a similar argument to that in \cite{Axiom_fo_d_KV}.\footnote{The author would like to thank Jouko V\"{a}\"{a}n\"{a}nen for suggesting this proof, and the formula used in Proposition \ref{well_founded} is taken essentially  from \cite{inclusion_logic_GH}.} 

Consider the signature $\mathscr{L}_a=(+,\times,<,0,1)$ of arithmetic. We first show that the non-well-foundedness of  $<$ is definable in \Inc.

\begin{proposition}\label{well_founded}
For any model $M$ in the signature $\mathscr{L}_a$  of arithmetic,
$M\models\exists x\exists y(y\subseteq x\wedge y<x)$ iff $<^M$ is not well-founded.
\end{proposition}
\begin{proof}
It is easy to prove that $\exists x\exists y(y\subseteq x\wedge y<x)$ holds in $M$ iff $M$ contains an infinite $<$-descending chain  $\dots<^Ma_n<^M\dots<^Ma_1<^Ma_0$. We leave the proof details to the reader.
%Suppose $\dots<a_n^M<\dots<a_1<^Ma_0$ is an infinite descending chain in $M$. Define functions $F:\{\emptyset\}\to \wp(M)\setminus\{\emptyset\}$ and $G:\{\emptyset\}(F/x)\to \wp(M)\setminus\{\emptyset\}$ as follows:
%\[F(\emptyset)=\{a_n\mid n<\omega\},\]
%\[G(\langle x,a_n\rangle)=\{a_{n+1}\}.\]
%We show that $M\models_{\{\emptyset\}(F/x)(G/y)}y\subseteq x\wedge y<x$. For any $s\in \{\emptyset\}(F/x)(G/y)$, we have $s(x)=a_n$ and $s(y)=a_{n+1}$ for some $n<\omega$. Thus $s(y)=a_{n+1}<^Ma_n=s(x)$, and there is $s'\in \{\emptyset\}(F/x)(G/y)$ such that $s'(x)=a_{n+1}=s(y)$, as required.
%
%Conversely, suppose $(M,<^M)\models\exists x\exists y(y\subseteq x\wedge y<x)$. Let $F,G$ be suitable functions such that $M\models_{\{\emptyset\}(F/x)(G/y)}y\subseteq x\wedge y<x$. Pick an $s_0\in \{\emptyset\}(F/x)(G/y)$. Since $M\models_{\{\emptyset\}(F/x)(G/y)}y\subseteq x$, there exists $s_1\in \{\emptyset\}(F/x)(G/y)$ such that $s_1(x)=s_0(y)$. Since $M\models_{\{s_1\}}y< x$, we also have 
%\[s_1(y)<s_1(x)=s_0(y).\]
%Repeat the argument infinitely we can find a set $\{s_n\mid n<\omega\}\subseteq \{\emptyset\}(F/x)(G/y)$ with the infinite descending chain
%\[\dots<s_n(y)<\dots<s_1(y)<s_0(y),\]
%as required.
\end{proof}

Now, put $\phi=\exists x\exists y(y\subseteq x\wedge y<x)$, and let $\gamma_\mathbf{PA}$ be a (first-order) sentence stating that each of the (finitely many) axioms of Peano arithmetic except for the axiom schema of induction  is true (or $\gamma_\mathbf{PA}$ is the conjunction of all axioms of Robinson arithmetic $\mathsf{Q}$). For any $\mathscr{L}_a$-sentence $\alpha$ of arithmetic, we have that
\begin{equation}\label{pa2inc}
\mathbb{N}\models\alpha~\text{ iff }~\models\alpha\vee\neg\gamma_\mathbf{PA}\vee\phi,
\end{equation}
where $\mathbb{N}$ is the standard model of Peano arithmetic. To see why, for the left to right direction, suppose that $\mathbb{N}\models\alpha$ and that $M$ is a model of $\gamma_\mathbf{PA}$ such that $M\not\models\phi$. By Proposition \ref{well_founded}, $<^M$ is well-founded. Now, $M$ is a model satisfying all axioms of Robinson arithmetic (including the axiom $\forall x(x=0\vee\exists y(y+1=x))$ and the axioms stating that $<$ is a linear ordering), and the ordering $<^M$ is a well-ordering. It is then easy to verify that $M$ also satisfies the (second-order) induction axiom.
%Peano arithmetic except for the axiom schema of induction and the ordering  $<^M$ is well-founded. 
Therefore $M$ is (isomorphic to) the standard model $\mathbb{N}$ of arithmetic, and $M\models\alpha$. Conversely, suppose $\models\alpha\vee\neg\gamma_{\mathbf{PA}}\vee\phi$. The standard model $\mathbb{N}$ of Peano arithmetic clearly satisfies  $\gamma_{\mathbf{PA}}$, and by Proposition \ref{well_founded} the model $\mathbb{N}$ falsifies $\phi$. Thus we must have that  $\mathbb{N}\models\alpha$.

The equivalence (\ref{pa2inc}) shows that truth in the standard model $\mathbb{N}$ can be reduced to logical validity in inclusion logic. This means that validity in inclusion logic is not arithmetical, and therefore inclusion logic cannot have any (effective) complete axiomatization. 

Nevertheless, there can be partial axiomatizations for the logic. The main objective of the present paper is to introduce a  system of natural deduction for \Inc that is complete for first-order consequences, in the sense that
\[\Gamma\vdash\alpha\iff\Gamma\models\alpha\]
holds whenever $\Gamma$ is a set of \Inc-formulas, and $\alpha$ is a first-order formula. Our completeness proof will mainly follow the argument of \cite{Axiom_fo_d_KV}, which roughly goes as follows: First, we show that any \Inc-sentence  is semantically equivalent to a formula $\phi$ in certain normal form. Also, in the system to be introduced every \Inc-formula implies its normal form. 
Then, we show that $\phi$ is equivalent over countable models to a first-order sentence $\Phi$  of infinite length (called its {\em game expression}). Next, we show that the game expression $\Phi$ can be approximated in a certain sense (in the sense of Theorem \ref{satu_md_approx_equi}) by some first-order sentences $\Phi^n$ ($n\in \omega$) of finite length. Finally, making essential use of these approximations $\Phi^n$ we will be able to prove the completeness theorem by certain model theoretic argument, together with a trick developed in \cite{Yang_neg18} using the weak classical negation $\cn$.

%\todo{recall the argument for the completeness theorem}

\section{Normal form}

In this section, we prove that every \Inc-formula $\phi(\mathsf{z})$ is (semantically) equivalent to a formula of the form
\(\exists\mathsf{x}\forall y(\iota(\mathsf{x},y)\wedge\alpha(\mathsf{x},\mathsf{z})),\)
where $\iota$ is a conjunction of inclusion atoms, and $\alpha$ is a first-order quantifier-free formula. This normal form is similar to the normal forms for dependence and independence logic as introduced in \cite{Van07dl,Hannula_fo_ind_13}. It is also more refined than the two  normal forms for \Inc-formulas introduced in the literature, which we recall in the following.

\begin{theorem}[\cite{Hierarchies_Ind_GHK}]\label{nf_thm_qf}
Every \Inc-formula $\phi(\mathsf{z})$  is semantically equivalent to a formula of the form
\begin{equation}\label{nf_qf}
Q^1x_1\dots Q^nx_n\theta(\mathsf{x},\mathsf{z}),
\end{equation}
where $Q^i\in \{\exists, \forall\}$ and $\theta$ is a quantifier free formula. %\todo{check negation throughout the paper!}\todo{Also empty set!}
\end{theorem}
\begin{proof}[Proof (sketch)]
%We may assume first that all first-order subformulas of $\phi$ are in negation normal form. 
The theorem  follows from the fact that %for any first-order formulas $\alpha,\beta$,
\begin{itemize}
\item $\neg\forall x\alpha\equiv\exists x\neg\alpha$ and $\neg\exists x\alpha\equiv \forall x\neg\alpha$ for any first-order formula $\alpha$,%\hfill\refstepcounter{equation}(\theequation\label{nf_thm_neg_eq})
%\item $\neg(\alpha\vee\beta)\equiv\neg\alpha\wedge\neg\beta$ and $\neg(\alpha\wedge\beta)\equiv\neg\alpha\vee\neg\beta$,
\end{itemize}
and the fact that if $x\notin \textsf{Fv}(\psi)$, then 
%\[\exists x\phi\wedge\psi\equiv\exists x(\phi\wedge\psi),\quad\exists x\phi\vee\psi\equiv\exists x(\phi\vee\psi),\]
%\[\forall x\phi\wedge\psi\equiv\forall x(\phi\wedge\psi),\quad\forall x\phi\vee\psi\equiv\exists y\exists z\forall x\big((\phi\wedge y=z)\vee(\psi\wedge y\neq z)\big),\]
\begin{itemize}
\item $\exists x\phi\wedge\psi\equiv\exists x(\phi\wedge\psi)$,
\item $\exists x\phi\vee\psi\equiv\exists x(\phi\vee\psi)$,\hfill\refstepcounter{equation}(\theequation\label{nf_thm_qf_eq1})
\item $\forall x\phi\wedge\psi\equiv\forall x(\phi\wedge\psi)$,
\item $\forall x\phi\vee\psi\equiv\exists y\exists z\forall x\big((\phi\wedge y=z)\vee(\psi\wedge y\neq z)\big)$, where $y,z$ are  fresh variables.\qedhere
\end{itemize}
\end{proof}

\begin{theorem}[\cite{Hannula_inc15}]\label{nf_thm_uni}
Every \Inc-formula $\phi(\mathsf{z})$ of the form (\ref{nf_qf}) is semantically equivalent to a formula of the form
\begin{equation}\label{nf_uni}
\exists \mathsf{x}\forall y\big(\mathop{\bigwedge_{1\leq j\leq n}}_{Q^j=\forall}\mathsf{z}x_1\dots x_{j-1}y\subseteq \mathsf{z}x_1\dots x_{j-1}x_j\wedge\theta(\mathsf{x},\mathsf{z})\big),
\end{equation}
where $\mathsf{x}=\langle x_1,\dots, x_n\rangle$, $y$ is  fresh and $\theta$ is the quantifier free formula in (\ref{nf_qf}).
\end{theorem}
\begin{proof}[Proof (idea)]
This theorem is proved by exhaustively  applying the equivalences 
\begin{equation}\label{nf_thm_uni_eq1}
\forall vQ\mathsf{x}\psi(v,\mathsf{x},\mathsf{z})\equiv \exists vQ\mathsf{x}\forall y(\mathsf{z}y\subseteq\mathsf{z}v \wedge \psi(v,\mathsf{x},\mathsf{z})),
\end{equation} 
and 
\(\forall y_1\forall y_2(\mathsf{z}_1y_1\subseteq \mathsf{z}_1v_1\wedge \mathsf{z}_2y_2\subseteq \mathsf{z}_2v_2\wedge\chi(\mathsf{x},\mathsf{z}_1,\mathsf{z}_2,v_1,v_2))\equiv \forall y(\mathsf{z}_1y\subseteq \mathsf{z}_1v_1\wedge \mathsf{z}_2y\subseteq \mathsf{z}_2v_2\wedge\chi(\mathsf{x},\mathsf{z}_1,\mathsf{z}_2,v_1,v_2)).\)
%\begin{equation}\label{nf_thm_uni_eq2}
%\forall vQ\mathsf{x}\forall y\psi(v,\mathsf{x},y,\mathsf{z})\equiv \exists vQ\mathsf{x}\forall y(\mathsf{z}y\subseteq\mathsf{z}v \wedge \psi(v,\mathsf{x},y,\mathsf{z})).
%\end{equation} 
\end{proof}

%The next lemma will be applied in the main argument of this section. We leave its proof to the reader, and cf. Proposition \ref{der_rules}(i).
%\begin{lemma}\label{ext_distr}
%If $x\notin \textsf{Fv}(\psi)$, then $\exists x\phi\wedge\psi\equiv\exists x(\phi\wedge\psi)$.
%\end{lemma}
%\begin{proof}
%Left to the reader. Also cf. Proposition \ref{der_rules}(i).
%\end{proof}

We show next that the quantifier-free formula $\theta$ in the above two theorems can also be turned into an equivalent formula in some normal form.

\begin{lemma}\label{qf_nf_thm}
Every quantifier-free \Inc-formula $\theta(\mathsf{z})$  is semantically equivalent to a formula of the form
\begin{equation}\label{qf_nf}
\exists \mathsf{w}\Big(\bigwedge_{i\in I}\mathsf{u}_i\subseteq \mathsf{v}_i\wedge\alpha(\mathsf{w},\mathsf{z})\Big),
\end{equation}
where $\alpha$ is a first-order quantifier-free formula, and each $\mathsf{u}_i$ and $\mathsf{v}_i$ are sequences of variables from $\mathsf{w}$.
\end{lemma}
\begin{proof}
We prove the lemma by induction on $\theta$. The case when $\theta$ is a first-order formula (including the case $\theta=\neg\alpha$) is trivial. If $\theta=\mathsf{x}\subseteq \mathsf{y}$, clearly
\(\mathsf{x}\subseteq \mathsf{y}\equiv\exists \mathsf{w}\mathsf{u}(\mathsf{w}\subseteq \mathsf{u}\,\wedge\, \mathsf{w}=\mathsf{x}\,\wedge\, \mathsf{u}=\mathsf{y}).\)

Assume  that 
\(\theta_0=\exists \mathsf{w}_0(\iota_0(\mathsf{w}_0)\wedge\alpha_0(\mathsf{w}_0,\mathsf{x}))\text{ and }\theta_1=\exists \mathsf{w}_1(\iota_1(\mathsf{w}_1)\wedge\alpha_1(\mathsf{w}_1,\mathsf{y})),\) where $\alpha_0,\alpha_1$ are first-order and quantifier-free, the sequences
$\mathsf{w}_0$ and $\mathsf{w}_1$ do not have variables in common, 
\begin{equation}\label{qf_nf_thm_eq1}
\iota_0(\mathsf{w}_0)=\bigwedge_{i\in I}\mathsf{u}_i\subseteq \mathsf{v}_i~~\text{ and }~~\iota_1(\mathsf{w}_1)=\bigwedge_{j\in J}\mathsf{u}_j\subseteq \mathsf{v}_j.
\end{equation}
%\[\theta_0\equiv \exists \mathsf{w_0}\Big(\bigwedge_{i\in I}(\mathsf{u_i}\subseteq \mathsf{v_i})\wedge\alpha(\mathsf{w_0},\mathsf{x})\Big)~\text{ and }~\theta_0\equiv \exists \mathsf{w_1}\Big(\bigwedge_{j\in J}(\mathsf{u_j}\subseteq \mathsf{v_j})\wedge\beta(\mathsf{w_1},\mathsf{y})\Big).\]
%$\theta_0=\exists \mathsf{w_0}(\delta_0\wedge\alpha_0)$ and $\theta_1=\exists \mathsf{w_1}(\delta_1\wedge\alpha_1)$.

If $\theta=\theta_0\wedge\theta_1$,  then by (\ref{nf_thm_qf_eq1}) we have
\(\theta_0\wedge\theta_1\equiv\exists \mathsf{w}_0(\iota_0\wedge\alpha_0)\wedge\exists \mathsf{w}_1(\iota_1\wedge\alpha_1)\equiv \exists \mathsf{w}_0\exists \mathsf{w}_1(\iota_0\wedge\iota_1\wedge\alpha_0\wedge\alpha_1).\)
%since $\exists x\phi\wedge\psi\equiv\exists x(\phi\wedge\psi)$ holds whenever $x\notin \textsf{Fv}(\psi)$ (cf. Proposition \ref{der_rules}(i)).

If $\theta=\theta_0\vee\theta_1$, %we have
%\[\theta_0\equiv \exists \mathsf{w_0}\Big(\bigwedge_{i\in I}(\mathsf{u_i}\subseteq \mathsf{v_i})\wedge\alpha(\mathsf{w_0},\mathsf{x})\Big)~\text{ and }~\theta_0\equiv \exists \mathsf{w_1}\Big(\bigwedge_{j\in J}(\mathsf{u_j}\subseteq \mathsf{v_j})\wedge\beta(\mathsf{w_1},\mathsf{y})\Big).\]
we show that $\theta$ is equivalent to
\begin{equation}\label{qf_nf_thm_eq2}
\begin{split}
\psi=\exists \mathsf{w}_0\exists \mathsf{w}_1\exists pqp'q'\Big(&\bigwedge_{i\in I}(\mathsf{u}_ipq\subseteq \mathsf{v}_ipq)\,\wedge \bigwedge_{j\in J}(\mathsf{u}_jp'q'\subseteq \mathsf{v}_jp'q')\\
&\wedge (\alpha_0\vee\alpha_1)\wedge (\alpha_0\leftrightarrow p=q)\wedge (\alpha_1\leftrightarrow p'=q')\Big).
\end{split}
\end{equation}

We first claim that   for any first-order formula $\alpha$, any \Inc-formula $\phi$,
\begin{equation}\label{soundness_incdstr}
\exists \mathsf{x}\big(\bigwedge_{i\in I}\mathsf{u}_i\subseteq \mathsf{v}_i\,\wedge\alpha\big)\vee\phi\equiv \exists \mathsf{x}\exists pq\Big(\bigwedge_{i\in I}\mathsf{u}_ipq\subseteq \mathsf{v}_ipq\,\wedge (\alpha\leftrightarrow p=q)\wedge (\alpha\vee\phi)\Big),
\end{equation}
where each $\mathsf{u}_i$ and $\mathsf{v}_i$ consist of variables from the sequence $\mathsf{x}=\langle x_1,\dots,x_n\rangle$. Then can prove $\theta_0\vee\theta_1\equiv\psi$ by consecutively applying  (\ref{soundness_incdstr}) as follows:
\begin{align*}
~~&\exists \mathsf{w}_0\Big(\bigwedge_{i\in I}\mathsf{u}_i\subseteq \mathsf{v}_i\,\wedge\alpha_0\Big)\vee\theta_1\\
\equiv~~&\exists\mathsf{w}_0\exists pq\Big(\bigwedge_{i\in I}\mathsf{u}_ipq\subseteq \mathsf{v}_ipq\,\wedge (\alpha_0\leftrightarrow p=q)\wedge\Big(\alpha_0\vee\exists \mathsf{w}_1\big(\bigwedge_{j\in J}\mathsf{u}_j\subseteq \mathsf{v}_j\,\wedge\alpha_1\big)\Big)\Big)%\tag{by (\ref{soundness_incdstr})}
\\
\equiv~~&\exists\mathsf{w}_0\exists pq\Big(\bigwedge_{i\in I}\mathsf{u}_ipq\subseteq \mathsf{v}_ipq\,\wedge (\alpha_0\leftrightarrow p=q)\\
&\quad\wedge \exists\mathsf{w}_1\exists p'q'\Big(\bigwedge_{j\in J}\mathsf{u}_jp'q'\subseteq \mathsf{v}_jp'q'\,\wedge (\alpha_1\leftrightarrow p'=q')\wedge (\alpha_0\vee\alpha_1)\Big) \Big)\\
%\tag{by (\ref{soundness_incdstr})}
%\\
\equiv~~&\psi.  
\tag{by (\ref{nf_thm_qf_eq1})}
%(5)~~&\psi\tag{\existsi}
\end{align*}

We now complete the proof by verifying  claim (\ref{soundness_incdstr}). For the direction left to right, 
%it is sufficient to prove that
%\(\exists\mathsf{x}\big(\bigwedge_{i\in I}\mathsf{u}_i\subseteq \mathsf{v}_i\,\wedge\alpha(\mathsf{x},\mathsf{z})\big)\vee\phi(\mathsf{y},\mathsf{z})\models \exists\mathsf{x}\exists pq\chi,\) where
%\(\chi=\bigwedge_{i\in I}\mathsf{u}_ipq\subseteq \mathsf{v}_ipq\,\wedge (\alpha\leftrightarrow p=q)\wedge (\alpha\vee\phi).\)
%Now, 
suppose $M\models_X\exists\mathsf{x}\big(\bigwedge_{i\in I}\mathsf{u}_i\subseteq \mathsf{v}_i\,\wedge\alpha(\mathsf{x},\mathsf{z})\big)\vee\phi(\mathsf{y},\mathsf{z})$. Then there are  teams $Y,Z\subseteq X$ and suitable sequence of functions $\mathsf{F}=\langle F_1,\dots,F_n\rangle$ for $\exists\mathsf{x}$ such that $X=Y\cup Z$, $M\models_{Y(\mathsf{F}/\mathsf{x})}\bigwedge_{i\in I}\mathsf{u}_i\subseteq \mathsf{v}_i\,\wedge\alpha(\mathsf{x},\mathsf{z})$ and $M\models_Z\phi(\mathsf{y},\mathsf{z})$. %Since the formula $\bigwedge_{i\in I}\mathsf{u_i}\subseteq \mathsf{v_i}\,\wedge\alpha(\mathsf{x})$ is closed under unions, we may w.l.o.g. assume that $Y$ is the biggest such team. 
We now define suitable  (sequence of) functions $\mathsf{F}'=\langle F_1',\dots,F_n'\rangle,G,H$ for the quantifications $\exists\mathsf{x},\exists p,\exists q$ as follows: Pick two distinct elements $a,b\in M$. 
%\[F_1'\upharpoonright Y=F_1~\text{ and }~F_1(s)=a\text{ for all }s\in X\setminus Y.\]
%\[F_i'\upharpoonright \textsf{dom}(F_i)=F_i~\text{ and }~F'(s)=a\text{ for }s\in \textsf{dom}(F'_i)\setminus\textsf{dom}(F_i).\]
\begin{itemize}
\item Define $\mathsf{F'}$ in such a way that  the resulting team $(\mathsf{F'}/\mathsf{x})$ satisfies 
\[Y(\mathsf{F'}/\mathsf{x})=Y(\mathsf{F}/\mathsf{x})\text{ and }(X\setminus Y)(\mathsf{F'}/\mathsf{x})=(X\setminus Y)(a/\mathsf{x}),\]
where $(X\setminus Y)(a/\mathsf{x}):=\{s(a/x_1)\cdots(a/x_n)\mid s\in X\setminus Y\}$. We omit here the precise technical definition.
%For each function $F_i':X(F_1'/x_{1})\dots(F_{i-1}'/x_{i-1})\to \wp(M)\setminus\{\emptyset\}$ ($1\leq i\leq n$), let 
%\[F_{i}'(s)=\begin{cases}
%F_i(s)&\text{if }s\upharpoonright \textsf{dom}(X)\in Y;\\
%\{a\}&\text{otherwise.}
%\end{cases}\]
\item Define $G:X(\mathsf{F'}/\mathsf{x})\to \wp(M)\setminus\{\emptyset\}$ by taking $G(s)=\{a\}$.
%such that $X(\mathsf{F'}/\mathsf{x})(G/p)=X(\mathsf{F'}/\mathsf{x})(a/p)$, that is, let $G(s)=\{a\}$. 
\item Define $H:X(\mathsf{F'}/\mathsf{x})(G/p)\to \wp(M)\setminus\{\emptyset\}$ by taking
\[H(s)=\begin{cases}
\{a\}&\text{if }M\models_{s}\alpha;\\%s\upharpoonright\mathsf{x}\mathsf{y}\in Y;\\
\{b\}&\text{otherwise.}
\end{cases}\]
%such that $Y'(H/q)=Y'(a/q)$ and $\overline{Y'}(H/q)=\overline{Y'}(b/q)$, where $Y'$ is the maximal subteam of $X(\mathsf{F'}/\mathsf{x})(G/p)$ satisfying $\alpha$, i.e.,
%\[Y'=\{s\in X(\mathsf{F'}/\mathsf{x})(G/p)\mid M\models_{s}\alpha\}.\]
%and $\overline{Y'}=X(\mathsf{F'}/\mathsf{x})(G/p)\setminus Y'$. That is, we let
%\[H(s)=\begin{cases}
%\{a\}&\text{if }M\models_{s}\alpha;\\%s\upharpoonright\mathsf{x}\mathsf{y}\in Y;\\
%\{b\}&\text{otherwise.}
%\end{cases}\]
\end{itemize}
Put $W=X(\mathsf{F'}/\mathsf{x})(G/p)(H/q)$. Clearly, $M\models_{W}\alpha\leftrightarrow p=q$. It remains to show that $M\models_W\alpha\vee\phi$ and $M\models_W \mathsf{u}_ipq\subseteq \mathsf{v}_ipq$ for all $i\in I$.

For the former, define 
\[U=Y(\mathsf{F'}/\mathsf{x})(G/p)(H/q)\text{ and }V=Z(\mathsf{F'}/\mathsf{x})(G/p)(H/q).\]
%\[U=\{s\in W\mid s\upharpoonright \textsf{dom}(X)\in Y\}~~\text{and }~~V=\{s\in W\mid s\upharpoonright \textsf{dom}(X)\in Z\}.\]
Clearly $W=U\cup V$, as $X=Y\cup Z$. Since $M\models_Z\phi(\mathsf{y},\mathsf{z})$, $M\models_{Y(\mathsf{F}/\mathsf{x})}\alpha(\mathsf{x},\mathsf{z})$ and $Y(\mathsf{F}/\mathsf{x})=Y(\mathsf{F'}/\mathsf{x})$, we obtain $M\models_V\phi(\mathsf{y},\mathsf{z})$ and $M\models_U\alpha(\mathsf{x},\mathsf{z})$ by  locality. 

For the latter, let $s\in W$ be arbitrary. If $s\in U$, since $M\models_Y\mathsf{u}_i\subseteq \mathsf{v}_i$, there exists $t_0\in Y$ such that $t_0(\mathsf{v}_i)=s(\mathsf{u}_i)$. Now, since $M\models_U\alpha(\mathsf{x},\mathsf{z})$, by the definition of $H$ and $G$, we know that $s(q)=a=s(p)$. Thus, for $t=t_0(a/p)(a/q)\in W$, we have
\(t(\mathsf{v}_ipq)=\langle t_0(\mathsf{v}_i),a,a\rangle=s(\mathsf{u}_ipq).\)

If $s\in W\setminus U$, then $s\upharpoonright \mathsf{dom}(X)\in X\setminus Y$ and thereby $s(\mathsf{x})=\langle a,\dots,a\rangle$ by the definition of $\mathsf{F'}$. 
Thus, $s(\mathsf{v}_ipq)=\langle a,\dots,a,s(p),s(q)\rangle=s(\mathsf{u}_ipq)$, namely that $s$ itself is the witness of $\mathsf{u}_ipq\subseteq \mathsf{v}_ipq$ for $s$.

\vspace{0.5\baselineskip}

For the  direction right to left of the claim (\ref{soundness_incdstr}), suppose there are suitable  (sequence of) functions $\mathsf{F}=\langle F_1,\dots,F_n\rangle,G,H$ for the quantifications $\exists\mathsf{x}\exists p\exists q$ such that for $W=X(\mathsf{F}/\mathsf{x})(G/p)(H/q)$, we have that $M\models_W\bigwedge_{i\in I}\mathsf{u}_ipq\subseteq \mathsf{v}_ipq\wedge(\alpha\leftrightarrow p=q)\wedge(\alpha(\mathsf{x},\mathsf{z})\vee\phi(\mathsf{y},\mathsf{z}))$. Then there are teams $U,V\subseteq W$ such that $W=U\cup V$, $M\models_{U}\alpha$ and $M\models_{V}\phi$. Since $\alpha$ is flat, we may let $U\subseteq W$ be the maximal such team.

Consider $Y=U\upharpoonright \mathsf{dom}(X)$ and $Z=V\upharpoonright \mathsf{dom}(X)$. Clearly, $X=Y\cup Z$, and $M\models_Z\phi(\mathsf{y},\mathsf{z})$ by locality. It remains to show that $M\models_Y\exists \mathsf{x}\big(\bigwedge_{i\in I}\mathsf{u}_i\subseteq \mathsf{v}_i\,\wedge\alpha\big)$. 

Define a suitable sequence of functions $\mathsf{F'}=\langle F_1',\dots,F_n'\rangle$ for $\exists\mathsf{x}$ in such a way that $Y(\mathsf{F'}/\mathsf{x})=U\upharpoonright\mathsf{dom}(X)\cup\{x_1,\dots,x_n\}$. We omit the precise technical definition here. Now, since $M\models_{U}\alpha(\mathsf{x},\mathsf{z})$, we have that $M\models_{Y(\mathsf{F'}/\mathsf{x})}\alpha(\mathsf{x},\mathsf{z})$ by locality. To show that $Y(\mathsf{F'}/\mathsf{x})$ satisfies each $\mathsf{u}_i\subseteq \mathsf{v}_i$, take any $s\in Y(\mathsf{F'}/\mathsf{x})$. Let  $\hat{s}\in W$ be an arbitrary extension of $s$. Since $M\models_{W}\mathsf{u}_ipq\subseteq \mathsf{v}_ipq$, there exists $t\in W$ such that $\hat{s}(\mathsf{u}_ipq)=t(\mathsf{v}_ipq)$. Since $M\models_{\hat{s}}\alpha(\mathsf{x},\mathsf{z})$ and $M\models_{W}\alpha\leftrightarrow p=q$, we have that $\hat{s}(p)=\hat{s}(q)$. It then follows that  $t(p)=t(q)$, which in turn implies that $M\models_{t}\alpha$. Then $t\in U$, as $U$ was assumed to be the maximal subteam of $W$ that satisfies $\alpha(\mathsf{x},\mathsf{z})$. Hence, $t_0=t\upharpoonright\mathsf{dom}(X)\cup\{x_1,\dots,x_n\} \in Y(\mathsf{F'}/\mathsf{x})$ and $t_0(\mathsf{v}_i)=t(\mathsf{v}_i)=\hat{s}(\mathsf{u}_i)=s(\mathsf{u}_i)$.
\end{proof}

Finally, by using the above normal form results we obtain the desired more refined normal form as follows.

\begin{theorem}\label{inc_nf_thm}
Every \Inc-formula $\phi(\mathsf{z})$ is semantically equivalent to a formula of the form
\begin{equation}\label{inc_nf_cp}
\exists \mathsf{w}\exists \mathsf{x}\forall y\Big(\bigwedge_{i\in I}\mathsf{u}_i\subseteq \mathsf{v}_i\,\wedge\bigwedge_{j\in J}\mathsf{z}x_1\dots x_{j-1}y\subseteq \mathsf{z}x_1\dots x_{j-1}x_j\,\wedge\alpha(\mathsf{w},\mathsf{x},\mathsf{z})\Big),
\end{equation}
where $\alpha$ is a first-order quantifier-free formula, and each $\mathsf{u}_i$ and $\mathsf{v}_i$ are sequences of variables from $\mathsf{w}$.%\todo{the length of $\mathsf{x}$ is not correct}\todo{order of w and x; order of existential and universal}
\end{theorem}
\begin{proof}
By Theorem \ref{nf_thm_qf}, we may assume that $\phi$ is in prenex normal form (\ref{nf_qf}). Furthermore, by Lemma \ref{qf_nf_thm}, the quantifier free formula $\theta$ in (\ref{nf_qf}) is equivalent to a formula of the form (\ref{qf_nf}). Hence, $\phi(\mathsf{z})$ is equivalent to a formula of the form
\[Q^1x_1\dots Q^nx_n\exists \mathsf{w}\Big(\bigwedge_{i\in I}\mathsf{u}_i\subseteq \mathsf{v}_i\,\wedge\alpha(\mathsf{w},\mathsf{x},\mathsf{z})\Big).\]
Finally, by applying Theorem \ref{nf_thm_uni} to the above formula (and rearranging the order of the existential quantifiers) we obtain an equivalent formula of the form (\ref{inc_nf_cp}).
\end{proof}

To simplify notations in the normal form (\ref{inc_nf_cp}), we now introduce some conventions. For any permutation $\mathsf{f}:\{1,\dots,n\}\to\{1,\dots,n\}$ and $k\leq n$, we define a function $\sigma^{\mathsf{f},k}_{(\cdot)}:\mathsf{Var}^n\to \mathsf{Var}^k$ by taking 
\(\sigma^{\mathsf{f},k}_{\mathsf{x}}=x_{\mathsf{f}(1)}\dots x_{\mathsf{f}(k)}\) for any sequence $\mathsf{x}=\langle x_1\dots x_n\rangle$. That is, $\sigma^{\mathsf{f},k}_{\mathsf{x}}$ is a sequence of variables from $\mathsf{x}$. When no confusion arises we drop the superscripts in $\sigma^{\mathsf{f},k}_{\mathsf{x}}$ and write simply $\sigma_{\mathsf{x}}$. We reserve the greek letters $\pi,\rho,\sigma,\tau$ (with or without superscripts) for such functions. The normal form of an \Inc-sentence (with no free variables) can then be written as
\begin{equation}\label{inc_nf}
\exists\mathsf{w}\exists \mathsf{x}\forall y\Big(\bigwedge_{i\in I}\rho^i_{\mathsf{w}}\subseteq \sigma^i_{\mathsf{w}}\,\wedge \bigwedge_{j\in J}\pi^j_{\mathsf{x}}y\subseteq \tau^j_{\mathsf{x}}\wedge\alpha(\mathsf{w},\mathsf{x})\Big).
\end{equation}
%where $\mathsf{x}=\langle x_1,\dots, x_m\rangle$, and each $\pi_{\mathsf{x}}^j=x_{1}\dots x_{j-1}$ and $\tau_{\mathsf{x}}^j=x_1\dots x_j$. \todo{do not specify $\pi_{\mathsf{x}}^j$ and $\tau_{\mathsf{x}}^j$}

%For a sequence $\mathsf{x}=x_1\dots x_n$, we sometimes write $\sigma_{\mathsf{x}}$ for a sequence $x_{\pi(1)}\dots x_{\pi(k)}$ ($k\leq n$) generated by $\pi$ and $\mathsf{x}$. In this case, for a sequence $\mathsf{y}=y_1\dots y_n$, $\sigma_{\mathsf{y}}=y_{\pi(1)}\dots y_{\pi(k)}$. \todo{improve}

%In the sequel we will mainly be focusing on \Inc-sentences, namely \Inc-formulas without free variables, which are always assumed to be in the normal form given in Theorem \ref{inc_nf_thm}, i.e., 
%%\todo{$\rho_{\mathsf{x}}$ is a position map}
%\begin{equation}\label{inc_nf}
%\exists\mathsf{w}\exists \mathsf{x}\forall y\Big(\bigwedge_{i=1}^k\rho^i_{\mathsf{w}}\subseteq \sigma^i_{\mathsf{w}}\,\wedge \bigwedge_{j=1}^m\pi^j_{\mathsf{x}}y\subseteq \tau^j_{\mathsf{x}}\wedge\alpha(\mathsf{w},\mathsf{x})\Big),
%\end{equation}
%where each $\pi_{\mathsf{x}}^j=x_{1}\dots x_{j-1}$ and $\tau_{\mathsf{x}}^j=x_1\dots x_j$. %are sequences of variables from $\mathsf{x}$. 

%\todo{stress (already in the normal form) that this is for the existential quantifier}

% \todo{this deals with the universal qf}

Observe that the formula in the above normal form has only one (explicit) universal quantifier (i.e., $\forall y$). Yet because of the inclusion atoms $\pi^j_{\mathsf{x}}y\subseteq \tau^j_{\mathsf{x}}$ in the formula, some existentially quantified variables from $\mathsf{x}$ are essentially universally quantified  (cf. equivalence \eqref{nf_thm_uni_eq1}). 
%Thanks to this type of implicit universal quantification, only one quantifier alternation is needed in the normal form (\ref{inc_nf}). This feature simplifies a lot the game expressions of \Inc-sentences that we will discuss in the next section. \todo{shorten}

%\todo{by using implicit universal quantifier, one can change the order of the quantifiers: existential quantifiers can be rearranged freely}

\section{Game expression and approximations}\label{sec:game_exp}

In this section, we define the game expression $\Phi$ for every \Inc-sentence $\phi$ (with no free variables) in normal form. Intuitively the  formula $\Phi$ is a first-order sentence of infinite length that simulates all possible plays in the semantic game (in team semantics) of the formula $\phi$. Over countable models $\Phi$ and $\phi$ are equivalent, as we will show in Theorem \ref{approx_game_expr}.
  %for each \Inc-formula $\phi$ in normal form, we introduce an equivalent (over countable models) game expression $\Phi$ in the infinitary logic $L_{\omega_1\omega}$.    \todo{explain the intuitive idea of game expression}
  For a game of finite length $n$, we define a  first-order formula $\Phi^n$ of finite length,  called the {\em $n$-approximation} of $\Phi$. It follows from the model-theoretic argument in \cite{Axiom_fo_d_KV} that $\Phi$ is equivalent to the (infinitary) conjunction of all its approximations $\Phi^n$ over countable recursively saturated models. These game expressions and their finite approximations will be cruicial for proving the completeness theorem for the system of \Inc to be introduced in the next section.%\todo{improve}

Now, let $\phi$ be an \Inc-sentence  (with no free variables). By Theorem \ref{inc_nf_thm}, we may assume that $\phi$ is in normal form (\ref{inc_nf}).
We now define the game expression of $\phi$ as the following first-order sentence $\Phi$  of infinite length: %\todo{simplify}
\begin{align*}
\Phi:=
& \exists \mathsf{w}_0 \exists \mathsf{x}_0\forall y_0\Big(\alpha(\mathsf{w}_0 ,\mathsf{x}_0)\wedge\\
&\exists \mathsf{w}^1\mathsf{x}^1\forall y_1\Big( \alpha_1(\mathsf{w}^1\mathsf{x}^1)
\wedge\gamma_1(\mathsf{w}_0\mathsf{w}^1)\wedge\delta_1(y_0,\mathsf{x}_0\mathsf{x}^1)\wedge\\
%&\exists \mathsf{w}_{k+m+1}\mathsf{x}_{k+m+1}\dots \mathsf{w}_{p_2}\mathsf{x}_{p_2}\forall y_2\Big(\alpha_2(\mathsf{w}_{k+m+1}\mathsf{x}_{k+m+1},\dots,\mathsf{w}_{p_2}\mathsf{x}_{p_2})\\
%&\quad\quad\quad\quad\quad\quad\quad\quad\quad\quad\quad\quad\quad\wedge\gamma_2(\mathsf{w}_{1},\dots, \mathsf{w}_{q_2})\wedge\delta_2(y_0y_1,\mathsf{x}_{0},\dots ,\mathsf{x}_{p_2})\wedge\\
&\quad\quad\quad\quad\quad\quad\quad\quad\quad\quad\quad\quad\quad\quad\dots~ \dots\\
&\exists \mathsf{w}^n\mathsf{x}^n\forall y_n\Big(\alpha_n(\mathsf{w}^n\mathsf{x}^n)\wedge\gamma_n(\mathsf{w}^{n-1}\mathsf{w}^n)\wedge\delta_n(y_0\dots y_n,\mathsf{x}_{0}\mathsf{x}^1\dots \mathsf{x}^{n})\wedge\dots\Big)\quad\dots\quad\Big)\Big),
%&\quad\quad\quad\quad\quad\quad\quad\quad\quad\quad\quad\quad\quad\quad\quad\quad~\dots~\dots~\quad\quad\quad\Big)\quad\quad\dots\quad\quad\Big),
\end{align*}
where 
\begin{itemize}
\item $\mathsf{w}^n=\langle \mathsf{w}_\xi\mid \xi\in E_n\cup U_n\rangle$ and $\mathsf{x}^n=\langle \mathsf{x}_\xi\mid \xi\in E_n\cup U_n\rangle$ with %\todo{too vague}
\begin{itemize}
\item $E_n$ being the set of indices $\langle \xi,i\rangle$ of variables $\mathsf{w}_{\xi,i}$ introduced  as witnesses for each $\rho_{\mathsf{w}}^i\subseteq\sigma^i_{\mathsf{w}}$ with respect to the variables $\mathsf{w}_{\xi}$ from $\mathsf{w}^{n-1}$,
\item $U_n$ being the set of indices $\langle \xi\eta,j\rangle$  of variables $\mathsf{x}_{\xi\eta,j}$ introduced as witnesses for each $\pi^j_{\mathsf{x}}y\subseteq \tau^j_{\mathsf{x}}$  with respect to all new pairs $\mathsf{x}_\xi y_\eta$  with $\mathsf{x}_\xi$ from $\mathsf{x}_0\mathsf{x}^1\dots\mathsf{x}^{n-1}$ and  $y_\eta$ from $y_0\dots y_n$ (write 
\[A_n=\{\xi\eta\mid \langle \xi\eta,j\rangle\in U_n\text{ for some }j\in J\},\] %denote the set consisting of all such new pairs $\xi\eta$ of indices by $A_n$, 
and note that we are requiring that $\xi\eta\notin A_1\cup\dots A_{n-1}$);

\end{itemize}
\item $\displaystyle\alpha_n(\mathsf{w}^n\mathsf{x}^n):=\bigwedge_{\xi\in E_n\cup U_n}\alpha(\mathsf{w}_\xi,\mathsf{x}_\xi)$;
\item $\displaystyle\gamma_n(\mathsf{w}^{n-1}\mathsf{w}^n):=\bigwedge_{\xi\in E_{n-1}}\bigwedge_{i\in I}\rho_{\mathsf{w}_\xi}^i=\sigma^i_{\mathsf{w}_{\xi,i}}$; %with $\langle \xi,i\rangle=p_{n-1}+k(\xi-p_{n-2})+i$;
%\[\mathsf{w}_\xi^{n,i}=\mathsf{w}_{p_{n-1}+k\xi_0+i}\text{ and }\xi=p_{n-2}+\xi_0;\] 
\item $\displaystyle\delta_n(y_0\dots y_{n-1},\mathsf{x}_{0}\mathsf{x}^1\dots\mathsf{x}^n):=\bigwedge_{\xi\eta\in A_n}\bigwedge_{j\in J}\pi^j_{\mathsf{x}_\xi} y_\eta=\tau^j_{\mathsf{x}_{\xi\eta,j}}$.
\end{itemize}
%\todo{move the introduction of $\Phi_n$ to here?}
The formula $\Phi$ is defined in layers that correspond essentially to the plays in the semantic game of the formula $\phi$ (see e.g., \cite{Pietro_thesis} for the definition of the semantic game for \Inc).  Each layer of $\Phi$ consists of the subformula $\exists \mathsf{w}^{n}\mathsf{x}^n\forall y_n(\alpha_n\wedge\gamma_n\wedge\delta_n\wedge\dots)$ with $\mathsf{w}^0\mathsf{x}^0=\mathsf{w}_0\mathsf{x}_0$ and $\alpha_0\wedge\gamma_0\wedge\delta_0=\alpha(\mathsf{w}_0,\mathsf{x}_0)\wedge\top\wedge\top\equiv\alpha(\mathsf{w}_0,\mathsf{x}_0)$.
The intuitive reading of each layer is as follows:  Each layer introduces new existentially quantified variables $\mathsf{w}^n\mathsf{x}^n$  and one universally quantified variable $y_n$, and specifies (in $\alpha_n$) that $\alpha$ holds for the existentially quantified variables $\mathsf{w}^n\mathsf{x}^n$. For each inclusion atom $\rho^i_{\mathsf{w}}\subseteq \sigma^i_{\mathsf{w}}$ in $\phi$, with respect to each sequence $\mathsf{w}_\xi$ of existentially quantified variables introduced in layer $n-1$, a witness sequence $\mathsf{w}_{\xi,i}$ of variables  (as specified in the formula $\gamma_n$), together with the accompanying sequence $\mathsf{x}_{\xi,i}$, are introduced in layer $n$ as part of $\mathsf{w}^n\mathsf{x}^n$. Similarly, for each inclusion atom $\pi^j_{\mathsf{x}}y\subseteq \tau^j_{\mathsf{x}}$ in $\phi$, with respect to each new combination $\mathsf{x}_\xi y_\eta\in A_n$ of existentially quantified variables $\mathsf{x}_\xi$ introduced  up to layer $n-1$ and universally quantified variables $y_\eta$ introduced up to layer $n$, a witness sequence $\mathsf{x}_{\xi\eta,j}$ of variables  (as specified in the formula $\delta_n$) together with the accompanying sequence $\mathsf{w}_{\xi,i}$ are  introduced in layer $n$ as part of $\mathsf{w}^n\mathsf{x}^n$.
Note that $E_{n+1}=\{\langle \xi,i\rangle\mid \xi\in E_{n}\cup U_{n},i\in I\}$ and $U_{n+1}=\{\langle \xi\eta,j\rangle\mid \xi\eta\in A_{n+1},j\in J\}$. %\todo{improve}

We assume that the reader is familiar with the game-theoretic semantics of first-order and infinitary logic. Let us now recall the semantic game $\mathcal{G}(M,\Phi)$ of the formula $\Phi$ over a model $M$, which is an infinite game played  between two players  $\forall$belard and $\exists$loise. At each round the players take turns to pick elements from $M$ for the quantified variables $\mathsf{w}^n\mathsf{x}^n$ and $y_n$, as illustrated in the following table:
%\begin{table}
\begin{center}
\scalebox{0.965}{
\begin{tabular}{c|c|cc|c|cc|c}
round&0&\multicolumn{2}{c|}{1}&~$\cdots$~&\multicolumn{2}{c|}{n}&~$\cdots$\\\hline
$\forall$&&$c_1$&&~$\cdots$~&$c_n$&&~$\cdots$\\\hline
$\exists$&$\mathsf{a}^0\mathsf{b}^0$&&$\mathsf{a}^1\mathsf{b}^1$&~$\cdots$~&&$\mathsf{a}^n\mathsf{b}^n$\!&~$\cdots$
\end{tabular}\label{tab:game}
}
%\caption{Moves in the semantics game $\mathcal{G}(M,\Phi)$.}
\end{center}
%\end{table}
The choices of the two players generate an assignment $\mathfrak{s}$ for the quantified variables $\mathsf{w}^n\mathsf{x}^ny_n$ defined as
\[\mathfrak{s}(\mathsf{w}^n\mathsf{x}^n)=\mathsf{a}^n\mathsf{b}^n~\text{ and }~\mathfrak{s}(y_n)=c_n.\]
The player $\exists$loise {\em wins} the (infinite) game if for each natural number $n$,
\[M\models_{\mathfrak{s}}\alpha_n\wedge\gamma_n\wedge\delta_n.\] 
%where we stipulate $\alpha_0\wedge\gamma_0\wedge\delta_0=\alpha(\mathsf{w}_0,\mathsf{x}_0)\wedge\top\wedge\top\equiv\alpha(\mathsf{w}_0,\mathsf{x}_0)$,
%\[\lambda_n=\gamma(\mathsf{x}_{p_{n-1}k+1},\dots,\mathsf{x}_{p_nk})\wedge\bigwedge_{\mathsf{x}_\xi y_\eta\in A_n}\!\delta(\mathsf{x}_{\xi}y_\eta,\mathsf{x}^{n,1}_{\xi\eta}\dots \mathsf{x}^{n,k}_{\xi\eta})\]
%\begin{equation}\label{win_cond}
%\begin{split}
%&M\models_{s}\theta_n,\text{ where }\\
%&\theta_n=\bigwedge_{\xi\eta\in A_n}\!(\mathsf{z}_{\xi\eta}=\mathsf{x}_\xi y_\eta)\wedge\gamma(\mathsf{x}_{p_{n-1}k+1},\dots,\mathsf{x}_{p_nk})\wedge\bigwedge_{\xi\eta\in A_n}\!\delta(\mathsf{z}_{\xi\eta},\mathsf{x}_{\tau_n(\xi\eta)}\dots \mathsf{x}_{\tau_n(\xi\eta)+k}),
%\end{split}
%\end{equation}
Finally,
\[M\models\Phi\iff\exists\text{loise has a winning strategy in the game }\mathcal{G}(M,\Phi),\]
where a {\em winning strategy} for $\exists$loise is a function that tells her what to choose at each round, and also guarantees her to win every play of the game.
We now show that an \Inc-sentence is semantically equivalent to its game expression  over countable models by using the game-theoretic semantics.%\todo{this will not be used in the completeness proof}%, in the sense that the two sentences have the same  countable models.

\begin{theorem}\label{approx_game_expr}
Let $\phi$ be an \Inc-sentence, and $M$ a model. Then
\begin{enumerate}[label=(\roman*)]
\item $M\models \phi\Longrightarrow M\models\Phi$,
\item and $M\models\Phi\Longrightarrow M\models\phi$, whenever $M$ is a countable model.
\end{enumerate}
\end{theorem}
\begin{proof}
(i) Suppose $M\models\phi$. Then, there exists a suitable sequence $\mathsf{F}$ of functions for $\exists \mathsf{w}\exists\mathsf{x}$ such that for $X=\{\emptyset\}(\mathsf{F}/\mathsf{w}\mathsf{x})$,
\begin{equation}\label{approx_game_expr_eq1}
M\models_{X(M/y)} \bigwedge_{i\in I}\rho^i_{\mathsf{w}}\subseteq \sigma^i_{\mathsf{w}}\,\wedge \bigwedge_{j\in J}\pi^j_{\mathsf{x}}y\subseteq \tau^j_{\mathsf{x}}\wedge\alpha(\mathsf{w},\mathsf{x}).
\end{equation}
We prove $M\models\Phi$ by constructing a winning strategy for $\exists$loise  in the semantic game $\mathcal{G}(M,\Phi)$ as follows: %\todo{add a pic?}

\begin{itemize}
\item 
 In round $0$, choose any assignment $s_0$ in $X$, and let $\exists$loise choose $\mathsf{a}^0=s_0(\mathsf{w})$ and $\mathsf{b}^0=s_0(\mathsf{x})$. Let $\mathfrak{s}_0$ be the assignment for $\mathsf{w}_0\mathsf{x}_0$ generated by $\exists$loise's choices so far. By (\ref{approx_game_expr_eq1}), we have that $M\models_{s_0}\alpha(\mathsf{w},\mathsf{x})$, which implies $M\models_{\mathfrak{s}_0}\alpha_0(\mathsf{w}_0,\mathsf{x}_0)$, thus the winning condition is maintained.

%Moreover, for any choice $b_0$ that $\forall$belard can make in round $1$,  the assignment $s(b_0/y)$ generated by the players's choices so far must belong to the team $X(M/y)$.
\item Let $\mathfrak{s}_{n-1}$ be the assignment generated by the choices of the two players up to round $n-1$. Assume that we have maintained that for each $\mathsf{w}_\xi\mathsf{x}_\xi$ in the domain of $\mathfrak{s}_{n-1}$, the assignment $s_\xi$ for $\mathsf{wx}$ defined as $s_\xi(\mathsf{w}\mathsf{x})=\mathfrak{s}_{n-1}(\mathsf{w}_\xi\mathsf{x}_\xi)$ is in $X$, and
%Moreover, for any choice $b_n$ that $\forall$belard makes in round $n$, the assignment $s_\xi(b_n/y)$ must be in $X(M/y)$. 
assume that $\forall$belard has chosen $c_n$ in round $n$. %We now describe $\exists$loise's response in round $n$.

 %Assume that $\forall$belard has chosen $b_n$ in round $n$.

%Clearly, for any choice $b_n$ that $\forall$belard makes in round $n$, we have $\mathfrak{s}_{n}'=\mathfrak{s}_{n-1}(b_n/y)\in X(M/y)$. We now describe $\exists$loise's response in round $n$.

\begin{itemize}
\item For any $\xi\eta\in A_n$ with $\mathsf{a}_{\xi}\mathsf{b}_{\xi}c_{\eta}$ the corresponding choices by the two players in (at most) two earlier than $n$ rounds, the assignment $s_\xi(c_\eta/y)$ must be in $X(M/y)$. For each $j\in J$, since $M\models_{X(M/y)}\pi^j_{\mathsf{x}}y\subseteq \tau^j_{\mathsf{x}}$, there exists $s'\in X(M/y)$ such that $s'(\tau^j_{\mathsf{x}})=\langle s_\xi(\pi^j_{\mathsf{x}}),c_\eta\rangle$. Let $\exists$loise choose $\mathsf{b}_{\xi\eta,j}=s'(\mathsf{x})$ and $\mathsf{a}_{\xi\eta,j}=s'(\mathsf{w})$. Clearly, $\delta_n$ is satisfied by the assignment generated by the players' choices so far, and $s_{\xi\eta,j}=s'\upharpoonright \mathsf{dom}(X)\in X$.

\item Similarly, for any $\xi\in E_{n-1}$ and any $i\in I$, by using the fact that $s_\xi\in X$ and $M\models_{X(M/y)}\rho^i_{\mathsf{w}}\subseteq \sigma^i_{\mathsf{w}}$, we can let $\exists$loise choose $\mathsf{a}_{\xi,i}\mathsf{b}_{\xi,i}$ so that $\gamma_n$ is satisfied by the assignment generated by the players' choices so far, and $s_{\xi,i}\in X$.

%For any $\mathsf{w}_\xi\in\mathsf{w}^{n-1}$ with the corresponding choice $\mathsf{a}_\xi$ of $\exists$loise in round $n-1$, by our construction, there must exists an assignment $s\in X(M/y)$ such that $s(\mathsf{w})=\mathsf{a}_\xi$. For each $i\in I$, since $M\models_{X(M/y)}\rho^i_{\mathsf{w}}\subseteq \sigma^i_{\mathsf{w}}$, there exists $s'\in X(M/y)$ such that $s'(\sigma^i_{\mathsf{w}})=s(\rho^i_{\mathsf{w}})$. Let $\exists$loise choose $\mathsf{a}_{\xi,i}=s'(\mathsf{w})$ and $\mathsf{b}_{\xi,i}=s'(\mathsf{x})$. Clearly, $\rho^i_{\mathsf{a}_\xi}=s(\rho^i_{\mathsf{w}})=s'(\sigma^i_{\mathsf{w}})=\sigma^i_{\mathsf{a}_{\xi,i}}$, thereby $\gamma_n$ is satisfied by the assignment generated by the player's choices so far.

%\item Similarly, for any $\xi\eta\in A_n$ with $\mathsf{a}_{\xi}\mathsf{b}_{\xi}c_{\eta}$ the corresponding choices by the two players in (at most) two rounds $\leq n$, by our construction, there must exists $s\in X(M/y)$ such that $s(\mathsf{wx})=\mathsf{a}_\xi\mathsf{b}_\xi$ and $s(y)=c_\eta$. For each $j\in J$, since $M\models_{X(M/y)}\pi^j_{\mathsf{x}}y\subseteq \tau^j_{\mathsf{x}}$, there exists $s'\in X(M/y)$ such that $s'(\tau^j_{\mathsf{x}})=s(\pi^j_{\mathsf{x}}y)$. Let $\exists$loise choose $\mathsf{b}_{\xi\eta,j}=s'(\mathsf{x})$ and $\mathsf{a}_{\xi\eta,j}=s'(\mathsf{w})$. Clearly, $\delta_n$ is satisfied by the assignment generated by the player's choices so far.
\end{itemize}

Moreover, since $M\models_{X(M/y)}\alpha(\mathsf{w},\mathsf{x})$ and we have maintained that $s_\xi\in X$ for each $\xi\in E_n\cup U_n$, we conclude that each $\alpha(\mathsf{w}_\xi,\mathsf{x}_\xi)$ is satisfied by the assignment $\mathfrak{s}_n$ generated by the choices of the players till round $n$.
%since $M\models_{X(M/y)}\alpha(\mathsf{w},\mathsf{x})$ and the assignment $\mathfrak{s}_n$ generated by the choices of the two players till round $n$ is guaranteed to be in $X(M/y)$, $\mathfrak{s}_n$ satisfies each $\alpha(\mathsf{w}_\xi,\mathsf{x}_\xi)$ for $\mathsf{w}_\xi\mathsf{x}_\xi$ from $\mathsf{w}^{n}\mathsf{x}^{n}$.
\end{itemize}

(ii) Suppose  $M$ is a countable model of $\Phi$, and $\exists$loise has a winning strategy in the game $\mathcal{G}(M,\Phi)$.  Let $\langle c_n\rangle_{n<\omega}$ enumerate all elements of $M$, and let $\forall$belard play $c_n$ at each round $n$.  Suppose $\mathfrak{s}$ is the assignment generated by such choices of $\forall$belard and the corresponding choices of $\exists$loise given by her winning strategy. Let
\[X=\{s_\xi\mid\xi\in E_n\cup U_n,~n<\omega\},\]
where recall that $s_\xi$ is the assignment for $\mathsf{w}\mathsf{x}$ defined as $s_\xi(\mathsf{w}\mathsf{x})=\mathfrak{s}(\mathsf{w}_\xi\mathsf{x}_\xi)$. Observe that $X=\{\emptyset\}(\mathsf{F}/\mathsf{wx})$ for some suitable sequence $\mathsf{F}$ of functions for  $\exists\mathsf{w}\exists\mathsf{x}$.  To show $M\models\phi$,  it suffices to verify that the team $X(M/y)$ satisfies (\ref{approx_game_expr_eq1}).

To see that $M\models_{X(M/y)}\alpha(\mathsf{w},\mathsf{x})$, for any $s_\xi(c_\eta/y)\in X(M/y)$, since $\exists$loise wins the game, we know that $M\models_{\mathfrak{s}}\alpha(\mathsf{w}_\xi,\mathsf{x}_\xi)$, which implies $M\models_{s_\xi}\alpha(\mathsf{w},\mathsf{x})$, as desired.

%For each $s_\xi(c_n/y)\in X(M/y)$, we have $s(\mathsf{wx})=\mathfrak{s}(\mathsf{w}_\xi\mathsf{x}_\xi)$ for some $\xi<\omega$.
%Since $\exists$loise wins the game, we know that $M\models_{\mathfrak{s}}\alpha(\mathsf{w}_\xi,\mathsf{x}_\xi)$, which implies $M\models_{s}\alpha(\mathsf{w},\mathsf{x})$, and thereby $M\models_{X(M/y)}\alpha(\mathsf{w},\mathsf{x})$.

To see that $X(M/y)$ satisfies each $\pi^j_{\mathsf{x}}y\subseteq \tau^j_{\mathsf{x}}$, take any $s_\xi(c_\eta/y)\in X(M/y)$ and assume  $\xi \eta\in A_n$. Since $\exists$loise wins the game, $\mathfrak{s}$ satisfies $\delta_n$, and in particular, $M\models_{\mathfrak{s}}\pi^j_{\mathsf{x}_\xi} y_\eta=\tau^j_{\mathsf{x}_{\xi\eta,j}}$. Thus, for any extension $\hat{s}\in X(M/y)$ of $s_{\xi\eta,j}\in X$, we have that $\hat{s}(\tau^j_{\mathsf{x}})=s_{\xi\eta,j}(\tau^j_{\mathsf{x}})=\mathfrak{s}(\tau^j_{\mathsf{x}_{\xi\eta,j}})=\mathfrak{s}(\pi^j_{\mathsf{x}_\xi} y_\eta)=\langle\mathfrak{s}(\pi^j_{\mathsf{x}_\xi}), c_\eta\rangle=s_\xi(c_\eta/y)(\pi^j_{\mathsf{x}} y)$,  as required.

%Then $s(\mathsf{w}\mathsf{x})=\mathfrak{s}(\mathsf{w}_\xi\mathsf{x}_\xi)$ and $s(y)=\mathfrak{s}(y_\eta)$ for some $\xi,\eta<\omega$, that is, the  elements $s(\mathsf{w}\mathsf{x})$ and $s(y)$ have been chosen by the two players in some round(s) of the game. Let $n$ be the number such that   $\xi \eta\in A_n$. By the definition of $X$, there exists $s'\in X(M/y)$ such that $s'(\mathsf{x})=\mathfrak{s}(\mathsf{x}_{\xi\eta,j})$. Since $\exists$loise wins the game, $\mathfrak{s}$ satisfies $\delta_n$. Thus, $s(\pi^j_{\mathsf{x}} y)=\mathfrak{s}(\pi^j_{\mathsf{x}_\xi} y_\eta)=\mathfrak{s}(\tau^j_{\mathsf{x}_{\xi\eta,j}})=s'(\tau^j_{\mathsf{x}})$,  as required.

By a similar argument, we can also show that $X(M/y)$ satisfies each $\rho^i_{\mathsf{w}}\subseteq \sigma^i_{\mathsf{w}}$. This then finishes the proof.
\end{proof}

For each natural number $n<\omega$, we define the {\em $n$-approximation} $\Phi_n$ of the infinitary sentence $\Phi$ as the finite first-order formula
\[\Phi_n:= \exists \mathsf{w}_0\mathsf{x}_0\forall y_0\big(\alpha_0\wedge  \exists \mathsf{w}^{1}\mathsf{x}^1\forall y_1\big(\alpha_1\wedge\gamma_1\wedge\delta_1\wedge\dots\wedge\exists \mathsf{w}^{n}\mathsf{x}^n\forall y_n(\alpha_n\wedge\gamma_n\wedge\delta_n\underbrace{)\dots\big)\big)}_{n+1}.
\]
The semantic game for $\Phi_n$ over a model $M$, denoted by $\mathcal{G}(M,\Phi_n)$, is defined exactly as the infinite game $\mathcal{G}(M,\Phi)$ except that $\mathcal{G}(M,\Phi_n)$ has only $n+1$ rounds. Using the game theoretic-semantics we show, as in \cite{Axiom_fo_d_KV}, that the $\Phi_n$'s do approximate $\Phi$ over {\em recursively saturated models}, which (recall from, e.g., \cite{BarwiseSchlipf76}) are models $M$ such that for any recursive set $\{\phi_n(\mathsf{x},\mathsf{y})\mid n<\omega\}$ of formulas,
\[M\models\forall \mathsf{x}\Big(\bigwedge_{n<\omega}\exists \mathsf{y}\bigwedge_{m\leq n}\phi_m(\mathsf{x},\mathsf{y})\to\exists \mathsf{y}\bigwedge_{n<\omega}\phi_n(\mathsf{x},\mathsf{y})\Big).\]

\begin{theorem}\label{satu_md_approx_equi}
If $M$ is a recursively saturated (or finite) model, then 
\[M\models\Phi\iff M\models\Phi_n\text{ for all }n<\omega.\]
In particular, if $M$ is a  recursively saturated countable (or finite) model, then
\[M\models\phi\iff M\models\Phi_n\text{ for all }n<\omega.\]
\end{theorem}
\begin{proof}
%\todo{omit}
The ``in particular" part follows from Theorem \ref{approx_game_expr}. The direction ``$\Longrightarrow$" of the main claim follows from the observation that a winning strategy for $\exists$loise in the infnite game $\mathcal{G}(M,\Phi)$ is clearly also a winning strategy for $\exists$loise in the finite game $\mathcal{G}(M,\Phi_n)$ for every $n<\omega$. 
The  other direction ``$\Longleftarrow$" follows from a similar argument to that of Proposition 15 in \cite{Axiom_fo_d_KV}, which we omit here. 
\end{proof}

\section{A system of natural deduction for \Inc}

%\todo{No closure under substitution has to be mentioned somewhere}

In this section, we introduce a  system of natural deduction for inclusion logic and prove the soundness theorem of the system. We also prove in the system that every \Inc-formula implies its normal form.

% the normal form of an \Inc-formula is derivable in the system from the formula itself.

\begin{table}[t]
\caption{Rules for equality, connectives and quantifiers}
{\normalsize
\begin{center}
\renewcommand{\arraystretch}{1.8}
\begin{tabular}{|C{0.45\linewidth}C{0.45\linewidth}|}
%\multicolumn{2}{c}{\textbf{NEW RULES}}\\\hline\hline
\hline
%Negation elimination&\emph{Ex falso}\\
%&\\
\AxiomC{}\RightLabel{\eqi}\UnaryInfC{$t=t$}\noLine\UnaryInfC{}\DisplayProof
&\AxiomC{}\noLine\UnaryInfC{$t=t'$}\AxiomC{$\phi(t/x)$}\RightLabel{\eqsub}\BinaryInfC{$\phi(t'/x)$}\noLine\UnaryInfC{}\DisplayProof
%&{\footnotesize No free variable in $t$ become bound in $\phi(t/x)$; the same for $t'$.}
\\\hline
\multicolumn{2}{|c|}{ \AxiomC{}\noLine\UnaryInfC{$[\alpha]$}\noLine\UnaryInfC{$D$}\noLine\UnaryInfC{$\bot$} \RightLabel{\negi ~{\footnotesize(1)}}\UnaryInfC{$\neg\alpha$}\noLine\UnaryInfC{} \DisplayProof
%\quad\quad \AxiomC{$\bot$}\RightLabel{$\mathsf{ex~ falso}$\todo{derivable}}\UnaryInfC{$\phi$} \DisplayProof
\quad \quad\quad \quad\quad\AxiomC{$\alpha$}\AxiomC{$\neg\alpha$}\RightLabel{\nege}\BinaryInfC{$\phi$} \DisplayProof
 %\quad\quad\AxiomC{$\neg\neg\alpha$}\RightLabel{\nnege\todo{not needed}}\UnaryInfC{$\alpha$} \DisplayProof
\quad \quad\quad \quad\quad \AxiomC{[$\neg\alpha$]}\noLine\UnaryInfC{$D$}\noLine\UnaryInfC{$\bot$}\RightLabel{$\mathsf{RAA}$ ~{\footnotesize(1)}}\UnaryInfC{$\alpha$} \DisplayProof}\\\hline
%Conjunction introduction& Conjunction elimination\\
 \AxiomC{}\noLine\UnaryInfC{$\phi$}\AxiomC{$\psi$} \RightLabel{\conji}\BinaryInfC{$\phi\wedge\psi$}\noLine\UnaryInfC{} \DisplayProof&\AxiomC{$\phi$} \RightLabel{\tensori}\UnaryInfC{$\phi\vee\psi$}\DisplayProof\quad\quad\AxiomC{$\phi$} \RightLabel{\tensori}\UnaryInfC{$\psi\vee\phi$}\DisplayProof
\\%\hline
%Disjunction introduction& Disjunction elimination\\
 \AxiomC{$\phi\wedge\psi$} \RightLabel{\conje}\UnaryInfC{$\phi$}\noLine\UnaryInfC{} \DisplayProof\quad\quad \AxiomC{$\phi\wedge\psi$} \RightLabel{\conje}\UnaryInfC{$\psi$}\noLine\UnaryInfC{} \DisplayProof&\AxiomC{$\phi\vee\psi$}\AxiomC{[$\phi$]}\noLine\UnaryInfC{$D_0$}\noLine\UnaryInfC{$\chi$} \AxiomC{}\noLine\UnaryInfC{[$\psi$]}\noLine\UnaryInfC{$D_1$}\noLine\UnaryInfC{$\chi$}\RightLabel{\tensore ~{\footnotesize(2)}} \TrinaryInfC{$\chi$}\noLine\UnaryInfC{}\DisplayProof
\\%[-4pt]
%&{\footnotesize (1) The undischarged assumptions in the derivations $D_0$ and $D_1$ contain first-order formulas only. }\\
\hline
%Existential quantifier introduction&Existential quantifier elimination\\
\multirow{2}{*}{\AxiomC{$\phi(t/x)$}\RightLabel{\existsi}\UnaryInfC{$\exists x\phi$}\DisplayProof}
&\AxiomC{$D_0$}\noLine\UnaryInfC{$\exists x\phi$} \AxiomC{}\noLine\UnaryInfC{$[\phi]$}\noLine\UnaryInfC{$D_1$}\noLine\UnaryInfC{$\psi$} \RightLabel{\existse ~{\footnotesize(3)}}\BinaryInfC{$\psi$}\noLine\UnaryInfC{}\DisplayProof\\
%[-6pt]
%&{\footnotesize (2) $x$ does not occur freely in $\psi$ or in any formula in the undischarged assumptions of $D_1$.}\\
%\end{tabular}
%\begin{tabular}{|C{0.55\linewidth}C{0.37\linewidth}|}
\hline
%Universal quantifier introduction&Universal quantifier elimination\\
\AxiomC{$D$}\noLine\UnaryInfC{$\phi$}\RightLabel{\uqi ~{\footnotesize(4)}}\UnaryInfC{$\forall x\phi$}\DisplayProof ~\AxiomC{}\noLine\UnaryInfC{}\noLine\UnaryInfC{$\forall x\alpha$}\RightLabel{\uqe}\UnaryInfC{$\alpha(t/x)$}\DisplayProof\AxiomC{}\noLine\UnaryInfC{}\noLine\UnaryInfC{$\forall x\phi(\mathsf{y})$}\RightLabel{\uqez ~{\footnotesize(5)}}\UnaryInfC{$\phi(\mathsf{y})$} \DisplayProof
&\AxiomC{$D_0$}\noLine\UnaryInfC{$\forall x\phi$} \AxiomC{}\noLine\UnaryInfC{}\noLine\UnaryInfC{$[\phi(y/x)]$}\noLine\UnaryInfC{$D_1$}\noLine\UnaryInfC{$\psi$} \RightLabel{\uqs~{\footnotesize(6)}}\BinaryInfC{$\forall y\psi$}\noLine\UnaryInfC{}\DisplayProof~\AxiomC{}\noLine\UnaryInfC{}\noLine\UnaryInfC{$\forall x\forall y\phi$}\RightLabel{\uqexc}\UnaryInfC{$\forall y\forall x\phi$}\DisplayProof%\quad\quad\AxiomC{}\noLine\UnaryInfC{}\noLine\UnaryInfC{$\forall x\forall y\phi$}\RightLabel{\uqexc}\UnaryInfC{$\forall y\forall x\phi$}\noLine\UnaryInfC{} \DisplayProof
\\
%{\footnotesize (3) $x$ does not occur freely in any formula in the undischarged assumptions of $D$.}&{\footnotesize (4) $y$ does not occur freely in $\forall x\phi$ or any formula in the undischarged assumptions of  $D_1$.}\\
%\AxiomC{}\noLine\UnaryInfC{$\forall x\alpha$}\RightLabel{\uqe}\UnaryInfC{$\alpha(t/x)$}\noLine\UnaryInfC{} \DisplayProof\quad\AxiomC{}\noLine\UnaryInfC{$\forall x\phi(\mathsf{y})$}\RightLabel{\uqez ~{\footnotesize(5)}}\UnaryInfC{$\phi(\mathsf{y})$}\noLine\UnaryInfC{}  \DisplayProof&
%\\ 
\AxiomC{}\noLine\UnaryInfC{$\forall x\phi$}\AxiomC{}\noLine\UnaryInfC{$\forall x\psi$}\RightLabel{\uqextc}\BinaryInfC{$\forall x(\phi\wedge\psi)$}\noLine\UnaryInfC{} \DisplayProof
&\AxiomC{}\noLine\UnaryInfC{$\forall x\phi(x,\mathsf{v})\vee\psi(\mathsf{v})$}\RightLabel{\!\uqextd~{\footnotesize(7)}}\doubleLine\UnaryInfC{$\exists y\exists z\forall x((\phi\wedge y=z)\vee(\psi\wedge y\neq z))$}\noLine\UnaryInfC{}\DisplayProof\\\hline
%\multicolumn{2}{|c|}{\AxiomC{}\noLine\UnaryInfC{$\forall x\phi(x,\mathsf{v})\vee\psi(\mathsf{v})$}\RightLabel{\!\uqext}\UnaryInfC{$\exists y\exists z\forall x((\phi\wedge y=z)\vee(\psi\wedge y\neq z))$}\noLine\UnaryInfC{}\DisplayProof}\\\hline
\multicolumn{2}{|L{0.9\linewidth}|}{ \footnotesize (1) The undischarged assumptions in the derivation $D$  contain first-order formulas only.  }\\[-10pt]
\multicolumn{2}{|L{0.9\linewidth}|}{ \footnotesize (2) The undischarged assumptions in the derivations $D_0$ and $D_1$ contain first-order formulas only.  }\\[-10pt]
\multicolumn{2}{|L{0.9\linewidth}|}{ \footnotesize (3) $x$ does not occur freely in $\psi$ or in any formula in the undischarged assumptions of $D_1$. }\\[-10pt]
\multicolumn{2}{|L{0.9\linewidth}|}{ \footnotesize (4) $x$ does not occur freely in any formula in the undischarged assumptions of $D$. }\\[-10pt]
\multicolumn{2}{|L{0.9\linewidth}|}{ \footnotesize (5) $x$ is not in the sequence $\mathsf{y}$ of free variables of $\phi$. }\\[-10pt]
\multicolumn{2}{|L{0.9\linewidth}|}{ \footnotesize  (6) $y$ does not occur freely in $\forall x\phi$ or in any formula in the undischarged assumptions of  $D_1$.}\\[-10pt]
\multicolumn{2}{|L{0.9\linewidth}|}{ \footnotesize  (7) $x$ does not occur freely in $\psi(\mathsf{v})$, and $y,z$ are fresh variables.}\\
\hline
%\hline
%%\AxiomC{}\noLine\UnaryInfC{$\exists x\phi\vee\psi$}\RightLabel{($\ast$)~~\existsext}\doubleLine\UnaryInfC{$\exists x(\phi\vee\psi)$} \DisplayProof\todo{both directions are derivable (from \tensore)}
%&\AxiomC{}\noLine\UnaryInfC{$\forall x\phi\vee\psi$}\RightLabel{($\ast$)~~\uqext}\UnaryInfC{$\forall x(\phi\vee\psi)$} \DisplayProof\\
%%{\footnotesize ($\ast$) where $x$ does not occur freely in $\psi$}
%&{\footnotesize ($\ast$) where $x$ does not occur freely in $\psi$}\\\hline
\end{tabular}
\end{center}
}\label{tab:con_quant_rule}
%\label{tab:quantifiers_rule}
\end{table}%

\begin{table}[t]
\caption{Rules for inclusion atoms}
%{\normalsize
\begin{center}
\renewcommand{\arraystretch}{1.8}
\def\ScoreOverhang{0.5pt}
\def\defaultHypSeparation{\hskip .1in}
\begin{tabular}{|C{0.36\linewidth}C{0.54\linewidth}|}
\hline
%Projection and permutation of inclusion&Transitivity of inclusion&Inclusion compression\\
 \multicolumn{2}{|C{0.9\linewidth}|}{\AxiomC{}\noLine\UnaryInfC{$\mathsf{x}\mathsf{y}\mathsf{z}\subseteq \mathsf{u}\mathsf{v}\mathsf{w}$}\RightLabel{\incexc}\UnaryInfC{$\mathsf{y}\mathsf{x}\mathsf{z}\subseteq \mathsf{v}\mathsf{u}\mathsf{w}$}\noLine\UnaryInfC{}\DisplayProof
\quad\quad\quad\AxiomC{}\noLine\UnaryInfC{$\mathsf{x}\mathsf{y}\subseteq \mathsf{u}\mathsf{v}$}\RightLabel{\incctr}\UnaryInfC{$\mathsf{x}\subseteq \mathsf{u}$}\noLine\UnaryInfC{}\DisplayProof
\quad\quad\quad\AxiomC{}\noLine\UnaryInfC{$\mathsf{x}\subseteq\mathsf{y}$}\AxiomC{$\mathsf{y}\subseteq\mathsf{z}$}\RightLabel{\inctr}\BinaryInfC{$\mathsf{x}\subseteq\mathsf{z}$}\noLine\UnaryInfC{}\DisplayProof
}
\\%\hline
\AxiomC{}\noLine\UnaryInfC{}\noLine\UnaryInfC{}\noLine\UnaryInfC{$\mathsf{y}\subseteq\mathsf{x}$}\AxiomC{$\alpha(\mathsf{x}/\mathsf{z})$}\RightLabel{\incc~{\footnotesize(1)}}\BinaryInfC{$\alpha(\mathsf{y}/\mathsf{z})$}\DisplayProof&
\AxiomC{}\noLine\UnaryInfC{[$\mathsf{y}\subseteq\mathsf{x}$]} \AxiomC{}\noLine\UnaryInfC{} \AxiomC{}\noLine\UnaryInfC{[$\neg\alpha(\mathsf{y}/\mathsf{z})$]}\noLine\TrinaryInfC{$D$}
\branchDeduce
\DeduceC{$\bot$}\RightLabel{$\subseteq\!\mathsf{Exp}$~{\footnotesize(2)}}\UnaryInfC{$\alpha(\mathsf{x}/\mathsf{z})$}\noLine\UnaryInfC{}\DisplayProof\\[8pt]\hline
\AxiomC{}\noLine\UnaryInfC{}\noLine\UnaryInfC{$\mathsf{x}\subseteq \mathsf{y}$}\RightLabel{\incwe~{\footnotesize(3)}}\UnaryInfC{$\exists w(\mathsf{x}w\subseteq \mathsf{y}z)$}\noLine\UnaryInfC{}\DisplayProof
&\AxiomC{}\noLine\UnaryInfC{}\noLine\UnaryInfC{$\mathsf{x}\subseteq \mathsf{y}$}\RightLabel{\incwu~{\footnotesize(3)}}\UnaryInfC{$\forall w(\mathsf{x}z\subseteq \mathsf{y}w)$}\noLine\UnaryInfC{}\DisplayProof\\
\AxiomC{}\noLine\UnaryInfC{$\forall \mathsf{x}\phi(\mathsf{x},\mathsf{z})$}\doubleLine\RightLabel{\inci~{\footnotesize(4)}}\UnaryInfC{$\exists \mathsf{x}\forall \mathsf{y}\big(\mathsf{z}\mathsf{y}\subseteq \mathsf{z}\mathsf{x}\wedge \phi(\mathsf{x},\mathsf{z})\big)$}\DisplayProof&
%{\footnotesize Variables from $\mathsf{y}$ are fresh.}
%&{\footnotesize $z$ does not occur freely in any formula in the undischarged assumptions of the derivation $D$.}\\
\AxiomC{}\noLine\UnaryInfC{}\noLine\UnaryInfC{$\displaystyle\exists\mathsf{x}\big(\bigwedge_{i\in I}\rho^i_{\mathsf{x}}\subseteq \sigma^i_{\mathsf{x}}\,\wedge\alpha\big)\vee\phi$}\RightLabel{\incdstr~{\footnotesize(5)}}\UnaryInfC{$\displaystyle\exists\mathsf{x}\exists uv\Big(\bigwedge_{i\in I}\rho^i_{\mathsf{x}}uv\subseteq \sigma^i_{\mathsf{x}}uv\,\wedge (\alpha\leftrightarrow u=v)\wedge (\alpha\vee\phi)\Big)$}\noLine\UnaryInfC{}\DisplayProof\\
%\multicolumn{2}{|c|}{{\footnotesize $p,q$ are fresh variables.}}\\
\hline
 \multicolumn{2}{|L{0.9\linewidth}|}{\footnotesize (1) The free variables of $\alpha(\mathsf{x}/\mathsf{z})$ are among $\mathsf{x}$. }\\[-10pt]
\multicolumn{2}{|L{0.9\linewidth}|}{\footnotesize  (2) The free variables of $\alpha(\mathsf{y}/\mathsf{z})$ are among $\mathsf{y}$, and the variables in  $\mathsf{y}$ do not occur freely in any undischarged assumptions in $D$. } \\[-10pt]
 \multicolumn{2}{|L{0.9\linewidth}|}{\footnotesize (3) $w$ is not among $\mathsf{xy}z$.} \\[-10pt]
 \multicolumn{2}{|L{0.9\linewidth}|}{\footnotesize  (4) $\mathsf{y}$ is a sequence of fresh variables.}\\[-10pt]
  \multicolumn{2}{|L{0.9\linewidth}|}{\footnotesize (5)  $u,v$ are fresh variables.}\\
\hline
\end{tabular}
\end{center}
%}
\label{tab:atom_rule}
\end{table}%

\begin{defn}
The  system of natural deduction  for \Inc consists of the rules for equality, connectives and quantifiers  in Table \ref{tab:con_quant_rule}, and the rules for inclusion atoms  in Table \ref{tab:atom_rule}, where $\alpha$ ranges over first-order formulas,  and the letters  $\mathsf{x},\mathsf{y},\mathsf{z},\dots$ in serif font stand for arbitrary (possibly empty) sequences of variables. The rules with double horizontal bars are invertible, i.e., they can be applied in both directions.%For the rules where the premise and conclusion are separated by double line, the premise and conclusion are provably equivalent. \todo{rephrase} %and $\mathsf{Qu}$ and $\mathsf{Qv}$ are arbitrary sequences of quantifiers.

We write $\Gamma\vdash_{\Inc}\phi$ or simply $\Gamma\vdash\phi$ if $\phi$ is derivable from the set  $\Gamma$ of formulas by applying the rules of the system of \Inc. We write simply $\phi\vdash\psi$ for $\{\phi\}\vdash\psi$. Two formulas $\phi$ and $\psi$ are said to be {\em provably equivalent}, written $\phi\dashv\vdash\psi$, if both $\phi\vdash\psi$ and $\psi\vdash\phi$. 
\end{defn}

As shown in Table \ref{tab:con_quant_rule}, restricted to first-order formulas only our system contains all rules of first-order logic (with equality). But classical rules are in general not sound for non-classical \Inc-formulas, such as the rules for negation and \uqe. As a consequence, our system does not admit {\em uniform substitution}.

%The disjunction $\vee$ is arguably the most interesting connective for logics based on team semantics. 
Recall that  the usual disjunction elimination rule (\tensore) is not sound for dependence and independence logic (see \cite{Axiom_fo_d_KV,Hannula_fo_ind_13}). In our system of \Inc the disjunction  does admit the rule \tensore under the side condition that the undischarged assumptions in the sub-derivations contain classical formulas only. This side condition however, makes, among other things, the usual derivation of the distributive law $\phi\wedge(\psi\vee\chi)/(\phi\wedge \psi)\vee(\phi\wedge\chi)$ not applicable in the system. This distributive law actually fails in \Inc in general, especially when  $\phi$ is not closed downwards. The nonstandard features of the disjunction are also reflected in the rules \uqextd and \incdstr. The invertible rule \uqextd extends the scope of a universal quantifier over a disjunction. The rule \incdstr extends over a disjunction the scope of a existential quantifier as well as that of inclusion atoms.
% that characterize how to extend the scope of a universal quantifier or a existential quantifier with inclusion atoms over a disjunction. 
 These two rules are in a sense ad hoc to the present system. Simplifying these rules is left as future work.

The universal quantifier of \Inc turns out to be a peculiar connective, especially the usual elimination rule $\forall x\phi/\phi(t/x)$  is not in general sound for arbitrary formulas. For instance, we have $\models\forall x(y\subseteq x)$, whereas $\not\models y\subseteq z$. The two weaker elimination rules \uqe and \uqez we include in the system restrict the subformula $\phi$ in the premise either to a first-order formula or a formula in which $x$ is not free. To compensate the weakness of the elimination rules we also add to our system a substitution rule \uqs, an exchange rule \uqexc, and two rules \uqextc and \uqextd for extending the scope of universal quantifier over conjunction and disjunction.  %and an exchange rule \uqexc.
%Apart from the already mentioned \uqext, our system has additionally a substitution rule \uqs.
%Our system also have the additional rules \uqs and \uqext for universal quantifier. 
In this nonstandard setting, the derivations of some  natural and simple rules for universal quantifier become not entirely trivial, as we will illustrate in the next proposition.

% \todo{comment on both \uqe, \uqez} is in general not sound for arbitrary formulas. To compensate the weakness of $\uqe$, 

%\todo{$\phi\wedge(\psi\vee\chi)\vdash(\phi\wedge \psi)\vee(\phi\wedge\chi)$ is not true!}%\todo{rules for $\forall$ are peculiar}

\begin{proposition}\label{der_rules2}
%Let $x\notin \textsf{Fv}(\psi)$ and $y\notin \textsf{Fv}(\phi)$.%, and $y,z\notin \textsf{Fv}(\phi)\cup \textsf{Fv}(\psi)$.
\begin{enumerate}[label=(\roman*)]
%\item\label{der_rules2_c6} $\forall x\phi(\mathsf{y})\vdash\phi(\mathsf{y})$ if $x\notin \mathsf{Fv}(\phi)$.
%\item\label{der_rules2_exc} $\forall x\forall y\phi\vdash\forall y\forall x\phi$
%\item\label{der_rules_c2}  If  $x\notin \textsf{Fv}(\psi)$, then $\forall x(\phi\wedge\psi)\dashv\vdash \forall x\phi\wedge\psi$.
%\item\label{der_rules2_c7} If $x\notin \textsf{Fv}(\phi)$, then $\forall x\phi(\mathsf{y})\vdash\phi(\mathsf{y})$.
%\item\label{der_rules2_c2} $\forall x\forall y\phi\vdash\forall y\phi(y/x)\wedge\forall x\phi(x/y)$.%\todo{$\forall x\forall y\phi\vdash\forall x\phi(x/y)$}
\item\label{der_rules2_c5} $\forall x\phi\vdash\forall y\phi(y/x)$ if $y\notin \mathsf{Fv}(\forall x\phi)$.
%\item\label{der_rules2_c6} $\forall x\forall y\phi\vdash\forall y\forall x\phi$. \todo{move it to somewhere after the next prop}\todo{made it a rule}
\item\label{der_rules2_c3} $\forall x(\phi\wedge\psi)\dashv\vdash\forall x\phi\wedge\forall x\psi$. %\todo{need also the other direction}
%\item\label{der_rules2_c4} $\forall x(\phi\wedge\psi)\vdash\forall x(\phi\wedge\forall y\psi(y/x))$ if $y\notin\mathsf{Fv}(\forall x\psi)$.
\end{enumerate}
\end{proposition}
\begin{proof}

\ref{der_rules2_c5}. Follows from \uqs, since $y\notin \mathsf{Fv}(\forall x\phi)$.

\ref{der_rules2_c3}. The direction $\forall x\phi\wedge\forall x\psi\vdash\forall x(\phi\wedge\psi)$ follows from \uqextc. For the other direction, since $\phi\wedge\psi\vdash\phi$, by \uqs we derive $\forall x(\phi\wedge\psi)\vdash\forall x\phi$. Similarly $\forall x(\phi\wedge\psi)\vdash\forall x\psi$. Thus $\forall x(\phi\wedge\psi)\vdash\forall x\phi\wedge\forall x\psi$ by \conji.
\end{proof}

%In the sequel, we will often apply the above Propositions  \ref{mon_lem_2} to \ref{qf_nf_syn} without any explicit references to them. 

%Table \ref{tab:atom_rule} contains  the usual rules for identity and the rules for inclusion atoms. 
The exchange rule \incexc and contraction rule \incctr  for inclusion atoms in our system, together with the  rule $\mathsf{xy}\subseteq\mathsf{uv}/\mathsf{xyy}\subseteq\mathsf{uvv}$ that we will derive in the next proposition, are clearly equivalent to the projection rule $x_1\dots x_n\subseteq y_1\dots y_n/x_{i_1}\dots x_{i_k}\subseteq y_{i_1}\dots y_{i_k}$ ($i_1,\dots,i_k\subseteq\{1,\dots,n\}$).  As we mentioned in the introduction, the projection rule, the transitivity rule \inctr and the reflexivity axiom $\mathsf{x}\subseteq\mathsf{x}$ (that we will also derive in the proposition below) form a complete axiomatization of the implication problem of inclusion dependencies in database theory (\cite{CasanovaFaginPapadimitriou84}). The inclusion compression rule \incc  is a slight generalization of a similar rule introduced in \cite{Hannula_fo_ind_13}. The inclusion expansion rule $\subseteq\!\mathsf{Exp}$ %can be viewed as a dual rule of \incc, as it 
characterizes the fact that 
\[\Gamma,\mathsf{y}\subseteq \mathsf{x}\models \alpha(\mathsf{y}/\mathsf{z})\Longrightarrow \Gamma\models \alpha(\mathsf{x}/\mathsf{z})\]
whenever variables in $\mathsf{y}$ are not free in $\Gamma$ (observe that in this case $\Gamma,\mathsf{y}\subseteq \mathsf{x},\neg \alpha(\mathsf{y}/\mathsf{z})\models\bot $ iff $\Gamma,\mathsf{y}\subseteq \mathsf{x}\models \alpha(\mathsf{y}/\mathsf{z})$). The weakening rule via existential quantifier \incwe was  introduced in \cite{HannulaKontinen16}, and the weakening rule via universal quantifier \incwu has a similar flavor. The invertiable simulation rule \inci  characterizes the fact that universal quantifiers can be simulated by existential quantifiers with the help of inclusion atoms. %The  rule \incdstr is a distributive rule over disjunction for formulas containing inclusion atoms. %Simplifying this rules is left as future work.

%Next, we illustrate our system further by deriving the reflexivity axiom $\mathsf{x}\subseteq\mathsf{x}$ and the rule $\mathsf{xy}\subseteq\mathsf{uv}/\mathsf{x}\mathsf{yy}\subseteq\mathsf{u}\mathsf{vv}$ for repetition of inclusion atoms. %As we have mentioned, these two axiom/rule together with the rules \incexc, \incctr, \inctr completely axiomatize the implication problem of inclusion dependencies in database theory (\cite{CasanovaFaginPapadimitriou84}).

% involving  inclusion atoms that will play a role in the proof of an important lemma, Lemma \ref{approx_der}, leading to the completeness theorem. %\todo{move it to somewhere}

\begin{proposition}\label{inc_atm_prop}
%Let $\mathsf{x},\mathsf{y},\mathsf{z},\mathsf{w},\mathsf{u}$ be sequences of variables of the same length.
\begin{enumerate}[label=(\roman*)]
\item\label{inc_atm_prop_c1} $\vdash \mathsf{x}\subseteq\mathsf{x}$.

%\item\label{inc_atm_prop_c6} $\vdash\exists \mathsf{x}(\mathsf{x}\subseteq \mathsf{y})$
%\item\label{inc_atm_prop_c5} $\vdash \exists \mathsf{x}(\mathsf{x}=\mathsf{y})$.
%\item $\vdash \exists \mathsf{x}(\mathsf{x}\subseteq \mathsf{y})$ \todo{not needed?}
\item\label{inc_atm_prop_c2} If $|\mathsf{x}|=|\mathsf{y}|=|\mathsf{z}|$, then $\mathsf{xy}\subseteq \mathsf{zz}\vdash \mathsf{x}=\mathsf{y}$.\footnote{$|\mathsf{x}|$ denotes the length of the sequence $\mathsf{x}$.}
\item\label{inc_atm_rep} $\mathsf{xy}\subseteq\mathsf{uv}\vdash\mathsf{x}\mathsf{yy}\subseteq\mathsf{u}\mathsf{vv}$.
\end{enumerate}
\end{proposition}
\begin{proof}
\ref{inc_atm_prop_c1} By \eqi we have that $\vdash\mathsf{x}=\mathsf{x}$, which implies $\vdash\forall z(\mathsf{x}=\mathsf{x})$ by \uqi. Now, by applying \inci we derive $\vdash \exists z\forall y(\mathsf{x}y\subseteq \mathsf{x}z\,\wedge\, \mathsf{x}=\mathsf{x})$. Thus $\vdash\exists z\forall y(\mathsf{x}\subseteq \mathsf{x})$ by \incctr. Finally we obtain $\vdash\mathsf{x}\subseteq \mathsf{x}$ by applying \existse and \uqez.

\ref{inc_atm_prop_c2} By \incc we have
$\mathsf{xy}\subseteq \mathsf{zz},\,\mathsf{z}=\mathsf{z}\vdash\mathsf{x}=\mathsf{y}$. Then, since $\vdash \mathsf{z}=\mathsf{z}$ by \eqi, we obtain
 $\mathsf{xy}\subseteq \mathsf{zz}\vdash \mathsf{x}=\mathsf{y}$.
%\begin{align*}
%\mathsf{xy}\subseteq \mathsf{zz},\,\mathsf{w}=\mathsf{z},\,\mathsf{u}=\mathsf{z}\,&\vdash \mathsf{xy}\subseteq \mathsf{wu}\,\wedge\, \mathsf{w}=\mathsf{u}\tag{\eqsub}\\
%&\vdash \mathsf{x}=\mathsf{y},\tag{\incc}
%\end{align*}
%which implies item \ref{inc_atm_prop_c2} by \existse.

%\vspace{0.5\baselineskip}

%We derive item (iv) as follows:
%\begin{align*}
%\mathsf{x}\subseteq \mathsf{y},\,\mathsf{z}\subseteq \mathsf{x}\,&\vdash\mathsf{z}\subseteq\mathsf{y}\tag{\inctr}\\
%&\vdash\exists \mathsf{w}(\mathsf{w}\mathsf{z}\subseteq\mathsf{y}\mathsf{y})\tag{\incw}\\
%&\vdash\exists \mathsf{w}(\mathsf{w}\subseteq\mathsf{y}\wedge \mathsf{w}=\mathsf{z}).\tag{\incctr, item (iii)}
%\end{align*}

\ref{inc_atm_rep} By \incwe we have that $\mathsf{xy}\subseteq\mathsf{uv}\vdash\exists \mathsf{z}(\mathsf{xyz}\subseteq\mathsf{uvv})$. Since $\mathsf{xyz}\subseteq\mathsf{uvv}\vdash \mathsf{y}=\mathsf{z}$ by item \ref{inc_atm_prop_c2}, we conclude that $\mathsf{xy}\subseteq\mathsf{uv}\vdash \exists \mathsf{z}(\mathsf{xyz}\subseteq\mathsf{uvv}\wedge \mathsf{y}=\mathsf{z})\vdash \mathsf{xyy}\subseteq\mathsf{uvv}$ by \eqsub.
\end{proof}

We now prove the Soundness Theorem of our system.

\begin{theorem}[Soundness]
Let $\Gamma\cup\{\phi\}$ be a set of \Inc-formulas. Then
\[\Gamma\vdash\phi\Longrightarrow\Gamma\models\phi.\]
\end{theorem}
\begin{proof}
%It suffices to show that all the rules in Tables \ref{tab:con_quant_rule} and \ref{tab:atom_rule} are sound. 
We  only verify the soundness of the nontrivial rules \tensore, \uqextd, \incc, $\subseteq\!\mathsf{Exp}$, \incwe, \incwu and \ince.
The soundness of \incdstr follows from (\ref{soundness_incdstr}) in the proof of the disjunction case of Lemma \ref{qf_nf_thm}. %The proof of the soundness of the other rules are left to the reader. %\todo{Did not use \incc, \inci and Pro, Trs. Are these derivable? \inci is special case of \incdstr (no); \incc is a special case of \ince (?) \ince not sound??}
%\todo{$\mathsf{x}_i\subseteq\mathsf{x}_j\vdash \exists \mathsf{z}(\mathsf{z}_j=\mathsf{x}_i\wedge \mathsf{z}\subseteq \mathsf{x}\wedge \mathsf{z}_i\subseteq \mathsf{z}_j)$. Then together with \inci, one derives $\mathsf{x}_i\subseteq\mathsf{x}_j,\alpha(\mathsf{x})\vdash \exists \mathsf{z}(\mathsf{z}_j=\mathsf{x}_i\wedge \mathsf{z}\subseteq \mathsf{x}\wedge \mathsf{z}_i\subseteq \mathsf{z}_j\wedge \alpha(\mathsf{x}))$}

\vspace{0.3\baselineskip}

\tensore: It suffices to show that $\Delta_0,\phi\models\chi$ and $\Delta_1,\psi\models\chi$ imply $\Delta_0,\Delta_1,\phi\vee\psi\models\chi$ for any two sets $\Delta_0,\Delta_1$ of first-order formulas. Suppose that $M\models_X\Delta_0\cup\Delta_1$, and also that $M\models_X\phi\vee\psi$. Then there exist $Y,Z\subseteq X$ such that $X=Y\cup Z$, $M\models_Y\phi$ and $M\models_Z\psi$. Since formulas in $ \Delta_0\cup\Delta_1$ are closed downwards, we have that $M\models_Y\Delta_0$ and $M\models_Z\Delta_1$. It then follows from the assumption that $M\models_Y\chi$ and $M\models_Z\chi$. Now, since $\chi$ is closed under unions, we conclude that $M\models_X\chi$, as required.

\vspace{0.3\baselineskip}

 \uqextd: We first show that $\forall x\phi(x,\mathsf{v})\vee\psi(\mathsf{v})\models \exists y\exists z\forall x((\phi\wedge y=z)\vee(\psi\wedge y\neq z))$, where $x\notin \textsf{Fv}(\psi)$ and $y,z\notin \textsf{Fv}(\phi)\cup \textsf{Fv}(\psi)$. Suppose $M\models_X\forall x\phi\vee\psi$, where we may w.l.o.g. assume $x,y,z\notin dom(X)$. 
 Then there exist $Y,Z\subseteq X$ such that $X=Y\cup Z$, $M\models_{Y(M/x)}\phi$ and $M\models_Z\psi$. Define functions $F:X\to\wp(M)\setminus \{\emptyset\}$ and $G:X(F/y)\to \wp(M)\setminus \{\emptyset\}$ as follows: Let $a,b$ be two distinct elements in $M$. 
 \[ F(s)=\begin{cases}
 \{a\}&\text{if }s\in Y\setminus Z,\\
  \{a,b\}&\text{if }s\in Y\cap Z,\\
   \{b\}&\text{if }s\in Z\setminus Y,\\
 \end{cases}\quad\text{and}\quad G(s)=\{a\}.\]
 Now, we split the team $X'=X(F/y)(G/z)(M/x)$ into $W=\{s\in X'\mid s(y)=a\}$ and $U=\{s\in X'\mid s(y)=b\}$. Clearly, $M\models_W y=z$ and $M\models_U y\neq z$. Observe that $W\upharpoonright (dom(X)\cup\{x\})=Y(M/x)$ and $U\upharpoonright dom(X)=Z$. Since $M\models_{Y(M/x)}\phi$ and $M\models_Z\psi$, we conclude that $M\models_{W}\phi$ and $M\models_U\psi$.
  
Conversely, suppose  $M\models_X \exists y\exists z\forall x((\phi(x,\mathsf{v})\wedge y=z)\vee(\psi(\mathsf{v})\wedge y\neq z))$. Then there are suitable functions $F,G$ and teams $Y,Z\subseteq X(F/y)(G/z)(M/x)$ such that $X(F/y)(G/z)(M/x)=Y\cup Z$, $M\models_{Y}\phi\wedge y=z$ and $M\models_Z\psi\wedge y\neq z$. Put $Y'=Y\upharpoonright dom(X)$ and $Z'=Z\upharpoonright dom(X)$. Clearly $X=Y'\cup Z'$. To show that $M\models_X\forall x\phi(x,\mathsf{v})\vee\psi(\mathsf{v})$, it then suffices to verify $M\models_{Y'}\forall x\phi$ and $M\models_{Z'}\psi$. The latter is clear, since $M\models_Z\psi$ and $x,y,z\notin \mathsf{Fv}(\psi)$. To see the former, first observe that for any $s\in Y$ and any $a\in M$, since $s(a/x)(y)=s(y)=s(z)=s(a/x)(z)$, we must have that $s(a/x)\notin Z$, or $s(a/x)\in Y$. This shows that $Y=Y(M/x)$, thus $Y\upharpoonright (dom(X)\cup\{x\})=(Y\upharpoonright dom(X))(M/x)= Y'(M/x)$. Now, since $Y\upharpoonright (dom(X)\cup\{x\})$ satisfies $\phi$, we conclude $M\models_{Y'(M/x)}\phi$, and thus $M\models_{Y'}\forall x\phi$ as required.

 \vspace{0.3\baselineskip}

\incc: Suppose $M\models_X\mathsf{y}\subseteq \mathsf{x}$ and $M\models_X\alpha(\mathsf{x}/\mathsf{z})$, where the free variables of $\alpha(\mathsf{x}/\mathsf{z})$ are among $\mathsf{x}$. We show that $M\models_X\alpha(\mathsf{y}/\mathsf{z})$. For any $s\in X$, since $M\models_X\mathsf{y}\subseteq \mathsf{x}$, there exists $s'\in X$ such that $s'(\mathsf{x})=s(\mathsf{y})$. Since $M\models_X\alpha(\mathsf{x}/\mathsf{z})$ and $\alpha$ is first-order, we have that $M\models_{s'}\alpha(\mathsf{x}/\mathsf{z})$, which implies $M\models_{s}\alpha(\mathsf{y}/\mathsf{z})$ by the locality property. Hence, we conclude that $M\models_X\alpha(\mathsf{y}/\mathsf{z})$.

\vspace{0.3\baselineskip}

$\subseteq\!\mathsf{Exp}$: Assume $\Gamma,\mathsf{y}\subseteq \mathsf{x},\neg\alpha(\mathsf{y}/\mathsf{z})\models\bot$, where the free variables of $\alpha(\mathsf{y}/\mathsf{z})$ are among $\mathsf{y}$, and the variables  in $\mathsf{y}$ do not occur freely in $\Gamma$. We show that $\Gamma\models\alpha(\mathsf{x}/\mathsf{z})$. Suppose that $M\models_X\Gamma$ for some nonempty team $X$. It suffices to show that $M\models_s \alpha(\mathsf{x}/\mathsf{z})$ for any $s\in X$. Consider the team $X(s(\mathsf{x})/\mathsf{y})$. Clearly, $M\models_{X(s(\mathsf{x})/\mathsf{y})}\mathsf{y}\subseteq\mathsf{x}$. On the other hand, since the variables  in $\mathsf{y}$ do not occur freely in $\Gamma$, by locality we obtain that $M\models_{X(s(\mathsf{x})/\mathsf{y})} \Gamma$. Now, since $X(s(\mathsf{x})/\mathsf{y})\neq\emptyset$, the assumption $\Gamma,\mathsf{y}\subseteq \mathsf{x},\neg\alpha(\mathsf{y}/\mathsf{z})\models\bot$ gives that $M\not\models_{X(s(\mathsf{x})/\mathsf{y})}\neg\alpha(\mathsf{y}/\mathsf{z})$, which by locality implies that $M\models_{s}\alpha(\mathsf{x}/\mathsf{z})$, as required.

% if $M\models_{X(s(\mathsf{x})/\mathsf{y})}\neg\alpha(\mathsf{y}/\mathsf{z})$, then the assumption gives that $X(s(\mathsf{x})/\mathsf{y})$, which by the empty team property

\vspace{0.3\baselineskip}

\incwe: It suffices to show that $\Gamma\models\mathsf{x}\subseteq\mathsf{y}$ implies $\Gamma\models\exists w(\mathsf{x}w\subseteq \mathsf{y}z)$, where $w$ is not among $\mathsf{xy}z$. %, where $w$ does not occur freely in any formula in $\Gamma$. 
Suppose $M\models_X \Gamma$. %where we may w.l.o.g. assume that $w\notin dom(X)$.
By the assumption, $M\models_{X}\mathsf{x}\subseteq\mathsf{y}$, meaning that for any $s\in X$, there exists $s'\in X$ such that $s'(\mathsf{y})=s(\mathsf{x})$. Now, to show that $M\models_X\exists w(\mathsf{x}w\subseteq \mathsf{y}z)$, we define a function $F:X\to \wp(M)\setminus\{\emptyset\}$ by taking
\(F(s)=\{s'(z)\}.\)

To show that $M\models_{X(F/w)}\mathsf{x}w\subseteq \mathsf{y}z$, 
take any $s\in X(F/w)$. Consider the witness $s_0'\in X$ for $\mathsf{x}\subseteq\mathsf{y}$ with respect to $s_0=s\upharpoonright dom(X)$. We have
\(s(\mathsf{x}w)=s_0(\mathsf{x})s(w)=s_0'(\mathsf{y})s_0'(z).\)
Hence, any extension of $s_0'$ in $X(F/w)$ witnesses $\mathsf{x}w\subseteq \mathsf{y}z$.

\vspace{0.3\baselineskip}

\incwu: It suffices to show that  $\Gamma\models\mathsf{x}\subseteq\mathsf{y}$ implies  $\Gamma\models\forall w(\mathsf{x}z\subseteq \mathsf{y}w)$, where $w$ is not among $\mathsf{xy}z$. Suppose $M\models_X\Gamma$, where we may assume w.l.o.g. that  $w\notin dom(X)$ (if not, rename the bound variable $w$ in $\forall w(\mathsf{x}z\subseteq \mathsf{y}w)$).
It then follows by locality that $M\models_{X(M/w)} \Gamma$ as well, and thus $M\models_{X(M/w)}\mathsf{x}\subseteq\mathsf{y}$ by assumption. To show  $M\models_{X(M/w)}\mathsf{x}z\subseteq \mathsf{y}w$, take an arbitrary $s\in X(M/w)$. Since $M\models_{X(M/w)}\mathsf{x}\subseteq\mathsf{y}$, there exists $s'\in X(M/w)$ such that $s'(\mathsf{y})=s(\mathsf{x})$.  Clearly, the assignment $s''=s'(s(z)/w)$ belongs to the team $X(M/w)$, and $s''(\mathsf{y}w)=s'(\mathsf{y})s(z)=s(\mathsf{x}z)$, as required.

\vspace{0.3\baselineskip}

\inci: For the top to bottom direction, suppose $M\models_X\forall \mathsf{x}\phi(\mathsf{x},\mathsf{z})$. We show that $M\models_X \exists \mathsf{x}\forall \mathsf{y}\big(\mathsf{z}\mathsf{y}\subseteq \mathsf{z}\mathsf{x}\wedge \phi(\mathsf{x},\mathsf{z})\big)$, where variables from $\mathsf{y}$ are fresh.  Let $\mathsf{x}=\langle x_1,\dots,x_n\rangle$. Define a sequence $\mathsf{F}=\langle F_1,\dots,F_n\rangle$ of functions for $\exists\mathsf{x}$ by taking $F_i(s)=M$ for each $F_i$ from $\mathsf{F}$, namely, we let $X(\mathsf{F}/\mathsf{x})=X(M/\mathsf{x})$.
It suffices to show that $M\models_{X(\mathsf{F}/\mathsf{x})(M/\mathsf{y})}\mathsf{z}\mathsf{y}\subseteq \mathsf{z}\mathsf{x}\wedge \phi(\mathsf{x},\mathsf{z})$, or $M\models_{X(M/\mathsf{x})(M/\mathsf{y})}\mathsf{z}\mathsf{y}\subseteq \mathsf{z}\mathsf{x}\wedge \phi(\mathsf{x},\mathsf{z})$

By assumption,  $M\models_{X(M/\mathsf{x})}\phi(\mathsf{x},\mathsf{z})$, which implies $M\models_{X(M/\mathsf{x})(M/\mathsf{y})}\phi(\mathsf{x},\mathsf{z})$. To show that $\mathsf{z}\mathsf{y}\subseteq \mathsf{z}\mathsf{x}$ is also satisfied by $X(M/\mathsf{x})(M/\mathsf{y})$, take any $s\in X(M/\mathsf{x})(M/\mathsf{y})$. Observe that the function $s'=s(s(\mathsf{y})/\mathsf{x})$ belongs to the team $X(M/\mathsf{x})(M/\mathsf{y})$, and we have that $s'(\mathsf{z}\mathsf{x})=s(\mathsf{z})s(\mathsf{y})$, as required.

\vspace{0.2\baselineskip}

For the bottom to top direction, suppose $M\models_X \exists \mathsf{x}\forall \mathsf{y}\big(\mathsf{z}\mathsf{y}\subseteq \mathsf{z}\mathsf{x}\wedge \phi(\mathsf{x},\mathsf{z})\big)$, where no variable from $\mathsf{y}$ are free in $\phi$, and we may assume w.l.o.g.  that $\mathsf{dom}(X)$ consists of all variables from $\mathsf{z}$. Then there are suitable sequence $\mathsf{F}$ of functions for $\exists\mathsf{x}$ such that $M\models_{X(\mathsf{F}/\mathsf{x})(M/\mathsf{y})}\mathsf{z}\mathsf{y}\subseteq \mathsf{z}\mathsf{x}\wedge \phi(\mathsf{x},\mathsf{z})$. We show that $M\models_X\forall \mathsf{x}\phi(\mathsf{x},\mathsf{z})$, or $M\models_{X(M/\mathsf{x})}\phi(\mathsf{x},\mathsf{z})$, which, by locality, is further  reduced to showing that $X(\mathsf{F}/\mathsf{x})=X(M/\mathsf{x})$.
 
 For any $s\in X(M/\mathsf{x})$, consider an arbitrary assignment  $t\in X(\mathsf{F}/\mathsf{x})(M/\mathsf{y})$ satisfying
$t(\mathsf{z})=s(\mathsf{z})$ and $t(\mathsf{y})=s(\mathsf{x})$.
Since $M\models_{X(\mathsf{F}/\mathsf{x})(M/\mathsf{y})}\mathsf{z}\mathsf{y}\subseteq \mathsf{z}\mathsf{x}$, there exists $t'\in X(\mathsf{F}/\mathsf{x})(M/\mathsf{y})$ such that 
\(t'(\mathsf{z}\mathsf{x})=t(\mathsf{z}\mathsf{y})=s(\mathsf{z})s(\mathsf{x}),\)
meaning that $s=t'\upharpoonright \mathsf{dom}(X)\cup \{x_1,\dots,x_n\}\in X(\mathsf{F}/\mathsf{x})$. Thus, $X(M/\mathsf{x})\subseteq X(\mathsf{F}/\mathsf{x})$,  thereby $X(M/\mathsf{x})= X(\mathsf{F}/\mathsf{x})$.
\end{proof}

%\begin{proposition}\label{proof_system_equi_fo}
%Let $\Delta\cup\{\alpha\}$ be a set of first-order formulas. Then
%\[\Delta\vdash\alpha\iff\Delta\vdash_{\FO}\alpha.\]
%\end{proposition}
%\todo{added}

We will prove the completeness theorem of our system in the next section. An important lemma for this proof is that every formula provably implies its normal form (\ref{inc_nf}). To prove this lemma we  first prove a few useful propositions.
The next three propositions concern the standard properties of quantifications as well as the monotonicity of the entailment relation in \Inc. In the sequel, we will often apply Propositions  \ref{der_rules} and \ref{mon_lem_2} without  explicit reference to them. 

%\todo{stress that universal quantifiers are special, many obvious rules have nontrivial derivations}
%Universal quantifier is peculiar for our logic. Some of the usual and natural rules for universal quantifier have nontrivial derivations, as we will show in the next proposition.

\begin{proposition}\label{der_rules}
Let $\alpha$ be a first-order formula, and $x\notin \textsf{Fv}(\psi)$. %\todo{we will not explicitly cite this prop. when apply it}
\begin{enumerate}[label=(\roman*)]
\item $\neg\forall x\alpha\dashv\vdash\exists x\neg\alpha$ and $\neg\exists x\alpha\dashv\vdash \forall x\neg\alpha$.
\item\label{der_rules_c2} $\forall x\phi\wedge\psi\dashv\vdash \forall x(\phi\wedge\psi)$.
\item\label{der_rules_c1} $\exists x\phi\wedge\psi\dashv\vdash \exists x(\phi\wedge\psi)$.
\item $\exists x\phi\vee\psi\dashv\vdash\exists x(\phi\vee\psi)$.
%\item If $x\notin \textsf{Fv}(\psi)$, then $\exists x(\phi\vee\psi)\vdash\exists x\phi\vee\psi$.
\end{enumerate}
\end{proposition}
\begin{proof}
Since our system behaves exactly like first-order logic when restricted to first-order formulas only, item (i) can be proved as usual. 
Items (iii) and (iv) are proved also as usual.
We only derive item (ii). For the right to left direction, we have by Proposition \ref{der_rules2}\ref{der_rules2_c3} that $\forall x(\phi\wedge\psi)\vdash\forall x\phi\wedge\forall x\psi$. Since $x\notin \textsf{Fv}(\psi)$, $\forall x\psi\vdash\psi$ by \uqez. Thus $\forall x(\phi\wedge\psi)\vdash\forall x\phi\wedge\psi$. For the other direction, since $\phi,\psi\vdash\phi\wedge\psi$ and $x\notin \textsf{Fv}(\psi)$, we derive by applying \uqs that $\forall x\phi,\psi\vdash\forall x(\phi\wedge\psi)$, thus $\forall x\phi\wedge\psi\vdash \forall x(\phi\wedge\psi)$.
% of item (i), we have the following derivation:
%%\def\ScoreOverhang{0.5pt}
%%\def\defaultHypSeparation{\hskip .1in}
%\begin{center}
%\AxiomC{$\forall x\phi\wedge\psi$}
%\RightLabel{\conje}\UnaryInfC{$\forall x\phi$}
%\AxiomC{$\forall x\phi\wedge\psi$}
%\RightLabel{\conje}\UnaryInfC{$\psi$}\AxiomC{[$\phi$]}
%%
%%\BinaryInfC{$(\mathsf{x}\subseteq\mathsf{y})\wedge(\mathsf{x}=\mathsf{x})\wedge (\mathsf{y}=\mathsf{y})$}
%\RightLabel{\conji}\BinaryInfC{$\phi\wedge\psi$}
%\RightLabel{\uqs ($\because$ $x\notin \textsf{Fv}(\psi)$)}\BinaryInfC{$\forall x(\phi\wedge\psi)$}\DisplayProof
%\end{center}
%\quad
%\end{center}
%\begin{center}
%\begin{center}
%\AxiomC{$\forall x(\phi\wedge\psi)$}
%\RightLabel{\conje, \uqs}\UnaryInfC{$\forall x\phi$}
%\AxiomC{$\forall x(\phi\wedge\psi)$}
%%
%%\BinaryInfC{$(\mathsf{x}\subseteq\mathsf{y})\wedge(\mathsf{x}=\mathsf{x})\wedge (\mathsf{y}=\mathsf{y})$}
%\RightLabel{\conje, \uqs}\UnaryInfC{$\forall x\psi$}
%\RightLabel{\uqez ($\because$ $x\notin \textsf{Fv}(\psi)$)}\UnaryInfC{$\psi$}
%\RightLabel{\conji}\BinaryInfC{$\forall x\phi\wedge\psi$}\DisplayProof
%\end{center}
\end{proof}

We write $\phi(\theta)$ to indicate that $\phi$ is a formula with an occurrence of $\theta$ as a subformula, and write $\phi[\theta'/\theta]$ for the formula obtained from $\phi$ by replacing the occurrence of $\theta$ by $\theta'$.

\begin{proposition}\label{mon_lem_2}
If $\theta\dashv\vdash\theta'$, then $\phi(\theta)\dashv\vdash\phi[\theta'/\theta]$. Moreover, if the occurrence of $\theta$ in $\phi(\theta)$ is not in the scope of a negation, then $\theta\vdash\theta'$ implies $\phi(\theta)\vdash\phi[\theta'/\theta]$.
%Let $\phi(\theta)$ be a formula in which the occurrence of $\theta$ is not in the scope of a negation. If $\theta\vdash\theta'$, then $\phi(\theta)\vdash\phi[\theta'/\theta]$.
\end{proposition}
\begin{proof}
A routine inductive proof. Apply \uqs in the case $\phi=\forall x\psi$. %\todo{This is not sufficient for the next prop, because of negation. There is a way out: be more careful, or make \uqext a invertible rule (still not quite right) and change this prop to replacement: two directions}\todo{still not quite right: \uqext makes a first-order formula more complex (yet still first-order). No, no, the other direction of  \uqext  is derivable: by $\wedge E$ and $\forall E_0$}
\end{proof}

\begin{proposition}\label{qf_nf_syn}
Let $\phi$ be a formula and $\mathsf{Q}\mathsf{x}\theta$ the semantically equivalent formula in prenex normal form as given in Theorem \ref{nf_thm_qf}, where $\mathsf{Qx}=Q^1x_1\cdots Q^nx_n$ ($Q^i\in \{\forall,\exists\}$)  is a sequence of quantifiers and $\theta$ is a quantifier free formula.  Then
$\phi\dashv\vdash\mathsf{Q}\mathsf{x}\theta$.
\end{proposition}
\begin{proof}
Repeatedly apply Propositions \ref{der_rules}, \ref{mon_lem_2} and \uqextd (c.f. the proof of Theorem \ref{nf_thm_qf}). %\todo{simplify \uqext if possible!}
\end{proof}

The next technical proposition shows, as a generalization of the rule \inci,  that  universal quantifiers in a more general context can also be simulated using existential quantifiers and inclusion atoms.

% concerns simulating universal quantifiers using existential quantifiers and inclusion atoms.

%Next, we derive some clauses involving inclusion atoms, some of which are crucial technical lemmas for the proof of the completeness theorem that we will give in the next section.

\begin{proposition}\label{inc_atm_prop_c4} 
$\forall \mathsf{x}\mathsf{Qu}\phi(\mathsf{u},\mathsf{x},\mathsf{z})\dashv\vdash \exists \mathsf{x}\mathsf{Qu}\forall \mathsf{y}\big(\mathsf{z}\mathsf{y}\subseteq \mathsf{z}\mathsf{x}\wedge \phi(\mathsf{u},\mathsf{x},\mathsf{z})\big)$.%, where $\mathsf{y}$ is a sequence of fresh variables.%, where $\mathsf{Qx}=Q^1x_1\cdots Q^nx_n$ (each $Q^i\in\{\forall,\exists\}$) is any sequence   of quantifiers.
\end{proposition}
\begin{proof}
We derive the proposition as follows:
\begin{align*}
\forall \mathsf{x}\mathsf{Qu}\phi(\mathsf{u},\mathsf{x},\mathsf{z})&\dashv\vdash\exists \mathsf{x}\forall \mathsf{y}\big(\mathsf{z}\mathsf{y}\subseteq \mathsf{z}\mathsf{x}\wedge \mathsf{Qu}\phi(\mathsf{u},\mathsf{x},\mathsf{z})\big)\tag{\inci}\\
&\dashv\vdash \exists \mathsf{x}\big(\forall \mathsf{y}(\mathsf{z}\mathsf{y}\subseteq \mathsf{z}\mathsf{x})\wedge \mathsf{Qu}\phi(\mathsf{u},\mathsf{x},\mathsf{z})\big)\\
&\dashv\vdash \exists \mathsf{x} \mathsf{Qu}\big(\forall \mathsf{y}(\mathsf{z}\mathsf{y}\subseteq \mathsf{z}\mathsf{x})\wedge\phi(\mathsf{u},\mathsf{x},\mathsf{z})\big)\\
&\dashv\vdash \exists \mathsf{x} \mathsf{Qu}\forall \mathsf{y}\big(\mathsf{z}\mathsf{y}\subseteq \mathsf{z}\mathsf{x}\wedge\phi(\mathsf{u},\mathsf{x},\mathsf{z})\big).
\end{align*}
%and 
%\begin{align*}
%\exists x \mathsf{Qu}\forall y\big(\mathsf{z}y\subseteq \mathsf{z}x\wedge\phi(\mathsf{u},x,\mathsf{z})\big)&\vdash \exists x \mathsf{Qu}\big(\forall y(\mathsf{z}y\subseteq \mathsf{z}x)\wedge\phi(\mathsf{u},x,\mathsf{z})\big)\\
%&\vdash \exists x\big(\forall y(\mathsf{z}y\subseteq \mathsf{z}x)\wedge \mathsf{Qu}\phi(\mathsf{u},x,\mathsf{z})\big)\\
%&\vdash\exists x\forall y\big(\mathsf{z}y\subseteq \mathsf{z}x\wedge \mathsf{Qu}\phi(\mathsf{u},x,\mathsf{z})\big)\\
%&\vdash \forall x\mathsf{Qu}\phi(\mathsf{u},x,\mathsf{z}).\tag{\ince}
%\end{align*}
\end{proof}

\begin{lemma}\label{nf_der}
For any \Inc-formula $\phi$, we have that $\phi\vdash\phi'$, where $\phi'$ is the semantically equivalent formula in normal form (\ref{inc_nf_cp}) as given in Theorem \ref{inc_nf_thm}.
\end{lemma}
\begin{proof}
We follow a similar argument to that of the semantic proof of Theorem \ref{inc_nf_thm}. First, by Proposition \ref{qf_nf_syn}, we obtain $\phi\vdash\mathsf{Qx}\theta$,
 where $\mathsf{Qx}\theta$ is the semantically equivalent formula of $\phi$ as given in Theorem \ref{nf_thm_qf}  with $\mathsf{Qx}=Q^1x_1\cdots Q^nx_n$ ($Q^i\in \{\forall,\exists\}$) a sequence of quantifiers and $\theta$ a quantifier free formula. 

%\[Q^1x_1\dots Q^nx_n\theta,\]
%where $Q^i\in \{\exists, \forall\}$ and $\theta$ is a quantifier free formula.

If we can show that  $\theta\vdash\exists \mathsf{w}\theta'$ for some formula 
\(\theta'=\bigwedge_{i\in I}\mathsf{u_i}\subseteq \mathsf{v_i}\,\wedge\alpha(\mathsf{w},\mathsf{x},\mathsf{z})\)
 as given in Lemma \ref{qf_nf_thm},  we may obtain $\phi\vdash \mathsf{Qx}\exists\mathsf{w}\theta'$ by Proposition \ref{mon_lem_2}.
%\[\phi\vdash Q^1x_1\dots Q^nx_n\exists \mathsf{w}\Big(\bigwedge_{i\in I}(\mathsf{u_i}\subseteq \mathsf{v_i})\wedge\alpha(\mathsf{w},\mathsf{x},\mathsf{z})\Big).\]
Next, we derive
%where $\mathsf{y}=\langle y_j\mid Q^j=\forall,~1\leq j\leq n\rangle$. Next, since each $y_j$ occurs only once in one inclusion atom,  we obtain that
\begin{align*}
\mathsf{Qx}\exists\mathsf{w}\theta'&\vdash\exists \mathsf{x}\exists \mathsf{w}\forall \mathsf{y}\Big(\mathop{\bigwedge_{1\leq j\leq n}}_{Q^j=\forall}\mathsf{z}x_1\dots x_{j-1}y_j\subseteq \mathsf{z}x_1\dots x_{j-1}x_j\,\wedge\theta'(\mathsf{w},\mathsf{x},\mathsf{z})\Big)\tag{Proposition \ref{inc_atm_prop_c4}}\\
&\quad\quad\text{where }\mathsf{y}=\langle y_j\mid Q^j=\forall,~1\leq j\leq n\rangle\\
&\vdash\exists \mathsf{w}\exists \mathsf{x}\Big(\mathop{\bigwedge_{1\leq j\leq n}}_{Q^j=\forall}\forall y_j(\mathsf{z}x_1\dots x_{j-1}y_j\subseteq \mathsf{z}x_1\dots x_{j-1}x_j)\,\wedge\theta'(\mathsf{w},\mathsf{x},\mathsf{z})\Big)\tag{Proposition \ref{der_rules}\ref{der_rules_c2}}\\
&\vdash\exists \mathsf{w}\exists \mathsf{x}\Big(\mathop{\bigwedge_{1\leq j\leq n}}_{Q^j=\forall}\forall y(\mathsf{z}x_1\dots x_{j-1}y\subseteq \mathsf{z}x_1\dots x_{j-1}x_j)\,\wedge\theta'(\mathsf{w},\mathsf{x},\mathsf{z})\Big)\tag{Proposition   \ref{der_rules2}\ref{der_rules2_c5}}\\
&\vdash\exists \mathsf{w}\exists \mathsf{x}\forall y\Big(\mathop{\bigwedge_{1\leq j\leq n}}_{Q^j=\forall}\mathsf{z}x_1\dots x_{j-1}y\subseteq \mathsf{z}x_1\dots x_{j-1}x_j\,\wedge\theta'(\mathsf{w},\mathsf{x},\mathsf{z})\Big).\tag{Proposition \ref{der_rules2}\ref{der_rules2_c3}}
\end{align*}
Putting all these together, we will complete the proof.

Now, we show that $\theta\vdash\exists \mathsf{w}\theta'$  by induction on  $\theta$.
The case $\theta$ is a  first-order formula (including the case $\theta=\neg\alpha$) is trivial. If $\theta=\mathsf{x}\subseteq \mathsf{y}$, we have that
\(\mathsf{x}\subseteq \mathsf{y}\vdash \exists \mathsf{w}\mathsf{u}\big(\mathsf{w}\subseteq \mathsf{u}\,\wedge \,\mathsf{w}= \mathsf{x}\,\wedge \,\mathsf{u}= \mathsf{y}\big).\) 
Indeed, we first derive that $\vdash\mathsf{x}=\mathsf{x}\wedge\mathsf{y}=\mathsf{y}\vdash\exists \mathsf{w}\exists \mathsf{u}(\mathsf{w}=\mathsf{x}\wedge \mathsf{u}=\mathsf{y})$ by \eqi and \existsi. Then, by \eqsub we derive that 
\(\mathsf{x}\subseteq \mathsf{y}\vdash\exists \mathsf{w}\mathsf{u}\big(\mathsf{x}\subseteq \mathsf{y}\,\wedge\,\mathsf{w}= \mathsf{x}\,\wedge \,\mathsf{u}= \mathsf{y}\big)\vdash\exists \mathsf{w}\mathsf{u}\big(\mathsf{w}\subseteq \mathsf{u}\,\wedge \,\mathsf{w}= \mathsf{x}\,\wedge \,\mathsf{u}= \mathsf{y}\big).\)
%\begin{align*}
%\mathsf{x}\subseteq \mathsf{y}\,&\vdash\exists \mathsf{w}\mathsf{u}\big(\mathsf{x}\subseteq \mathsf{y}\,\wedge\,\mathsf{w}= \mathsf{x}\,\wedge \,\mathsf{u}= \mathsf{y}\big)\\
%&\vdash\exists \mathsf{w}\mathsf{u}\big(\mathsf{w}\subseteq \mathsf{u}\,\wedge \,\mathsf{w}= \mathsf{x}\,\wedge \,\mathsf{u}= \mathsf{y}\big).\tag{\eqsub}
%\end{align*}

Assume that 
\(\theta_0\vdash\exists \mathsf{w_0}(\iota_0(\mathsf{w}_0)\wedge\alpha_0(\mathsf{w}_0,\mathsf{x}))\text{ and }\theta_1\vdash\exists \mathsf{w_1}(\iota_1(\mathsf{w}_1)\wedge\alpha_1(\mathsf{w}_1,\mathsf{y})),\)
where $\alpha_0,\alpha_1$ are first-order and quantifier-free, the sequences
$\mathsf{w}_0$ and $\mathsf{w}_1$ do not have variables in common, and $\iota_0$ and $\iota_1$ are as in (\ref{qf_nf_thm_eq1}) in the proof of Lemma \ref{qf_nf_thm}.
%\[\delta(\mathsf{w}_0)=\bigwedge_{i\in I}\mathsf{u}_i\subseteq \mathsf{v}_i~~\text{ and }~~\eta(\mathsf{w}_1)=\bigwedge_{j\in J}\mathsf{u}_j\subseteq \mathsf{v}_j.\]

If $\theta=\theta_0\wedge\theta_1$, then we derive that
\(\theta_0\wedge\theta_1\vdash\exists \mathsf{w_0}(\iota_0\wedge\alpha_0)\wedge\exists \mathsf{w_1}(\iota_1\wedge\alpha_1)\vdash \exists \mathsf{w_0}\exists \mathsf{w_1}(\iota_0\wedge\iota_1\wedge\alpha_0\wedge\alpha_1)\)  by Proposition \ref{der_rules}\ref{der_rules_c1}.

If $\theta=\theta_0\vee\theta_1$,  let $\psi$ be the formula (\ref{qf_nf_thm_eq2}) as in the proof of the disjunction case of Lemma \ref{qf_nf_thm}. We  derive $\theta\vdash\psi$ by following the semantic argument as in Lemma \ref{qf_nf_thm}, in which we apply the rule \incdstr  in the crucial steps.%\todo{combined with $\exists$ extension over $\vee$}
\end{proof}

We end this section by proving some facts concerning the weak classical negation \cn in the context of \Inc. This connective was introduced in \cite{Yang_neg18}, and a trick using \cn was developed in the paper to generalize the proof of the completeness theorem of dependence logic given in \cite{Axiom_fo_d_KV}. We will also apply this  trick  to prove the completeness theorem for our system  in the next section. Recall that the team semantics of \cn is defined as
\begin{itemize}
\item $M\models_X\cn\phi$ \ \ iff \ \ $X=\emptyset$ or $M\not\models_X\phi$. 
\end{itemize}
The weak classical negations $\cn\phi$ of \Inc-formulas $\phi$ are not in general expressible in \Inc (because positive greatest fixed point logic, being expressively equivalent to \Inc, is not closed under classical negation).
%existential second-order logic, being more expressive than \Inc, is not closed under classical negation)\todo{not accurate}. 
Nevertheless, the weak classical negations $\cn\alpha$ of first-order formulas $\alpha$ are expressible (uniformly) in \Inc,:

 %as the next show (or see Proposition 4.4 in \cite{Yang_neg18}) that
\begin{fact}\label{cnalpha_df}
If $\alpha(\mathsf{x})$ is a first-order formula, then $\cn\alpha(\mathsf{x})\equiv\exists \mathsf{y}(\mathsf{y}\subseteq\mathsf{x}\wedge\neg\alpha(\mathsf{y}/\mathsf{x}))$, where $\mathsf{y}$ is a sequence of fresh variables.
\end{fact}
\begin{proof}
Since $\alpha$ is flat, for any nonempty team $X$, $M\not\models_X\alpha(\mathsf{x})$, iff $M\models_s\neg\alpha(\mathsf{x})$ for some $s\in X$, iff $M\models_X\exists \mathsf{y}(\mathsf{y}\subseteq\mathsf{x}\wedge\neg\alpha(\mathsf{y}/\mathsf{x}))$.
\end{proof}
Stipulating the string $\cn\alpha(\mathsf{x})$ as a shorthand for the formula $\exists \mathsf{y}(\mathsf{y}\subseteq\mathsf{x}\wedge\neg\alpha(\mathsf{y}))$ of \Inc, we show next  that the {\em reductio ad absurdum} ($\mathsf{RAA}$) rule for \cn with respect to first-order formulas $\alpha$, i.e., the rule %\todo{It was not an effective rule, but now effective}
\[\AxiomC{[$\cn\alpha$]}\noLine\UnaryInfC{$\vdots$}\noLine\UnaryInfC{$\bot$}\RightLabel{$\mathsf{RAA}_{\cn}$}\UnaryInfC{$\alpha$} \DisplayProof\]
 is derivable in our system from the rule $\subseteq\!\mathsf{Exp}$.

\begin{lemma}\label{neg_trans}
If $\Gamma,\cn\alpha\vdash\bot$,  then $\Gamma\vdash\alpha$. %\todo{true only when $\Gamma$ is a set of sentences}\todo{Not right!!} \todo{can add: \AxiomC{}\noLine\UnaryInfC{}\noLine\UnaryInfC{}\noLine\UnaryInfC{[$\mathsf{y}\subseteq\mathsf{x}$]}\AxiomC{[$\neg\alpha(\mathsf{y})$]}\noLine\BinaryInfC{$D$}\noLine\UnaryInfC{$\bot$}\RightLabel{$\incc^\neg$}\UnaryInfC{$\alpha(\mathsf{x}/\mathsf{y})$}\DisplayProof where $y$ is not free in any undischarged assumptions in $D$.}
\end{lemma}
\begin{proof}
Let $\alpha=\alpha(\mathsf{x})$ and $\cn\alpha(\mathsf{x})=\exists \mathsf{y}(\mathsf{y}\subseteq\mathsf{x}\wedge\neg\alpha(\mathsf{y}/\mathsf{x}))$, where $\mathsf{y}$ is a sequence of fresh variables. Suppose $\Gamma,\cn\alpha\vdash\bot$. By $\subseteq\!\mathsf{Exp}$, it suffices to show that $\Gamma,\mathsf{y}\subseteq \mathsf{x},\neg\alpha(\mathsf{y}/\mathsf{x})\vdash\bot$. But this follows easily from \existsi and the assumption $\Gamma,\cn\alpha\vdash\bot$.
%Suppose $\Gamma,\cn\alpha\vdash\bot$. In view of \textsf{RAA} (for $\neg$) it suffices to show $\Gamma,\neg\alpha\vdash\bot$. 
%Let $\mathsf{x}$ list all free variables in $\alpha$. By Proposition \ref{inc_atm_prop}\ref{inc_atm_prop_c1}, we have $\vdash\mathsf{x}\subseteq\mathsf{x}$, which implies
%$\vdash\exists \mathsf{y}(\mathsf{y}\subseteq\mathsf{x})$ by \existsi. Then, by \incc we derive that 
%\[\neg\alpha(\mathsf{x})\vdash\exists \mathsf{y}(\mathsf{y}\subseteq\mathsf{x}\wedge\neg\alpha(\mathsf{x}))\vdash\exists \mathsf{y}(\mathsf{y}\subseteq\mathsf{x}\wedge\neg\alpha(\mathsf{y}/\mathsf{x}))\vdash\cn\alpha(\mathsf{x}).\]
%Now, since $\Gamma,\cn\alpha\vdash\bot$, we conclude that $\Gamma,\neg\alpha\vdash\bot$. 
%We have the following derivation:
%\begin{align*}
%\Gamma,\neg\alpha(\mathsf{x})\vdash\,& \exists \mathsf{y}(\mathsf{y}\subseteq\mathsf{x})\tag{Prop. \ref{inc_atm_prop}\ref{inc_atm_prop_c6}}\\
%\vdash& \exists \mathsf{y}(\mathsf{y}\subseteq\mathsf{x}\wedge\neg\alpha(\mathsf{x}))\tag{Prop. \ref{der_rules}\ref{der_rules_c1}}\\
%\vdash& \exists \mathsf{y}(\mathsf{y}\subseteq\mathsf{x}\wedge\neg\alpha(\mathsf{y}))(\text{ i.e., }\cn\alpha(\mathsf{x}))\tag{\incc}\\
%\vdash& \bot\tag{assumption}.
%\end{align*}
\end{proof}

\section{The completeness theorem}

In this section, we prove the completeness theorem for our system of \Inc with respect to first-order consequences. To be precise, we  prove that
\begin{equation}\label{fo_comp_eq}
\Gamma\vdash\alpha\iff\Gamma\models\alpha
\end{equation}
holds whenever $\Gamma$ is a set of \Inc-formulas, and $\alpha$ is a first-order formula. As sketched in Section \ref{sec:pre}, our proof combines the technique introduced in \cite{Axiom_fo_d_KV} and a trick developed in \cite{Yang_neg18} using the weak classical negation \cn and the $\mathsf{RAA}$ rule for \cn.
%that involves the weak classical negation \cn and the $\mathsf{RAA}$ rule for \cn. 
 The former treats the case when the set $\Gamma\cup\{\alpha\}$ of formulas in  \eqref{fo_comp_eq}  are sentences (with no free variables) only, while the  trick of the latter allows us to handle (open) formulas as well. 
 %However, the crucial rule used in the trick, the $\mathsf{RAA}$ rule for \cn, is not in general effective
 % this trick involves the $\mathsf{RAA}$ rule for \cn
Since the weak classical negation $\cn\alpha$ of first-order formulas $\alpha$ are definable uniformly in \Inc (Fact \ref{cnalpha_df}), and the $\mathsf{RAA}$ rule for \cn is derivable in our system of \Inc (Lemma \ref{neg_trans}), we will be able to apply the trick of \cite{Yang_neg18} in a smoother manner than in the systems of dependence and independence logic  \cite{Yang_neg18} (in which the $\mathsf{RAA}$ rule for \cn was added in an ad hoc and non-effective manner).

% the ad hoc treatment of the $\mathsf{RAA}$ rule for \cn in the dependence and independence logic cases \cite{Yang_neg18}.
%
%the trick of \cite{Yang_neg18} can be applied smoothly for \Inc (in contrast to the cases for dependence and independence logic, which require a more complex treatment for the $\mathsf{RAA}$ rule for \cn, as done in \cite{Yang_neg18}). %in which the $\mathsf{RAA}$ rule for \cn has to be added manually, as is done in  \cite{Yang_neg18}). \todo{rephrase}

%We now  prove the remaining lemmas needed for the argument of \cite{Axiom_fo_d_KV}.
%Note that the \textsf{RAA} rule for \cn played a crucial role in \cite{Yang_neg18} \todo{remarks: do not need to add new rule for \cn}

We have prepared in the previous sections most relevant lemmas for the argument in \cite{Axiom_fo_d_KV} concerning the normal form of \Inc-formulas (especially Lemma \ref{nf_der}), the game expression and its approximations. Another important lemma for the completeness theorem is that any \Inc-formula $\phi$  implies every approximation $\Phi_n$ of its game expression (as introduced in Section \ref{sec:game_exp}). %\todo{improve}
\begin{lemma}\label{approx_der}
For any \Inc-sentence $\phi$, we have that $\phi\vdash\Phi_n$ for every $n<\omega$.
\end{lemma}

In order to prove the above lemma, we first need to prove a number of technical propositions and lemmas.

\begin{proposition}\label{inc_atm_prop_c0}
Let $\rho:\mathsf{Var}^n\to \mathsf{Var}^k,\sigma:\mathsf{Var}^n\to \mathsf{Var}^m$ be functions. Then
\[\rho_{\mathsf{x}}\mathsf{z}\subseteq \sigma_\mathsf{x},\mathsf{x}_0\mathsf{y}_0\mathsf{z}_0\subseteq\mathsf{x}\mathsf{y}\mathsf{z}\vdash \exists \mathsf{x}_1\mathsf{y}_1(\mathsf{x}_1\mathsf{y}_1\subseteq \mathsf{x}\mathsf{y}\,\wedge\, \rho_{\mathsf{x}_0}\mathsf{z}_0=\sigma_{\mathsf{x}_1}),\] 
where $|\mathsf{x}|=|\mathsf{x}_0|=|\mathsf{x}_1|$, $|\mathsf{y}|=|\mathsf{y}_0|=|\mathsf{y}_1|$ and $|\mathsf{z}|=|\mathsf{z}_0|$.
In particular, when $\mathsf{z}$ and $\mathsf{z}_0$ are the empty sequence we have that   $\rho_{\mathsf{x}}\subseteq \sigma_\mathsf{x},\mathsf{x}_0\mathsf{y}_0\subseteq\mathsf{x}\mathsf{y}\vdash \exists \mathsf{x}_1\mathsf{y}_1(\mathsf{x}_1\mathsf{y}_1\subseteq \mathsf{x}\mathsf{y}\,\wedge\, \rho_{\mathsf{x}_0}=\sigma_{\mathsf{x}_1})$.
\end{proposition}
\begin{proof}
Assume that $\mathsf{p}(\mathsf{x})=\sigma_{\mathsf{x}}\tau_{\mathsf{x}}$ for some permutation $\mathsf{p}$ of the sequence $\mathsf{x}$. Then we have\vspace{-4pt}
\begin{align*}
&\rho_{\mathsf{x}}\mathsf{z}\subseteq \sigma_{\mathsf{x}},\mathsf{x}_{0}\mathsf{y}_0\mathsf{z}_0\subseteq \mathsf{x}\mathsf{y}\mathsf{z}\\
\vdash&\rho_{\mathsf{x}}\mathsf{z}\subseteq \sigma_{\mathsf{x}}\wedge\rho_{\mathsf{x}_0}\mathsf{z}_0\subseteq \rho_{\mathsf{x}}\mathsf{z}\tag{\incctr, \incexc}\\
\vdash&\rho_{\mathsf{x}_0}\mathsf{z}_0\subseteq \sigma_{\mathsf{x}}\tag{\inctr}\\
\vdash& \exists \mathsf{w}\mathsf{y}_1(\rho_{\mathsf{x}_0}\mathsf{z}_0\mathsf{w}\mathsf{y}_1\subseteq \sigma_{\mathsf{x}}\tau_{\mathsf{x}}\mathsf{y})\tag{\incwe, where $|\mathsf{w}|=|\tau_\mathsf{x}|$}\\
\vdash& \exists \mathsf{x}_1\mathsf{w}\mathsf{y}_1(\mathsf{p}(\mathsf{x}_1)=\rho_{\mathsf{x}_0}\mathsf{z}_0\mathsf{w}\,\wedge\,\rho_{\mathsf{x}_0}\mathsf{z}_0\mathsf{w}\mathsf{y}_1\subseteq \sigma_{\mathsf{x}}\tau_{\mathsf{x}}\mathsf{y})\tag{\eqi, \existsi, $|\rho_{\mathsf{x}_0}\mathsf{z}_0\mathsf{w}|=|\mathsf{x}|$}\\
\vdash& \exists \mathsf{x}_1\mathsf{w}\mathsf{y}_1(\sigma_{\mathsf{x}_1}\tau_{\mathsf{x}_1}=\rho_{\mathsf{x}_0}\mathsf{z}_0\mathsf{w}\,\wedge\,\mathsf{p}(\mathsf{x}_1)\mathsf{y}_1\subseteq \mathsf{p}(\mathsf{x})\mathsf{y})\tag{\eqsub}\\
\vdash& \exists \mathsf{x}_1\mathsf{y}_1(\sigma_{\mathsf{x}_1}=\rho_{\mathsf{x}_0}\mathsf{z}_0\,\wedge\,\mathsf{p}(\mathsf{x}_1)\mathsf{y}_1\subseteq \mathsf{p}(\mathsf{x})\mathsf{y})\tag{since $|\sigma_{\mathsf{x}_1}|=|\sigma_{\mathsf{x}}|=|\rho_{\mathsf{x}_0}\mathsf{z}_0|$}\\
\vdash&\exists \mathsf{x}_1\mathsf{y}_1(\sigma_{\mathsf{x}_1}=\rho_{\mathsf{x}_0}\mathsf{z}_0\,\wedge\,\mathsf{x}_1\mathsf{y}_1\subseteq \mathsf{x}\mathsf{y}).\tag{\incexc}
\end{align*}\vspace{-6pt}
\end{proof}

We say that an occurrence of a subformula $\theta$ in $\phi(\theta)$ is {\em not in the scope of a disjunction or negation} if (1) $\phi=\theta$; or (2) $\phi=\psi(\theta)\wedge\chi$ or $\chi\wedge\psi(\theta)$, and $\theta$ is not in the scope of a disjunction or negation in $\psi(\theta)$; or (3) $\phi=Qx\psi(\theta)$ ($Q\in\{\forall,\exists\}$) and $\theta$ is not in the scope of a disjunction or negation in $\psi(\theta)$. For example, in the formula $(\phi(\theta)\vee\psi)\wedge\exists x\theta$, the leftmost occurrence of $\theta$ is in the scope of a disjunction, while the rightmost occurrence of $\theta$ is not. %We write $\phi[\chi/\theta]$ for the formula obtained from $\phi(\theta)$ by replacing the single occurrence of $\theta$ by $\chi$.

\begin{lemma}\label{mon_lem_1}
If the occurrence of the subformula $\theta$ in $\phi(\theta)$  is not in the scope of a disjunction or negation, then $\phi(\theta),\psi\vdash\phi[\theta\wedge\psi/\theta]$.
\end{lemma}
\begin{proof}
A routine inductive proof. Apply \uqs, \existse, \existsi in the quantifier cases. 
%that makes use of Proposition \ref{der_rules}\ref{der_rules_c1}\ref{der_rules_c2}.
\end{proof}

\begin{proposition}\label{inc_atm_prop_c7_prop}
Suppose that $\phi(\mathsf{x}\subseteq\mathsf{y})$ is a formula in which the occurrence of $\mathsf{x}\subseteq\mathsf{y}$ is not in the scope of a disjunction or negation, and  the variables from $\mathsf{y}$ are  free in $\phi$. If $\mathsf{z}$ does not have any common variable with $\mathsf{x}\mathsf{y}$, and $\mathsf{w}$ contains some variables occurring in $\phi$ (either free or bound), then $\forall \mathsf{z}\phi(\mathsf{x}\subseteq\mathsf{y})\vdash \forall \mathsf{z}\phi[\mathsf{x}\mathsf{w}\subseteq\mathsf{y}\mathsf{z}/\mathsf{x}\subseteq\mathsf{y}]$. %\todo{CHECK!}
\end{proposition}
\begin{proof}
We prove the proposition by induction on $\phi$. If $\phi=\mathsf{x}\subseteq\mathsf{y}$, since no variable from $\mathsf{z}$ occurs in $\mathsf{x}\subseteq\mathsf{y}$, we derive by \uqez  that $\forall \mathsf{z}(\mathsf{x}\subseteq \mathsf{y})\vdash\mathsf{x}\subseteq \mathsf{y}$. Next, we obtain by \incwu that $\mathsf{x}\subseteq \mathsf{y}\vdash\forall \mathsf{z}(\mathsf{x}\mathsf{w}\subseteq \mathsf{y}\mathsf{z})$. Thus, $\forall \mathsf{z}(\mathsf{x}\subseteq \mathsf{y})\vdash\forall \mathsf{z}(\mathsf{x}\mathsf{w}\subseteq \mathsf{y}\mathsf{z})$ follows.

%then we derive by applying \uqez and \incwu that 
%\(\forall \mathsf{z}(\mathsf{x}\subseteq \mathsf{y})\vdash\mathsf{x}\subseteq \mathsf{y}\vdash\forall \mathsf{z}(\mathsf{x}\mathsf{w}\subseteq \mathsf{y}\mathsf{z}).\)
%\begin{align*}
%\forall \mathsf{z}(\mathsf{x}\subseteq \mathsf{y})\,&\vdash\mathsf{x}\subseteq \mathsf{y}\tag{\uqez}\\
%&\vdash\forall \mathsf{z}(\mathsf{x}\mathsf{w}\subseteq \mathsf{y}\mathsf{z}).\tag{\incwu}
%\end{align*}

If $\phi=\psi(\mathsf{x}\subseteq\mathsf{y})\wedge\chi$, then we have that
\begin{align*}
\forall \mathsf{z}(\psi(\mathsf{x}\subseteq\mathsf{y})\wedge\chi)\vdash&\forall \mathsf{z}\psi(\mathsf{x}\subseteq\mathsf{y})\wedge\forall \mathsf{z}\chi\tag{Proposition \ref{der_rules2}\ref{der_rules2_c3}}\\
\vdash&\forall \mathsf{z}\psi[\mathsf{x}\mathsf{w}\subseteq\mathsf{y}\mathsf{z}/\mathsf{x}\subseteq\mathsf{y}]\wedge\forall \mathsf{z}\chi\tag{induction hypothesis}\\
%\vdash&\forall \mathsf{z}\psi[\mathsf{x}\mathsf{w}\subseteq\mathsf{y}\mathsf{z}/\mathsf{x}\subseteq\mathsf{y}]\wedge\forall \mathsf{v}\chi(\mathsf{v}/\mathsf{z})\tag{Prop. \ref{der_rules2}\ref{der_rules2_c5}, where $\mathsf{v}$ are fresh}\\
%\vdash&\forall \mathsf{z}\forall \mathsf{v}(\psi[\mathsf{x}\mathsf{w}\subseteq\mathsf{y}\mathsf{z}/\mathsf{x}\subseteq\mathsf{y}]\wedge\chi(\mathsf{v}/\mathsf{z}))\tag{Prop. \ref{der_rules}\ref{der_rules_c2}}\\
\vdash&\forall \mathsf{z}(\psi[\mathsf{x}\mathsf{w}\subseteq\mathsf{y}\mathsf{z}/\mathsf{x}\subseteq\mathsf{y}]\wedge\chi).\tag{Proposition \ref{der_rules2}\ref{der_rules2_c3}}
\end{align*}
The case  $\phi=\chi\wedge\psi(\mathsf{x}\subseteq\mathsf{y})$ is symmetric.

If $\phi=\forall v\psi(\mathsf{x}\subseteq\mathsf{y})$, then we have that
\begin{align*}
\forall \mathsf{z}\forall v\psi(\mathsf{x}\subseteq\mathsf{y})\,&\vdash\forall v\forall \mathsf{z}\psi(\mathsf{x}\subseteq\mathsf{y})\tag{\uqexc}\\
&\vdash\forall v\forall \mathsf{z}\psi[\mathsf{x}\mathsf{w}\subseteq\mathsf{y}\mathsf{z}/\mathsf{x}\subseteq\mathsf{y}]\tag{induction hypothesis}\\
&\vdash\forall \mathsf{z}\forall v\psi[\mathsf{x}\mathsf{w}\subseteq\mathsf{y}\mathsf{z}/\mathsf{x}\subseteq\mathsf{y}].\tag{\uqexc}
\end{align*}

If $\phi=\exists v\psi(\mathsf{x}\subseteq\mathsf{y})$, where $\mathsf{Fv}(\exists v\psi)=\mathsf{u}$ (note that all variables from $\mathsf{y}$ are among $\mathsf{u}$), then we have that
\begin{align*}
\forall \mathsf{z}\exists v\psi(\mathsf{x}\subseteq\mathsf{y})
\vdash&\exists \mathsf{z}\exists v\forall \mathsf{z}_0(\mathsf{u}\mathsf{z}_0\subseteq\mathsf{u}\mathsf{z}\wedge\psi(\mathsf{x}\subseteq\mathsf{y}))\tag{Proposition \ref{inc_atm_prop_c4}, where $\mathsf{z}_0$ are fresh}\\
\vdash&\exists \mathsf{z}\exists v(\forall \mathsf{z}_0(\mathsf{u}\mathsf{z}_0\subseteq\mathsf{u}\mathsf{z})\wedge\forall \mathsf{z}_0\psi(\mathsf{x}\subseteq\mathsf{y}))\tag{Proposition \ref{der_rules2}\ref{der_rules2_c3}}\\
%\vdash&\exists \mathsf{z}\exists v(\forall \mathsf{z}_0(\mathsf{u}\mathsf{z}_0\subseteq\mathsf{u}\mathsf{z})\wedge\forall \mathsf{z}_1\psi(\mathsf{x}\subseteq\mathsf{y}))\tag{Prop. \ref{der_rules2}\ref{der_rules2_c5}, where $\mathsf{z}_1$ are fresh, and note $\mathsf{z}_0\notin \mathsf{Fv}(\psi(\mathsf{x}\subseteq\mathsf{y}))$}\\
\vdash&\exists \mathsf{z}\exists v(\forall \mathsf{z}_0(\mathsf{u}\mathsf{z}_0\subseteq\mathsf{u}\mathsf{z})\wedge\forall \mathsf{z}_0\psi[\mathsf{x}\mathsf{w}\subseteq\mathsf{y}\mathsf{z}_0/\mathsf{x}\subseteq\mathsf{y}])\tag{induction hypothesis}\\
%\vdash&\exists \mathsf{z}\exists v\forall \mathsf{z}_0\forall \mathsf{z}_1(\mathsf{u}\mathsf{z}_0\subseteq\mathsf{u}\mathsf{z}\wedge\psi[\mathsf{x}\mathsf{w}\subseteq\mathsf{y}\mathsf{z}_1/\mathsf{x}\subseteq\mathsf{y}])\tag{Prop. \ref{der_rules}\ref{der_rules_c2}}\\
\vdash&\exists \mathsf{z}\exists v\forall \mathsf{z}_0(\mathsf{u}\mathsf{z}_0\subseteq\mathsf{u}\mathsf{z}\wedge\psi[\mathsf{x}\mathsf{w}\subseteq\mathsf{y}\mathsf{z}_0/\mathsf{x}\subseteq\mathsf{y}])\tag{Proposition \ref{der_rules2}\ref{der_rules2_c3}}\\
\vdash&\exists \mathsf{z}\exists v\forall \mathsf{z}_0(\mathsf{u}\mathsf{z}_0\subseteq\mathsf{u}\mathsf{z}\wedge \mathsf{y}\mathsf{z}_0\subseteq\mathsf{y}\mathsf{z}\wedge\psi[\mathsf{x}\mathsf{w}\subseteq\mathsf{y}\mathsf{z}_0/\mathsf{x}\subseteq\mathsf{y}])\tag{\incctr}\\
\vdash&\exists \mathsf{z}\exists v\forall \mathsf{z}_0(\mathsf{u}\mathsf{z}_0\subseteq\mathsf{u}\mathsf{z}\wedge\psi[\mathsf{x}\mathsf{w}\subseteq\mathsf{y}\mathsf{z}_0\wedge  \mathsf{y}\mathsf{z}_0\subseteq\mathsf{y}\mathsf{z}/\mathsf{x}\subseteq\mathsf{y}])\tag{Proposition \ref{mon_lem_1}, $\because$ $\mathsf{x}\mathsf{w}\subseteq\mathsf{y}\mathsf{z}_0$ is not in the scope of a disjunction or negation}\\
\vdash&\exists \mathsf{z}\exists v\forall \mathsf{z}_0(\mathsf{u}\mathsf{z}_0\subseteq\mathsf{u}\mathsf{z}\wedge\psi[\mathsf{x}\mathsf{w}\subseteq\mathsf{y}\mathsf{z}/\mathsf{x}\subseteq\mathsf{y}])\tag{\inctr, Proposition \ref{mon_lem_2}, $\because$ $\mathsf{x}\subseteq\mathsf{y}$ cannot occur in the scope of a negation}\\
\vdash&\forall \mathsf{z}\exists v\psi[\mathsf{x}\mathsf{w}\subseteq\mathsf{y}\mathsf{z}/\mathsf{x}\subseteq\mathsf{y}].\tag{Proposition \ref{inc_atm_prop_c4}, $\because$ variables in $\mathsf{w}$ are either bound in $\psi[\mathsf{x}\mathsf{w}\subseteq\mathsf{y}\mathsf{z}/\mathsf{x}\subseteq\mathsf{y}]$, or among $\mathsf{u}$}
\end{align*}
\end{proof}

Now we are ready to give the proof of Lemma \ref{approx_der}.
\begin{proof}[Proof of Lemma \ref{approx_der}]
By Lemma \ref{nf_der}, we may assume that $\phi$ is in normal form (\ref{inc_nf}). %, i.e.,
%\[\phi=\exists\mathsf{w}\exists \mathsf{x}\forall y\Big(\bigwedge_{i\in I}\rho^i_{\mathsf{w}}\subseteq \sigma^i_{\mathsf{w}}\,\wedge \bigwedge_{j\in J}\pi^j_{\mathsf{x}}y\subseteq \pi^j_{\mathsf{x}}x_j\wedge\alpha(\mathsf{w},\mathsf{x})\Big).\]
We prove the lemma by proving a stronger claim that $\phi\vdash\Phi_n'$ holds for every $n<\omega$, where  %$\Phi_n':=$% $\Phi_0':=\exists \mathsf{x}_0\forall y_1(\alpha(\mathsf{x}_0)\wedge \mu_0)$ and for $n\geq1$,
\[\Phi_n':= \exists \mathsf{w}_0\exists \mathsf{x}_0 \forall y_0\Big(\alpha_0\wedge\mu_0\wedge  \exists \mathsf{w}^{1}\mathsf{x}^1 \forall y_1\big(\lambda_1\wedge\mu_1\wedge\dots\wedge\exists \mathsf{w}^{n}\mathsf{x}^n \forall y_n(\lambda_n\wedge\mu_n)\underbrace{)\dots\big)\big)\Big)}_{n+1},
\]
where each $\lambda_n=\alpha_n\wedge\gamma_n\wedge\delta_n$,
\[\mu_0=\bigwedge_{i\in I}\rho^i_{\mathsf{w_0}}\subseteq \sigma^i_{\mathsf{w_0}}\,\wedge \bigwedge_{j\in J}\pi^j_{\mathsf{x}_0}y_0\subseteq \tau^j_{\mathsf{x}_0}~\text{ and }~\mu_n=\bigwedge_{\xi\in E_n\cup U_n}\mathsf{w}_\xi\mathsf{x}_\xi\subseteq\mathsf{w}_0\mathsf{x}_0.\]
%\[\eta_0=\bigwedge_{i=1}^k\rho^i_{\mathsf{x}_0y_1}\subseteq \sigma^i_{\mathsf{x}_0}~~\text{ and }~~\eta_n=\bigwedge_{x_\xi y_\eta\in A_{n}\setminus A_{n-1}}\bigwedge_{i=1}^k\rho^i_{\mathsf{x}^{n,i}_{\xi\eta} y_\eta}\subseteq \sigma^i_{\mathsf{x}^{n,i}_{\xi\eta}}\text{ for }n\geq1.\]

If $n=0$, then  $\Phi_0':=\exists\mathsf{w}_0\mathsf{x}_0\forall y_0(\alpha(\mathsf{w}_0,\mathsf{x}_0)\wedge \mu_0)$, and $\phi\vdash\Phi_0'$ can be derived  by simply renaming the variables. Now, assuming $\phi\vdash\Phi'_n$, we show that $\phi\vdash\Phi'_{n+1}$ by deriving $\Phi'_n\vdash\Phi'_{n+1}$. 
%Noting that 
%\[A_{n+1}=A_n\cup\{\mathsf{x}^{n,i}_{\xi\eta}\mid \xi\eta\in A_n\}~\text{ and }~\eta_{n+1}=\eta_n\wedge \bigwedge_{\xi\eta\in A_{n+1}}\bigwedge_{i=1}^k\rho^i_{\mathsf{x}^{n,i}_{\xi\eta} y_\eta}\subseteq \sigma^i_{\mathsf{x}^{n,i}_{\xi\eta}},\]

First, for each $\xi\in E_n$ and each $i\in I$, by Proposition \ref{inc_atm_prop_c0} we derive that
\[\rho^i_{\mathsf{w_0}}\subseteq \sigma^i_{\mathsf{w_0}},\mathsf{w}_{\xi}\mathsf{x}_\xi\subseteq \mathsf{w}_0\mathsf{x}_0\vdash\exists \mathsf{w}\mathsf{x}(\mathsf{w}\mathsf{x}\subseteq \mathsf{w}_0\mathsf{x}_0\wedge \rho^i_{\mathsf{w}_\xi}=\sigma^i_{\mathsf{w}}),\]
which, by \incc,  yields
\begin{equation}\label{approx_der_eq1}
\alpha(\mathsf{w}_0,\mathsf{x}_0),\rho^i_{\mathsf{w_0}}\subseteq \sigma^i_{\mathsf{w_0}},\mathsf{w}_{\xi}\mathsf{x}_\xi\subseteq \mathsf{w}_0\mathsf{x}_0\vdash \exists \mathsf{w}\mathsf{x}(\mathsf{w}\mathsf{x}\subseteq \mathsf{w}_0\mathsf{x}_0\wedge \rho^i_{\mathsf{w}_\xi}=\sigma^i_{\mathsf{w}}\wedge\alpha(\mathsf{w},\mathsf{x})).
\end{equation}
%\begin{equation}\label{approx_der_eq1}
%\begin{split}
%&\alpha(\mathsf{w}_0,\mathsf{x}_0),\rho^i_{\mathsf{w_0}}\subseteq \sigma^i_{\mathsf{w_0}},\mathsf{w}_{\xi}\mathsf{x}_\xi\subseteq \mathsf{w}_0\mathsf{x}_0\\
%\vdash&\exists \mathsf{w}\mathsf{x}(\mathsf{w}\mathsf{x}\subseteq \mathsf{w}_0\mathsf{x}_0\wedge \rho^i_{\mathsf{w}_\xi}=\sigma^i_{\mathsf{w}})\quad\quad\quad\quad\quad\quad\quad\quad\text{(Prop. \ref{inc_atm_prop}\ref{inc_atm_prop_c0})}\\
%\vdash&\exists \mathsf{w}\mathsf{x}(\mathsf{w}\mathsf{x}\subseteq \mathsf{w}_0\mathsf{x}_0\wedge \rho^i_{\mathsf{w}_\xi}=\sigma^i_{\mathsf{w}}\wedge\alpha(\mathsf{w},\mathsf{x})).\quad\quad\quad\quad\quad\quad{(\incc)}
%\end{split}
%\end{equation}
Similarly, for each $\xi\eta\in A_{n+1}$ and  $ j\in J$, we derive also by Proposition \ref{inc_atm_prop_c0} and \incc that
\begin{equation}\label{approx_der_eq2}
%\alpha(\mathsf{w}_0,\mathsf{x}_0),\pi^j_{\mathsf{x}_0}y_0\subseteq \tau^j_{\mathsf{x}_0},\mathsf{w}_{\xi}\mathsf{x}_\xi y_\eta\subseteq \mathsf{w}_0\mathsf{x}_0y_0\vdash\exists \mathsf{w}\mathsf{x}(\mathsf{w}\mathsf{x}\subseteq \mathsf{w}_0\mathsf{x}_0\wedge \pi^j_{\mathsf{x}_\xi}y_\eta= \tau^j_{\mathsf{x}}\wedge\alpha(\mathsf{w},\mathsf{x})).
\alpha(\mathsf{w}_0,\mathsf{x}_0),\pi^j_{\mathsf{x}_0}y_0\subseteq \tau^j_{\mathsf{x}_0},\mathsf{w}_{\xi}\mathsf{x}_\xi y_\eta\subseteq \mathsf{w}_0\mathsf{x}_0y_0\vdash\!\exists \mathsf{w}\mathsf{x}(\mathsf{w}\mathsf{x}\subseteq \mathsf{w}_0\mathsf{x}_0\wedge \pi^j_{\mathsf{x}_\xi}y_\eta= \tau^j_{\mathsf{x}}\wedge\alpha(\mathsf{w},\mathsf{x})).
%&\alpha(\mathsf{w}_0,\mathsf{x}_0),\pi^j_{\mathsf{x}_0}y_0\subseteq \tau^j_{\mathsf{x}_0},\mathsf{w}_{\xi}\mathsf{x}_\xi y_\eta\subseteq \mathsf{w}_0\mathsf{x}_0y_0\\
%\vdash&\exists \mathsf{w}\mathsf{x}(\mathsf{w}\mathsf{x}\subseteq \mathsf{w}_0\mathsf{x}_0\wedge \pi^j_{\mathsf{x}_\xi}y_\eta= \tau^j_{\mathsf{x}})\quad\quad\quad\quad\quad\quad\quad\quad\text{(Prop. \ref{inc_atm_prop}\ref{inc_atm_prop_c0})}\\
%\vdash&\exists \mathsf{w}\mathsf{x}(\mathsf{w}\mathsf{x}\subseteq \mathsf{w}_0\mathsf{x}_0\wedge \pi^j_{\mathsf{x}_\xi}y_\eta= \tau^j_{\mathsf{x}}\wedge\alpha(\mathsf{w},\mathsf{x})).\quad\quad\quad\quad\quad\quad{(\incc)}
\end{equation}
Next, we derive that
\begin{align*}	
\Phi_n'
\vdash&\exists \mathsf{w}_0\exists \mathsf{x}_0 \forall y_0\Big(\alpha_0\wedge\mu_0\wedge  \exists \mathsf{w}^{1}\mathsf{x}^{1}  \forall y_1\big(\dots\wedge\exists \mathsf{w}^{n}\mathsf{x}^{n}  \forall y_n\big(\lambda_n\wedge\mu_n\\
&\quad\quad\wedge\alpha(\mathsf{w}_0,\mathsf{x}_0)\wedge\big(\bigwedge_{i\in I}\rho^i_{\mathsf{w_0}}\subseteq \sigma^i_{\mathsf{w_0}}\big)\wedge\big(\bigwedge_{\xi\in E_n\cup U_n}\mathsf{w}_{\xi}\mathsf{x}_\xi\subseteq \mathsf{w}_0\mathsf{x}_0\big)\\
&\quad\quad\quad\quad\wedge\big(\bigwedge_{j\in J}\pi^j_{\mathsf{x}_0}y_0\subseteq \tau^j_{\mathsf{x}_0}\big)\wedge\bigwedge_{\xi\eta\in A_{n+1}}\mathsf{w}_{\xi}\mathsf{x}_\xi\subseteq \mathsf{w}_0\mathsf{x}_0\big)\dots\big)\Big)\tag{Proposition \ref{der_rules}\ref{der_rules_c2}\ref{der_rules_c1}}\\	
\vdash&\exists \mathsf{w}_0\exists \mathsf{x}_0 \forall y_0\Big(\alpha_0\wedge\mu_0\wedge   \exists \mathsf{w}^{1}\mathsf{x}^{1}  \forall y_1\big(\dots\wedge\exists \mathsf{w}^{n}\mathsf{x}^{n}  \forall y_n\big(\lambda_n\wedge\mu_n\\
&\quad\quad\wedge\alpha(\mathsf{w}_0,\mathsf{x}_0)\wedge\big(\bigwedge_{i\in I}\rho^i_{\mathsf{w_0}}\subseteq \sigma^i_{\mathsf{w_0}}\big)\wedge\big(\bigwedge_{\xi\in E_n\cup U_n}\mathsf{w}_{\xi}\mathsf{x}_\xi\subseteq \mathsf{w}_0\mathsf{x}_0\big)\\
&\quad\quad\quad\quad\wedge\big(\bigwedge_{j\in J}\pi^j_{\mathsf{x}_0}y_0\subseteq \tau^j_{\mathsf{x}_0}\big)\wedge\bigwedge_{\xi\eta\in A_{n+1}}\mathsf{w}_{\xi}\mathsf{x}_\xi y_\eta\subseteq \mathsf{w}_0\mathsf{x}_0y_0\big)\dots\Big)\\
&\tag{Proposition \ref{inc_atm_prop_c7_prop} applied to the subformula $\forall y_0(\alpha_0\wedge\dots)$ and each $\mathsf{w}_{\xi}\mathsf{x}_\xi\subseteq \mathsf{w}_0\mathsf{x}_0$}\\
\vdash&\exists \mathsf{w}_0\exists \mathsf{x}_0 \forall y_0\Big(\alpha_0\wedge\mu_0\wedge  \dots\wedge\exists \mathsf{w}^{n}\mathsf{x}^{n}  \forall y_n\big(\lambda_n\wedge\mu_n\wedge\\
&\quad\bigwedge_{\xi\in E_n\cup U_n, i\in I}\exists \mathsf{w}_{\xi,i}\mathsf{x}_{\xi,i}(\alpha(\mathsf{w}_{\xi,i},\mathsf{x}_{\xi,i})\wedge \rho^i_{\mathsf{w}_{\xi}}=\sigma^i_{\mathsf{w}_{\xi,i}}\wedge\mathsf{w}_{\xi,i}\mathsf{x}_{\xi,i}\subseteq \mathsf{w}_0\mathsf{x}_0)\\
&\quad\quad\quad\wedge\bigwedge_{\xi\eta\in A_{n+1},j\in J}\exists \mathsf{w}_{\xi\eta,j}\mathsf{x}_{\xi\eta,j}(\alpha(\mathsf{w}_{\xi\eta,j},\mathsf{x}_{\xi\eta,j})\wedge\pi^j_{\mathsf{x}_\xi}y_\eta= \tau^j_{\mathsf{x}_{\xi\eta,j}}\wedge\mathsf{w}_{\xi\eta,j}\mathsf{x}_{\xi\eta,j}\subseteq\mathsf{w}_0\mathsf{x}_0)\big)\dots\Big)\tag{by (\ref{approx_der_eq1}) \& (\ref{approx_der_eq2})}\\
\vdash&\exists \mathsf{w}_0\exists \mathsf{x}_0 \forall y_0\big(\alpha_0\wedge\mu_0\wedge  \dots\wedge\exists \mathsf{w}^n\mathsf{x}^{n} \forall y_n\big(\lambda_n\wedge\mu_n\\%\exists\mathsf{v}^{n+1}(\lambda_{n+1}^\exists\wedge\mu_{n+1}^\exists)\wedge\exists\mathsf{u}^{n+1}(\lambda_{n+1}^\forall\wedge\mu_{n+1}^\forall)\big)\dots\big)\\
&\quad\quad\quad\wedge\exists \langle \mathsf{w}_\xi\mathsf{x}_\xi\mid \xi\in E_{n+1}\rangle \big(\gamma_{n+1}\wedge \bigwedge_{\xi\in E_{n+1}}(\alpha(\mathsf{w}_\xi,\mathsf{x}_\xi)\wedge \mathsf{w}_\xi\mathsf{x}_\xi\subseteq \mathsf{w}_0\mathsf{x}_0)\big)\\
&\quad\quad\quad\quad\wedge\exists \langle \mathsf{w}_\xi\mathsf{x}_\xi\mid \xi\in U_{n+1}\rangle \big(\delta_{n+1}\wedge \bigwedge_{\xi\in U_{n+1}}(\alpha(\mathsf{w}_\xi,\mathsf{x}_\xi)\wedge \mathsf{w}_\xi\mathsf{x}_\xi\subseteq \mathsf{w}_0\mathsf{x}_0)\big)\big)\dots\Big)\\
\vdash&\exists \mathsf{w}_0\exists \mathsf{x}_0 \forall y_0\big(\alpha_0\wedge\mu_0\wedge  \dots\wedge\exists \mathsf{w}^n\mathsf{x}^{n} \forall y_n\big(\lambda_n\wedge\mu_n\wedge\exists\mathsf{w}^{n+1}\mathsf{x}^{n+1}(\lambda_{n+1}\wedge\mu_{n+1})\big)\!\dots\!\big)\\
\vdash&\exists \mathsf{w}_0\exists \mathsf{x}_0 \forall y_0\big(\alpha_0\wedge\mu_0\wedge  \dots\wedge\exists\mathsf{w}^{n+1}\mathsf{x}^{n+1}\forall y_{n+1}(\lambda_{n+1}\wedge\mu_{n+1})\dots\big)\text{, i.e., }\Phi_{n+1}'.\tag{\uqi, as $y_{n+1}$ does not occur in $\lambda_{n+1}\wedge\mu_{n+1}$}
\end{align*}
This finishes the proof.
\end{proof}

Finally, we are in a position to prove the completeness theorem of our system.

%\begin{theorem}[Completeness]
%Let $\Gamma$ be a set of \Inc-formulas, and $\alpha$ a first-order formula. Then 
%\[\Gamma\models\alpha\iff\Gamma\vdash_\Inc\alpha.\]
%\end{theorem}
%\begin{proof}
%Suppose $\Gamma\models\alpha(\mathsf{x})$. Since \Inc is compact, without loss of generality we may assume that $\Gamma$ is a finite set. Then $\Gamma,\cn\alpha\models\bot$, which implies $\exists \mathsf{z}(\bigwedge\Gamma\wedge \cn\alpha)\models\bot$, i.e., $\exists \mathsf{z}(\bigwedge\Gamma\wedge \exists \mathsf{y}(\mathsf{y}\subseteq\mathsf{x}\wedge\neg\alpha(\mathsf{y})))\models\bot$, where $\mathsf{z}$ lists all free variables in $\Gamma$ and $\cn\alpha$. Since $\exists \mathsf{z}(\bigwedge\Gamma\wedge \exists \mathsf{y}(\mathsf{y}\subseteq\mathsf{x}\wedge\neg\alpha(\mathsf{y})))$ is an \Inc-sentence and $\bot$ is first-order, by Theorem \ref{cmp_sent} we have $\exists \mathsf{z}(\bigwedge\Gamma\wedge \exists \mathsf{y}(\mathsf{y}\subseteq\mathsf{x}\wedge\neg\alpha(\mathsf{y})))\vdash_\Inc\bot$, from which we derive by \existsi that $\Gamma,\exists \mathsf{y}(\mathsf{y}\subseteq\mathsf{x}\wedge\neg\alpha(\mathsf{y}))\vdash_\Inc\bot$. Finally we conclude $\Gamma\vdash_\Inc\alpha$ by applying Proposition \ref{neg_trans}.
%\end{proof}

\begin{theorem}[Completeness]\label{com_thm}
Let $\Gamma$ be a set of \Inc-formulas, and $\alpha$ a first-order formula. Then 
\[\Gamma\models\alpha\iff\Gamma\vdash_\Inc\alpha.\]
\end{theorem}
\begin{proof}
The direction ``$\Longleftarrow$'' follows from the soundness theorem. For the direction ``$\Longrightarrow$'', since \Inc is compact,  we may without loss of generality  assume that $\Gamma$ is finite. Suppose now $\Gamma\models\alpha$ and $\Gamma\nvdash_\Inc\alpha$.  Claim that $\exists \mathsf{z}(\bigwedge\Gamma\wedge\cn\alpha)\nvdash_{\Inc}\bot$, where %$\cn\alpha(\mathsf{x})=\exists \mathsf{y}(\mathsf{y}\subseteq\mathsf{x}\wedge\neg\alpha(\mathsf{y}))$, and
 $\mathsf{z}$ lists all free variables in $\Gamma$ and $\cn\alpha$. Indeed, if $\exists \mathsf{z}(\bigwedge\Gamma\wedge\cn\alpha)\vdash_\Inc\bot$, then we derive $\Gamma,\cn\alpha\vdash_\Inc\bot$ by \existsi, and further $\Gamma\vdash_\Inc\alpha$ by Lemma \ref{neg_trans}; a contradiction.

Now, let $\Delta=\{\Phi_n\mid \phi=\exists \mathsf{z}(\bigwedge\Gamma\wedge\cn\alpha)\text{ and }n<\omega\}$. By Lemma \ref{approx_der}, we must have that $\Delta\nvdash_{\Inc}\bot$. It follows that $\Delta\nvdash_{\FO}\bot$, since $\Delta\cup\{\bot\}$ is a set of first-order formulas, and the deduction system of \Inc has the same rules as that of first-order logic when restricted to first-order formulas. By the completeness theorem of first-order logic, we know that the set $\Delta$ of approximations of $\phi$ has a model $M$. By \cite{BarwiseSchlipf76}, every infinite model is elementary equivalent to a recursively saturated countable model. Thus,  we may assume that $M$ is a  recursively saturated  countable or finite model. By Theorem \ref{satu_md_approx_equi}, $M$ is also a model of $\exists \mathsf{z}(\bigwedge\Gamma\wedge\cn\alpha)$,  thereby $M\models_{\{\emptyset\}(\mathsf{F}/\mathsf{z})}\Gamma$ and $M\not\models_{\{\emptyset\}(\mathsf{F}/\mathsf{z})}\alpha$ for some suitable sequence $\mathsf{F}$ of functions for $\exists \mathsf{z}$. Hence $\Gamma\not\models\alpha$.
\end{proof}

\section{Applications}

In this final section of the paper, we illustrate the power of our system of \Inc by discussing some applications.

Recall from Proposition \ref{well_founded} that the sentence $\exists x\exists y(y\subseteq x\wedge y<x)$ defines the fact that $<$ is not well-founded. By the completeness theorem (Theorem \ref{com_thm}) we proved in the previous section, all first-order consequences of the non-well-foundedness of $<$ are derivable in our system. For instance, the property that there is a $<$-chain of length $n$ for any natural number $n$, and the property that this $<$-chain of length $n$ descends from the greatest element (if exists).
We now give explicit derivations of these properties in the example below.

\begin{example}
Write $ x_1<x_2<\dots<x_n$ for $\bigwedge_{i=1}^{n-1} x_i<x_{i+1}$.
For any  $n\in\mathbb{N}$,
\begin{enumerate}[label=(\roman*)]
\item $\exists x\exists y(y\subseteq x\wedge y<x)\vdash \exists x_1\dots\exists x_n(x_1<x_2<\dots<x_n)$,
\item $\exists x\exists y(y\subseteq x\wedge y<x),\forall y(y<x_0\vee y=x_0)\vdash \exists x_1\dots\exists x_n(x_1<\dots<x_n<x_0)$.
%\item $\exists x\exists y(y\subseteq x\wedge y<x),\forall x\forall y\forall z((x<y\wedge y<z)\to x<z)\vdash \exists x_0(\exists y(y<x_0)\wedge \forall y(y<x_0\to \exists z(z<y)))$
%\item $\exists x\exists y(y\subseteq x\wedge y<x),\forall x\forall y\forall z((x<y\wedge y<z)\to x<z),\forall y(y<x_0\vee y=x_0)\vdash \exists y(y<x_0)\wedge \forall y(y<x_0\to \exists z(z<y))$
%\item $\exists x\exists y(y\subseteq x\wedge xEy), \forall x_0\dots \forall x_n\bigvee_{0\leq i<j\leq n}x_i=x_j$
\end{enumerate}
\end{example}
\begin{proof}
(i) We only give an example of the proof for $n=3$.
\begin{align*}
\exists x\exists y(y\subseteq x\wedge y<x)
\vdash &\exists x\exists y\exists z(yz\subseteq xy\wedge y<x)\tag{\incwe}\\
\vdash &\exists x\exists y\exists z(yz\subseteq xy\wedge z<y\wedge y<x)\tag{\incc}\\
%\vdash &\exists x\exists y\exists z( z<y\wedge y<x)\tag{\conje}\\
\vdash &\exists x_1\exists x_2\exists x_3( x_1<x_2\wedge x_2<x_3)\tag{\conje and renaming bound variables}
\end{align*}

(ii) In view of item (i), it suffices to show $\exists x_1\dots\exists x_n(x_1< \dots<x_n),\forall y(y<x_0\vee y=x_0)\vdash \exists x_1\dots\exists x_n(x_1<\dots<x_n<x_0)$. But this is derivable in the system of first-order logic, and the same proof can also be performed in the system of \Inc.
\end{proof}

In Proposition \ref{inc_atm_prop} in section 5 we have derived some interesting clauses in our system of \Inc.
It is interesting to note that the formulas on the right side of the turnstile ($\vdash$) in items \ref{inc_atm_prop_c1}\ref{inc_atm_rep} of the proposition are not first-order formulas. While our completeness theorem (Theorem \ref{com_thm}) does not apply to these cases, these clauses are indeed derivable. We now give some more examples in which our system can be successfully applied to derive non-first-order consequences in \Inc.

%It is interesting to note that the completeness theorem for \Inc (Theorem \ref{com_thm}) we proved in the previous section does not actually apply to the clauses in the above example (except one), as all formulas  on the right side of $\vdash$ in the above example expect one are not first-order formulas.

Consider the so-called {\em anonymity atoms}, introduced  in \cite{Pietro_thesis} and studied recently by V\"{a}\"{a}n\"{a}nen \cite{Vaananen_anonymity19} motivated by concerns in data safety.  These atoms are strings of the form  $x_1\dots x_n\Upsilon y_1\dots y_m$ with the team semantics:
\begin{itemize}
\item $M\models_X\mathsf{x}\Upsilon\mathsf{y}$ \ \  iff \ \  for all $s\in X$, there exists $s'\in X$ such that $s(\mathsf{x})=s'(\mathsf{x})$ and $s(\mathsf{y})\neq s'(\mathsf{y})$.
\end{itemize}
Note that the anonymity atoms corresponds exactly to {\em afunctional dependencies} studied in database theory (see e.g., \cite{BraParedaens83,DeBraParedaens82}). It was proved in \cite{Pietro_thesis} that first-order logic extended with anonymity atoms is expressively equivalent to inclusion logic, and in particular, 
\[\mathsf{x}\Upsilon\mathsf{y}\equiv \exists \mathsf{v}(\mathsf{xv}\subseteq\mathsf{xy}\wedge  \mathsf{v}\neq\mathsf{y}),\] 
where $\mathsf{v}\neq\mathsf{y}$ is short for $\bigvee_i v_i\neq y_i$. We will then use $\mathsf{x}\Upsilon\mathsf{y}$ as a shorthand for the above equivalent \Inc-formula. Write $\Upsilon\mathsf{x}$ for $\langle\,\rangle\Upsilon\mathsf{x}$, and  stipulate  $\mathsf{x}\Upsilon=\mathsf{x}\Upsilon\langle\,\rangle:=\bot$.
 %Stipulate that $\mathsf{x}\Upsilon\langle\,\rangle:=\bot$ and write $\Upsilon\mathsf{x}$ for $\langle\,\rangle\Upsilon\mathsf{x}$.
  The implication problem of anonymity atoms is shown in \cite{Vaananen_anonymity19} to be completely axiomatized by the rules listed in the next example (read the clauses in the example as rules).
%$\anm(\mathsf{x})\equiv \exists \mathsf{v}(\mathsf{v}\subseteq \mathsf{x}\wedge \mathsf{v}\neq \mathsf{x})$
%It was proved in \cite{} that the clauses in the next example (read as rules) completely axiomatized the implication problem of anonymity atoms. 
We now illustrate that in our system of \Inc  all these rules are derivable.
\begin{example}
\begin{enumerate}[label=(\roman*)]
\item $\mathsf{xyz}\Upsilon\mathsf{uvw}\vdash \mathsf{yxz}\Upsilon\mathsf{uvw}\wedge \mathsf{xyz}\Upsilon\mathsf{vuw}$ (permutation).
\item $\mathsf{xy}\Upsilon\mathsf{z}\vdash\mathsf{x}\Upsilon\mathsf{zu}$ (monotonicity).
\item $\mathsf{xy}\Upsilon\mathsf{zy}\vdash\mathsf{xy}\Upsilon\mathsf{z}$ (weakening).
\item $\mathsf{x}\Upsilon\vdash\bot$.
\end{enumerate}
\end{example}
\begin{proof}
Item (i) follows easily from \incexc, and item (iv) is trivial. We only prove the other two items. For item (ii), note that $\mathsf{xy}\Upsilon\mathsf{z}:=\exists \mathsf{v}(\mathsf{xyv}\subseteq \mathsf{xyz}\wedge \mathsf{v}\neq \mathsf{z})$, and we have that \vspace{-4pt}
\begin{align*}
\exists \mathsf{v}(\mathsf{xyv}\subseteq \mathsf{xyz}\wedge \mathsf{v}\neq \mathsf{z})
\vdash &\exists \mathsf{v}(\mathsf{xv}\subseteq \mathsf{xz}\wedge \mathsf{v}\neq \mathsf{z})\tag{\incctr}\\
\vdash &\exists \mathsf{vw}(\mathsf{xvw}\subseteq \mathsf{xzu}\wedge \mathsf{v}\neq \mathsf{z})\tag{\incwe}\\
\vdash &\exists \mathsf{vw}(\mathsf{xvw}\subseteq \mathsf{xzu}\wedge \mathsf{vw}\neq \mathsf{zu})\tag{\tensori}\\
=: &\,\mathsf{x}\Upsilon\mathsf{zu}.
\end{align*}

For item (iii), note that $\mathsf{xy}\Upsilon\mathsf{zy}:=\exists \mathsf{uv}(\mathsf{xyuv}\subseteq \mathsf{xyzy}\wedge \mathsf{uv}\neq \mathsf{zy})$, and we have
\begin{align*}
\exists \mathsf{uv}(\mathsf{xyuv}\subseteq \mathsf{xyzy}\wedge \mathsf{uv}\neq \mathsf{zy})
\vdash &\exists \mathsf{uv}(\mathsf{xyuv}\subseteq \mathsf{xyzy}\wedge \mathsf{uv}\neq \mathsf{zy}\wedge \mathsf{y}= \mathsf{v})\tag{Proposition \ref{inc_atm_prop}\ref{inc_atm_prop_c2}}\\
\vdash &\exists \mathsf{uv}(\mathsf{xyuv}\subseteq \mathsf{xyzy}\wedge \mathsf{u}\neq \mathsf{z})\tag{\tensore}\\
\vdash &\exists \mathsf{u}(\mathsf{xyu}\subseteq \mathsf{xyz}\wedge \mathsf{u}\neq \mathsf{z})\tag{\incctr}\\
=: &\,\mathsf{xy}\Upsilon\mathsf{z}.
\end{align*}
\end{proof}

The above example indicates that the actual strength of our deduction system of \Inc goes beyond the completeness theorem (Theorem \ref{com_thm})  proved in this paper.  How far can we actually go then? There are obviously barriers, as inclusion logic cannot be effectively axiomatized after all. For instance, in the context of anonymity atoms, the author was not able to derive a simple (sound) implication ``$\Upsilon\mathsf{x}$ and $\mathsf{x}\subseteq \mathsf{y}$ imply $\Upsilon\mathsf{y}$" in the system of \Inc. An easy solution for generating derivations of simple facts like this one would be to extend the current system with new rules. 
%A relevant natural question in this respect is that 
But then how many new rules or which new rules should we add to the current system in order to derive ``sufficient" amount of sound consequences of \Inc?  One such candidate that is worth mentioning is the natural and handy rule $\phi\vee\neg\alpha,\alpha\vee\psi/\phi\vee\psi$ (for $\alpha$ being first-order) that is sound and does not seem to be derivable in our system. Finding other such rules is left for future research.

\begin{acknowledgement}
The author would like to thank Miika Hannula and Jouko V\"{a}\"{a}n\"{a}nen for interesting discussions related to this paper, and Davide Quadrellaro and an anonymous referee for pointing out some mistakes in an earlier version of the paper. 
\end{acknowledgement}

\bibliographystyle{acm}
%\bibliography{../fan}

\end{document}